\documentclass[11pt, reqno]{amsart}
\usepackage{amsmath, amsthm, amscd, amsfonts, amssymb, graphicx, color, mathtools, mathrsfs}
\usepackage[bookmarksnumbered, colorlinks, plainpages]{hyperref}
\usepackage{enumerate}
\usepackage{setspace}
\usepackage{multicol}
\usepackage[margin=1in]{geometry}
\usepackage{comment}
\usepackage{booktabs}
\usepackage{caption}
\usepackage{subcaption}
\usepackage{pgfplots}

\theoremstyle{definition}
\newtheorem{theorem}{Theorem}[section]
\newtheorem{lemma}[theorem]{Lemma}
\newtheorem{proposition}[theorem]{Proposition}

\newtheorem{definition}[theorem]{Definition}

\newtheorem{algorithm}[theorem]{Algorithm}
\newtheorem{remark}[theorem]{Remark}
\numberwithin{equation}{section}

\newcommand{\cal}{\mathcal}
\newcommand{\bff}{\boldsymbol}
\newcommand{\bb}{\mathbb}
\newcommand{\leqs}{\lesssim}
\newcommand{\geqs}{\gtrsim}

\newcommand{\dt}{\mathrm{d}t}
\newcommand{\dtt}{\mathrm{d}_t}
\newcommand{\ddt}{\frac{\mathrm{d}}{\mathrm{d}t}}
\newcommand{\dx}{\mathrm{d}\bff{x}}

\newcommand{\dtau}{\mathrm{d}\tau}
\newcommand{\norm}[2]{\left\|{#1}\right\|_{#2}}
\newcommand{\inpro}[2]{\left\langle#1,#2\right\rangle}
\newcommand{\abs}[1]{\left|{#1}\right|}

\usepackage[normalem]{ulem}
\newcommand\tsout{\bgroup\markoverwith{\textcolor{red}{\rule[0.5ex]{2pt}{1.4pt}}}\ULon}
\newcommand{\stkout}[1]{\ifmmode\text{\tsout{\ensuremath{#1}}}\else\tsout{#1}\fi}

\allowdisplaybreaks

\begin{document}
	\setcounter{page}{1}
	
	\title[The Landau--Lifshitz--Bloch equation with spin diffusion]
	{The Landau--Lifshitz--Bloch equation with spin diffusion: global strong solution and finite element approximation}

	\author{Agus L. Soenjaya}
	\address{School of Mathematics and Statistics, The University of New South Wales, Sydney 2052, Australia}
	\email{\textcolor[rgb]{0.00,0.00,0.84}{a.soenjaya@unsw.edu.au}}


	\begin{abstract}
	The spin-diffusion Landau--Lifshitz--Bloch (SDLLB) system is a nonlinearly coupled system of quasilinear vector-valued PDEs which models the interaction between spin-polarised currents and magnetisation at high temperatures. The aim of this paper is twofold. Firstly, assuming the initial data is sufficiently small, we show the existence of a unique global strong solution to the SDLLB equation in a bounded domain $\Omega\subset \bb{R}^d$, where $d\leq 3$, thus ensuring well-posedness of the model. Secondly, we propose a decoupled linearised fully-discrete finite element scheme to solve the problem. Despite the strong nonlinearity of the system, the proposed scheme only requires the solution of two completely decoupled linear systems per time-step. Assuming adequate regularity of the exact solution and a certain time-step constraint, we rigorously show that the numerical scheme converges at an optimal rate. Several numerical experiments corroborate our theoretical results.
	\end{abstract}
	\maketitle

\section{Introduction}

Micromagnetics is a field of physics dealing with the prediction of magnetic behaviour at sub-micrometre length scales. More recently, it was found that magnetisation can be manipulated by spin-polarised currents, even in the absence of any external magnetic field. This gives birth to the field of spintronics. Modern applications of this theory include the development of magnetoresistive random access memory (MRAM) and heat-assisted magnetic recording (HAMR) devices. It is important to note that in various applications, the temperature inside the devices may exceed the Curie temperature of the material. This underlies the importance of a model which incorporates interactions between magnetisation and spin-polarised currents at high temperatures. One such model is given by the spin-diffusion Landau--Lifshitz--Bloch (SDLLB) system, first analysed in~\cite{GuoLi20}. In the absence of currents, one recovers the Landau--Lifshitz--Bloch (LLB) equation~\cite{ChuNowChaGar06, Gar97, Le16}.

Fix a bounded domain $\Omega\subset \bb{R}^d$ ($d=1,2,3$) with boundary $\partial\Omega$. Let $\bff{m}:[0,T]\times \Omega\to \bb{R}^3$ and $\bff{s}:[0,T]\times \Omega\to \bb{R}^3$ be the local magnetisation and the spin accumulation, respectively.
The dynamics of the local magnetisation coupled with the spin accumulation in the regime above the Curie temperature can be described by a coupled system of vector-valued quasilinear equations~\cite{GuoLi20}, which we refer to as the SDLLB system:
\begin{subequations}\label{equ:sdllb}
	\begin{alignat}{2}
		\label{equ:sdllb a}
		&\partial_t \bff{m}
		=
		-\gamma \bff{m}\times\bff{H}
		+
		\alpha \bff{H}
		-
		\gamma' \bff{m}\times \bff{s}
		\qquad && \text{for $(t,\bff{x})\in(0,T)\times \Omega$,}
		\\[1ex]
		\label{equ:sdllb b}
		&\bff{H}= \alpha' \Delta \bff{m} - \kappa\mu \bff{m} - \kappa |\bff{m}|^2 \bff{m},
		\qquad && \text{for $(t,\bff{x})\in(0,T)\times \Omega$,}
		\\[1ex]
		\label{equ:sdllb c}
		&\partial_t \bff{s}
		=
		-\nabla\cdot \bff{J}
		- \frac{D_0}{\tau_{\mathrm{sf}}} \bff{s}
		- \frac{D_0}{\tau_{\mathrm{J}}} \bff{s}\times \bff{m}
		\qquad && \text{for $(t,\bff{x})\in(0,T)\times \Omega$,}
		\\[1ex]
		\label{equ:sdllb d}
		&\bff{J}
		=
		\beta' \bff{m}\otimes \bff{j}
		-
		D_0 \nabla \bff{s}
		+
		\beta D_0 (\bff{m}\otimes\bff{m}) \nabla \bff{s}
		\qquad && \text{for $(t,\bff{x})\in(0,T)\times \Omega$,}
		\\[1ex]
		\label{equ:sdllb e}
		&\bff{m}(0,\bff{x})= \bff{m}_0(\bff{x}),\;
		\bff{s}(0,\bff{x})= \bff{s}_0(\bff{x})
		\qquad && \text{for } \bff{x}\in {\Omega},
		\\[1ex]
		\label{equ:sdllb g}
		&\partial_{\bff{n}} \bff{m}=\partial_{\bff{n}} \bff{s}=\bff{0}
		\qquad && \text{for } (t,\bff{x})\in (0,T) \times \partial \Omega.
	\end{alignat}
\end{subequations}
{A more detailed discussion about this model can be found in Section~\ref{sec:model}.} 
Here, $\bff{n}$ is the outward pointing normal vector to $\partial\Omega$ and $\gamma':= \gamma \delta$. The physical meanings of other coefficients in~\eqref{equ:sdllb} are as follows: $\alpha'$ is the exchange field intensity, $\kappa=(2\chi)^{-1}$ is a positive constant related to the longitudinal susceptibility $\chi$ of the material, $\mu$ is a constant related to the equilibrium magnetisation and temperature, $\tau_{\mathrm{sf}}$ is the spin flip relaxation time, $\tau_{\mathrm{J}}$ is the spin transfer torque characteristic time, and $\beta, \beta' \in (0,1)$ are spin polarisation parameters.  All numerical coefficients in~\eqref{equ:sdllb} are positive.

For physical reasons, we assume that the function $\bff{j}:=\bff{j}(t,\bff{x})$ is a prescribed current density vector field which is tangential to the boundary ($\bff{j}\cdot\bff{n}=0$ on $\partial\Omega$), while $D_0:=D_0(\bff{x})$ is a given diffusion scalar field which is bounded above and below by positive constants. 
As is commonly done in the literature, we only consider energy contribution from the exchange interactions~\cite{ChuNowChaGar06, Le16, Soe24}:
\begin{equation}\label{equ:energy}
	\mathcal{E}(\bff{m}):= \frac{\alpha'}{2} \int_{\Omega} \abs{\nabla\bff{m}(\bff{x})}^2
	+
	\frac{\kappa}{4} \int_{\Omega} \abs{\bff{m}(\bff{x})}^4
	+
	\kappa\mu \int_{\Omega} \abs{\bff{m}(\bff{x})}^2 \,\dx,
\end{equation}
leading to the effective field $\bff{H}$ in~\eqref{equ:sdllb b}. Lower-order contributions from the first-order anisotropy field and the applied field could be considered without difficulty, but is omitted here for simplicity. Contribution from the demagnetisation field poses no significant problem analytically, but may be a source of numerical bottleneck. We refer the reader to~\cite{DiJunPraSla23, LiMaDuChen22, YanCheHu21} for further details about these issues.

We note that the spin diffusion equation coupled with the Landau--Lifshitz--Gilbert equation (SDLLG) is proposed in~\cite{GarWan07, GarWan07b} as a model for micromagnetics in the presence of spin currents for temperatures much lower than the Curie temperature. Numerical methods to solve such problem, focusing on multi-layer system, are proposed in~\cite{AbeHrkPagPraRugSue14, GarWan07, RugAbeHrkSuePra16}, where convergence (along subsequence) to a weak solution is shown (see also~\cite{BarPro06, DiPfePraRug20, EWan00, Pro01}, among others, for numerical integrators to the LLG equation without spin diffusion). Analysis of the SDLLG problem focusing on various mathematical issues such as regularity, long-time behaviour, and optimal control are done in~\cite{AnMajProTra22, DiJunPraSla23, PuGuo10, PuWan20, Ron22}. More detail about the physical background can be found in~\cite{Rug16}. On the other hand, the SDLLB system is much less studied. While the issues of well-posedness and numerical approximation of the Landau--Lifshitz--Bloch equation (without spin diffusion) have been considered in~\cite{BenEssAyo24, Le16, LeSoeTra24, LiGuoLiuLiu21, Soe25a}, not much is known about the SDLLB system, especially for $d=3$. The existence of a smooth solution to~\eqref{equ:sdllb} (which is global for sufficiently small initial data) is shown in~\cite{GuoLi20} for $d\leq 2$. As far as we know, numerical integrator for~\eqref{equ:sdllb} has not been considered before.

One of the aims of this paper is to continue the study in~\cite{GuoLi20} for the case $d=3$. Specifically, under certain smallness assumption on the initial data, we show the existence and uniqueness of global strong solution to the SDLLB in a smooth and bounded domain. We further propose a fully-discrete numerical scheme based on the Galerkin finite element method in space and the linearised Euler method in time to approximate the solution. We remark that despite the nonlinear coupling in~\eqref{equ:sdllb}, the proposed scheme only requires solving two completely decoupled linear systems in each time-step, which is an advantage in micromagnetics simulation~\cite{SunCheDuWan23}. Differently from the analysis of numerical schemes for the LLG or the SDLLG equation where only convergence along subsequence without rate is shown~\cite{AbeHrkPagPraRugSue14}, here we proceed to show an error estimate for the approximation, since the solution to the SDLLB is expected to be regular for a smooth and sufficiently small initial data (or small $\beta$) in a regular bounded domain. Assuming adequate regularity of the exact solution, an optimal rate of convergence is shown, as corroborated by several numerical experiments.

In contrast with the SDLLG equation, the problem~\eqref{equ:sdllb} does not automatically admit a pointwise bound on the magnetisation magnitude $\abs{\bff{m}}$, which complicates the analysis. On the other hand, the SDLLB equation possesses a stronger damping term (given by $\alpha\bff{H}$) which can be exploited to show a better regularity for the solution, at least for sufficiently small initial data (or small spin polarisation parameters) and in a smooth domain. Note that a certain smallness assumption on $\bff{m}$ is expected for well-posedness, otherwise, at least formally, the last term in~\eqref{equ:sdllb d} will dominate, turning \eqref{equ:sdllb c} into a backward diffusion equation for which ill-posedness of the associated initial-value problem is known. 

The existence and uniqueness results for the SDLLB are established by proving several uniform a priori estimates for the solution, which require careful analysis due to the nonlinearities involved in the coupled system. The error analysis of the numerical integrator is performed by defining appropriate elliptic projections tailored to the problem at hand and analysing their properties. Here, we work with a single-layer magnetic domain $\Omega$ as in~\cite{DiJunPraSla23, GuoLi20, PuGuo10, PuWan20, Ron22}, which is physically relevant. We briefly remark that for a multi-layer domain, the magnetisation vector field $\bff{m}$ on the smaller domain needs to be extended by zero to the larger domain, potentially creating discontinuity in the coefficients of the drift-diffusion equation (cf.~\cite[Remark~3.1]{DiJunPraSla23}). This causes substantial difficulties in the theoretical and numerical analysis, and will be a subject for future research.

To summarise, the main contributions of this paper include:
\begin{enumerate}
	\item proving the existence and uniqueness of a global strong solution to the SDLLB system with small initial data (or small spin polarisation parameters) in spatial dimensions $d\leq 3$ (Theorem~\ref{the:exist}),
	\item proposing a linear, decoupled, fully-discrete finite element method for solving the SDLLB system (Algorithm~\ref{alg:scheme}),
	\item establishing an error estimate for the proposed numerical scheme under a suitable time-step restriction and assumptions on the norm of the initial data (Theorem~\ref{the:error}).
\end{enumerate}
The paper is organised as follows: 
\begin{itemize}
	\item Section~\ref{sec:prelim} discusses the model in more detail, and gathers the notations and auxiliary results used throughout the paper;
	\item Section~\ref{sec:exist} establishes the global well-posedness for the SDLLB system;
	\item Section~\ref{sec:fem} presents and analyses the fully decoupled finite element approximation of the SDLLB system;
	\item Section~\ref{sec:num exp} contains numerical experiments that verify the theoretical convergence rates of the proposed numerical scheme.
\end{itemize}

\section{Preliminaries}\label{sec:prelim}

\subsection{Notations}
We begin by defining some notations used in this paper. Let $\Omega\subset \bb{R}^d$ be a domain. The function space $\bb{L}^p := \bb{L}^p(\Omega; \bb{R}^3)$ denotes the space of $p$-th integrable functions taking values in $\bb{R}^3$ and $\bb{W}^{k,p} := \bb{W}^{k,p}(\Omega; \bb{R}^3)$ denotes the Sobolev space of 
functions on $\Omega$, taking values in $\bb{R}^3$. As usual, $\bb{H}^k(\Omega) := \bb{W}^{k,2}(\Omega)$. {The space $\widetilde{\bb{H}}^{-1}(\Omega)$ denotes the dual of $\bb{H}^1(\Omega)$, with the duality pairing defined as the extension of the $\bb{L}^2(\Omega)$ inner product.} For brevity, we write $\bb{L}^p$, $\bb{W}^{k,p}$, or $\bb{H}^k$ in lieu of $\bb{L}^p(\Omega)$, $\bb{W}^{k,p}(\Omega)$, or $\bb{H}^k(\Omega)$ respectively. Denote $\Omega_T:= (0,T)\times \Omega$. Let $\Delta$ be the Neumann Laplacian operator acting on $\bb{R}^3$-valued functions with domain
\begin{equation*}
	\bb{H}^2_{\bff{n}}:= \left\{\bff{v}\in \bb{H}^2 : \frac{\partial \bff{v}}{\partial \bff{n}} = \bff{0} \text{ on } \partial{\Omega} \right\}.
\end{equation*}

For a Banach space $\bb{X}$, the spaces $L^p(0,T;\bb{X})$ and $W^{k,p}(0,T;\bb{X})$ denote respectively the Lebesgue and Sobolev spaces of functions on $(0,T)$ taking values in $\bb{X}$. The space $C([0,T];\bb{X})$ denotes the space of continuous function on $[0,T]$ taking values in $\bb{X}$. For simplicity, we will write
\[ 
	L^p_T(\bb{X}) := L^p(0,T; \bb{X}) 
	\quad\text{and}\quad 
	W^{k,p}_T(\bb{X}) := W^{k,p}(0,T; \bb{X}).
\] 
We do not distinguish between the scalar product of $\bb{L}^2$ vector-valued functions taking values in $\bb{R}^3$ and the scalar product of $\bb{L}^2$ matrix-valued functions taking values in $\bb{R}^{3\times 3}$, and denote them both by $\inpro{\cdot}{\cdot}$. 

Finally, the constant $C$ in the estimate denotes a
generic constant which may take different values at different occurrences. If
the dependence of $C$ on a variable, e.g.~$T$, is emphasised, we will write
$C_T$. The notation $a\leqs b$ means $a\leq Cb$ for some constant $C$.

\subsection{The model}\label{sec:model}

We now provide a more detailed discussion of the SDLLB model under consideration.
{Recall that $\bff{m}$ and $\bff{s}$ represent the local magnetisation and the spin accumulation, respectively. The vector field $\bff{H}: [0,T]\times \Omega\to \bb{R}^3$ is the effective field, which is defined as the negative variational derivative of the micromagnetic energy $\mathcal{E}$, i.e.
	\begin{equation*}
		\bff{H}[\bff{m}]:= -\nabla_{\bff{m}} \mathcal{E}.
	\end{equation*} 
	The time evolution of the magnetisation vector field at elevated temperatures can be described by the LLB equation \cite{ChuNowChaGar06, Gar91}:
	\begin{align}\label{equ:llb general}
		\partial_t \bff{m}
		=
		-\gamma \bff{m}\times \bff{H} + \frac{\gamma \alpha_{\parallel}}{|\bff{m}|^2}\, (\bff{m}\cdot \bff{H})\bff{m} - \frac{\gamma\alpha_{\perp}}{|\bff{m}|^2}\, \bff{m} \times (\bff{m}\times \bff{H}),
	\end{align}
	where $\alpha_\parallel$ and $\alpha_\perp$ are, respectively, the longitudinal and transverse damping parameters, and $\gamma$ is the gyromagnetic ratio. The damping parameters $\alpha_\parallel$ and $\alpha_\perp$ are related to the temperature $\Theta$, the Curie temperature $T_\mathrm{c}$, and the Gilbert damping parameter $\alpha_G$ in the standard Landau--Lifshitz equation via
	\begin{align}\label{equ:alpha par and perp}
		\begin{cases}
			\alpha_{\parallel} = \frac{2\alpha_G \Theta}{3T_{\mathrm{c}}}, \quad
			\alpha_{\perp} = \alpha_G \left(1- \frac{\Theta}{3T_{\mathrm{c}}} \right), &\text{ if $\Theta<T_\mathrm{c}$},
			\\[1ex]
			\alpha_{\parallel} = \alpha_{\perp} = \frac{2\alpha_G}{3}, &\text{ if $\Theta\geq T_\mathrm{c}$}.
		\end{cases}
	\end{align}
	
	In contrast to the LLG equation, equation~\eqref{equ:alpha par and perp} does not preserve the magnitude of the magnetisation vector in general.
	However, note that at zero kelvin, $\alpha_\parallel= 0$ and $\alpha_\perp=\alpha_G$, thus equation \eqref{equ:llb general} reduces to
	\begin{align}\label{equ:ll special}
		\partial_t \bff{m}
		=
		-\gamma \bff{m}\times \bff{H} - \frac{\gamma\alpha_G}{|\bff{m}|^2}\, \bff{m} \times (\bff{m}\times \bff{H}).
	\end{align}
	Formally taking the inner product of \eqref{equ:ll special} with $\bff{m}$ shows that $\abs{\bff{m}}$ is a constant in this special case, reducing \eqref{equ:ll special} to the standard Landau--Lifshitz equation.
	In the high-temperature regime, noting that $\alpha_\parallel=\alpha_\perp= 2\alpha_G/3$ for $\Theta\geq T_\mathrm{c}$ from \eqref{equ:alpha par and perp}, by the vector triple product identity the LLB equation simplifies to
	\begin{align}\label{equ:llb above Curie}
		\partial_t \bff{m} = -\gamma \bff{m}\times \bff{H} + \alpha \bff{H},
	\end{align}
	where $\alpha:= 2\gamma \alpha_G/{3}$. This is the form of the LLB equation we are considering in this paper. A derivation of \eqref{equ:llb general} and \eqref{equ:llb above Curie} within the continuum thermodynamic framework can be found in~\cite{BerGio16}.
	
	In the presence of an electric current, the local magnetisation experiences a torque due to the spin accumulation $\bff{s}$, whose dynamics is described by a nonlinear drift–diffusion equation~\cite{ShpLevZha03, ZhaLevFer02}. Following~\cite{BooChuCha20, ZhaLevFer02, ZamJun16}, the coupling between the magnetisation and spin accumulation, with strength $\delta>0$, can be introduced by appending the term $\delta \bff{s}$ to the effective field $\bff{H}$ in the LLB magnetisation equation \eqref{equ:llb above Curie}. This coupling generates additional torques in \eqref{equ:llb above Curie}: a transverse (precessional) torque $-\gamma\delta \bff{m}\times \bff{s}$ and a longitudinal torque $\alpha\delta\bff{s}$.
	
	For analytical tractability, we consider a simplified model in the regime above the Curie temperature. In this regime, strong thermal agitation reduces the efficiency of a spin-polarised current in generating a net spin accumulation, resulting in a relatively small effective coupling $\delta$. Moreover, spin relaxation rates increase with temperature due to enhanced scattering, causing $\bff{s}$ to decay rapidly and remain small compared with other contributions to the effective field~\cite{BooChuCha20}. Since $\alpha_G \ll 1$~\cite{ZhaSonYan16} and $\delta$ is small, the longitudinal torque $\alpha\delta\bff{s}$ is therefore negligible relative to the damping term $\alpha \bff{H}$ or the precessional torque $\gamma\delta \bff{m}\times \bff{s}$, and it is dropped in our model. This leads to the SDLLB system~\eqref{equ:sdllb} studied here, which coincides with the model in~\cite{GuoLi20}.}

\subsection{Problem formulations}\label{sec:assum}

Let ${\Omega}\subset \bb{R}^d$ be a bounded smooth or a convex polytopal domain. {The following standing assumptions are used throughout the paper:}
\begin{enumerate}[(i)]
	\item the function $\bff{j}:\Omega_T \to \bb{R}^d$ is given such that $\bff{j}\in L^\infty(\Omega_T)$ and $\bff{j}\cdot \bff{n}=0$ on $\partial\Omega$;
	\item the function $D_0:\Omega \to\bb{R}$ is given such that $D_0$ is bounded above and below by positive constants, i.e. there exist positive constants $D_\ast$ and $D^\ast$ such that $D_\ast \leq D_0 \leq D^\ast$. 
\end{enumerate}
Further regularity assumptions will be introduced as necessary in the corresponding theorems. {In the analysis, we always assume $d=3$ for ease of presentation, noting that similar argument holds for $d\leq 2$.} We introduce the following notions of solution to \eqref{equ:sdllb}.

\begin{definition}[weak solution]\label{def:weak}
Let $\bff{m}_0\in \bb{H}^1$ and $\bff{s}_0\in \bb{H}^1$ be given. A \emph{weak solution} to \eqref{equ:sdllb} is a pair $(\bff{m},\bff{s})$, where $\bff{m}:\Omega_T\to \bb{R}^3$ and $\bff{s}:\Omega_T\to \bb{R}^3$ {satisfy}:
\begin{enumerate}[(i)]
	\item $\bff{m}\in H^1(0,T;\widetilde{\bb{H}}^{-1}) \cap C([0,T];\bb{L}^2) \cap L^\infty(0,T;\bb{L}^\infty) \cap L^2(0,T;\bb{H}^1)$;
	\item $\bff{s} \in H^1(0,T;\widetilde{\bb{H}}^{-1}) \cap C([0,T];\bb{L}^2) \cap L^2(0,T;\bb{H}^1)$;
	\item $\bff{m}(0)=\bff{m}_0$ and $\bff{s}(0)=\bff{s}_0$;
	\item for almost all $t\in (0,T)$,
	\begin{align}
		\label{equ:weak m}
		\inpro{\partial_t \bff{m}}{\bff{\phi}}
		&=
		\gamma\alpha' \inpro{\bff{m}\times \nabla \bff{m}}{\nabla\bff{\phi}}
		-
		\alpha\alpha' \inpro{\nabla\bff{m}}{\nabla\bff{\phi}}
		-
		\alpha\kappa\mu \inpro{\bff{m}}{\bff{\phi}}
		\nonumber\\
		&\quad
		-
		\alpha\kappa \inpro{|\bff{m}|^2 \bff{m}}{\bff{\phi}}
		-
		\gamma' \inpro{\bff{m}\times \bff{s}}{\bff{\phi}},
		\quad \forall \bff{\phi}\in \bb{H}^1,
		\\
		\label{equ:weak s}
		\inpro{\partial_t \bff{s}}{\bff{\psi}}
		&=
		\beta' \inpro{\bff{m}\otimes \bff{j}}{\nabla \bff{\psi}}
		- \inpro{D_0\nabla\bff{s}}{\nabla\bff{\psi}}
		+ \beta \inpro{D_0(\bff{m}\otimes \bff{m})\nabla \bff{s}}{\nabla\bff{\psi}}
		\nonumber\\
		&\quad
		-
		\frac{1}{\tau_{\mathrm{sf}}} \inpro{D_0\bff{s}}{\bff{\psi}}
		-
		\frac{1}{\tau_{\mathrm{J}}} \inpro{D_0\bff{s}\times \bff{m}}{\bff{\psi}}
		,
		\quad \forall \bff{\psi}\in \bb{H}^1;
	\end{align}
\end{enumerate}
A weak solution is local if the solution exists only for sufficiently small $T$. It is global if the solution exists for arbitrary positive $T$.
\end{definition}

\begin{definition}[strong solution]\label{def:strong}
A \emph{strong solution} $(\bff{m},\bff{s})$ to~\eqref{equ:sdllb} is a weak solution with additional regularity:
\begin{enumerate}[(i)]
	\item $\bff{m}\in H^1(0,T;\bb{H}^1) \cap C([0,T]; \bb{H}^2) \cap L^2(0,T;\bb{H}^3)$;
	\item $\bff{s}\in H^1(0,T;\bb{H}^1) \cap C([0,T];\bb{H}^2 ) \cap L^2(0,T;\bb{H}^3)$.
\end{enumerate}
In this case, $(\bff{m},\bff{s})$ satisfies \eqref{equ:sdllb} for almost every $(t,\bff{x})\in \Omega_T$.
\end{definition}

The above definition of weak solution is in the spirit of~\cite{AnMajProTra22, Le16}, noting that $\bff{j}\cdot \bff{n}=0$ on $\partial\Omega$ by assumption. 
Strong solution of the SDLLB system satisfies the following energy equality {for all $t\in (0,T)$}:
\begin{align}\label{equ:energy ineq}
	\mathcal{E}(\bff{m}(t))
	+
	\alpha \int_0^t \norm{\bff{H}(\tau)}{\bb{L}^2}^2 \dtau
	=
	\mathcal{E}(\bff{m}_0)
	+
	\int_0^t \gamma' \inpro{\bff{m}(\tau)\times \bff{s}(\tau)}{\bff{H}(\tau)} \dtau,
\end{align}
where $\mathcal{E}$ is the micromagnetic energy functional defined by~\eqref{equ:energy}.
The energy equality~\eqref{equ:energy ineq} can be motivated as follows: Formally taking the inner product of~\eqref{equ:sdllb a} with $\bff{H}$, and of \eqref{equ:sdllb b} with $\partial_t \bff{m}$, then subtracting the resulting equations, we obtain
\begin{align}\label{equ:energy arg}
	\frac{\alpha'}{2} \ddt \norm{\nabla\bff{m}}{\bb{L}^2}^2
	+
	\frac{\kappa}{4} \ddt \norm{\bff{m}}{\bb{L}^4}^4
	+
	\kappa\mu \ddt \norm{\bff{m}}{\bb{L}^2}^2
	+
	\alpha\norm{\bff{H}}{\bb{L}^2}^2
	=
	\gamma' \inpro{\bff{m}\times \bff{s}}{\bff{H}}.
\end{align}
Integrating this over $(0,t)$ then yields~\eqref{equ:energy ineq}.

%

For ease of presentation, in the analysis we set all numerical coefficients in~\eqref{equ:weak m} and \eqref{equ:weak s}, except for $\beta$, to 1.

\subsection{Finite elements}

Let $\big\{\mathcal{T}_h\big\}_{h>0}$ be a family of shape-regular and quasi-uniform triangulations of ${\Omega}$ into triangles or tetrahedra with maximal mesh-size $h$.
Next, we introduce the conforming finite element space $\bb{V}_h \subset \bb{H}^1$ given by
\begin{equation}\label{equ:Vh}
\bb{V}_h := \left\{\bff{\phi}_h \in \bff{C}(\overline{\Omega}; \bb{R}^3): \left.\bff{\phi}_h\right|_K \in \cal{P}_r(K;\bb{R}^3), \; \forall K \in \cal{T}_h \right\},
\end{equation}
where $\cal{P}_r(K; \bb{R}^3)$ denotes the space of polynomials of {degree} at most $r$ on $K$ taking values in $\bb{R}^3$.
For a technical reason (see Lemma~\ref{lem:stab elliptic}), we assume that in~\eqref{equ:Vh} the polynomial degree $r$ in $\bb{V}_h$ is:
\begin{equation}\label{equ:deg r}
	\begin{cases}
		r\geq 1, &\text{if } d\in \{1,2\},
		\\
		r\geq 2, &\text{if } d=3.
	\end{cases}
\end{equation}

Due to the regularity of the triangulation, we have the following best approximation property: for $p\in [1,\infty]$, there exists a constant $C$ independent of $h$ such that for any $\bff{v} \in \bb{W}^{r+1,p}$,
\begin{align}\label{equ:fin approx}
	\inf_{\chi \in {\bb{V}}_h} \big\{ \norm{\bff{v} - \bff{\chi}}{\bb{L}^p} 
	+ 
	h \norm{\nabla (\bff{v}-\bff{\chi})}{\bb{L}^p} 
	\big\} 
	\leq 
	C h^{r+1} \norm{\bff{v}}{\bb{W}^{r+1,p}}.
\end{align}
In the analysis, we will use the finite element projection operator $P_h: \bb{L}^2\to \bb{V}_h$ defined by
\begin{align}\label{equ:L2 proj}
	\inpro{P_h \bff{v}-\bff{v}}{\bff{\chi}}=0, \quad \forall \bff{\chi}\in \bb{V}_h.
\end{align}
Let $p\in [1,\infty]$. The projector $P_h$ satisfies the following boundedness and approximation properties~\cite{CroTho87, DouDupWah74}:
\begin{align}
	\label{equ:Ph Lp stab}
	\norm{P_h \bff{v}}{\bb{L}^p} &\leq C_p \norm{\bff{v}}{\bb{L}^p},
	\\
	\label{equ:Ph H1 stab}
	\norm{P_h \bff{v}}{\bb{W}^{1,p}} &\leq C_p \norm{\bff{v}}{\bb{W}^{1,p}},
	\\
	\label{equ:Ph approx}
	\norm{\bff{v} - P_h(\bff{v})}{\bb{L}^2}
	+
	h \norm{\nabla\big(\bff{v} - P_h(\bff{v})\big)}{\bb{L}^2}
	&\leq 
	C h^{r+1} \|\bff{v}\|_{\bb{H}^{r+1}}.
\end{align}
Other projection operators used in this paper will be introduced and analysed in Section~\ref{sec:fem}.

Finally, the following inverse estimates are well-known: there exist constants $C_{\textrm{inv}}$ and $\widetilde{C}_{\textrm{inv}}$ (which depend on the regularity of the triangulation, but is independent of $h$) such that
\begin{align}
	\label{equ:inverse}
	\norm{\bff{v}}{\bb{L}^\infty} &\leq C_{\textrm{inv}} h^{-d/2} \norm{\bff{v}}{\bb{L}^2}, \quad \forall \bff{v}\in \bb{V}_h,
	\\
	\label{equ:inverse infty H1}
	\norm{\bff{v}}{\bb{L}^\infty} &\leq \widetilde{C}_{\textrm{inv}} \ell_h \norm{\bff{v}}{\bb{H}^1}, \quad\forall \bff{v}\in \bb{V}_h,
\end{align}
where
\begin{equation}\label{equ:ell h}
	\ell_h:=
	\begin{cases}
		1, &\text{if } d=1,
		\\
		\abs{\log h}^{1/2}, &\text{if } d=2,
		\\
		h^{-1/2}, &\text{if }d=3.
	\end{cases}
\end{equation}
The inverse estimates hold under the quasi-uniformity assumption on the triangulation.

\section{The existence and uniqueness of global strong solution} \label{sec:exist}

{To establish the global existence and uniqueness of solutions to~\eqref{equ:sdllb}, one could apply the~Faedo--Galerkin method: derive suitable uniform estimates for the approximate solutions and apply weak compactness argument to extract a subsequence which solves the problem. However, in order to streamline the presentation, we will work directly (and formally) with smooth solution $(\bff{m},\bff{s})$ of problem~\eqref{equ:sdllb} in lieu of its Faedo--Galerkin approximations. The formal a priori estimates presented below can be rigorously justified by appealing to the Galerkin approximation framework, following the approach in~\cite{Le16, LeSoeTra24}. Uniqueness of strong solutions (in the sense of Definition~\ref{def:strong}) will be shown by an energy-based argument.} 
	
The main result of this section is stated in Theorem~\ref{the:exist}. For clarity of exposition, we recall that all numerical coefficients in \eqref{equ:sdllb} have been set to 1, except for $\beta$. We begin with a lemma that demonstrates the exponential decay of the $\bb{L}^p$ norm of the solution $\bff{m}$. This decay corresponds to the eventual loss of magnetisation above the Curie temperature (see also Remark~\ref{rem:phys}, and~\cite{AtxChuKaz07, KilFin12, LeSoeTra24} for supporting evidence).

\begin{lemma}
Let $(\bff{m},\bff{s})$ be a smooth solution of~\eqref{equ:sdllb} with initial data $\bff{m}_0\in \bb{L}^p$, where $p\in [2,\infty]$. Then for all $t\in [0,T]$,
\begin{align}\label{equ:u Lp}
	\norm{\bff{m}(t)}{\bb{L}^p} 
	&\leq 
	e^{-t} \norm{\bff{m}_0}{\bb{L}^p}.
\end{align}
Moreover, for $p\in [2,\infty)$,
\begin{align}\label{equ:u grad Lp}
	\int_0^t \norm{|\bff{m}(\tau)|^{\frac{p-2}{2}} |\nabla \bff{m}(\tau)|}{\bb{L}^2}^2 \dtau
	&\leq C,
\end{align}
where $C$ depends on the coefficients of the equation and $\norm{\bff{m}_0}{\bb{L}^p}$.
\end{lemma}

\begin{proof}
Taking the inner product of \eqref{equ:sdllb a} with $|\bff{m}|^{p-2} \bff{m}$ for $p\geq 2$, and noting \eqref{equ:sdllb b}, we obtain
\begin{align*}
	\frac{1}{p} \ddt \norm{\bff{m}}{\bb{L}^p}^p
	+
	\inpro{\nabla \bff{m}}{\nabla (\abs{\bff{m}}^{p-2} \bff{m})}
	+
	\norm{\bff{m}}{\bb{L}^p}^p
	+
	\norm{\bff{m}}{\bb{L}^{p+2}}^{p+2}
	=
	0.
\end{align*}
Therefore,
\begin{align*}
	\frac{1}{p} \ddt \norm{\bff{m}}{\bb{L}^p}^p
	&+
	\norm{|\bff{m}|^{\frac{p-2}{2}} |\nabla \bff{m}|}{\bb{L}^2}^2
	+
	(p-2) \norm{|\bff{m}|^{\frac{p-4}{2}} |\bff{m}\cdot\nabla\bff{m}|}{\bb{L}^2}^2
	\\
	&+
	\norm{\bff{m}}{\bb{L}^p}^p
	+
	\norm{\bff{m}}{\bb{L}^{p+2}}^{p+2}
	=
	0.
\end{align*}
Integrating this over $(0,t)$ gives \eqref{equ:u grad Lp}. Furthermore, dropping irrelevant terms, we have
\[
\ddt \norm{\bff{m}(t)}{\bb{L}^p}^p + p \norm{\bff{m}(t)}{\bb{L}^p}^p
\leq 0,
\]
so that
\[
\ddt
\Big(
e^{p t} \norm{\bff{m}(t)}{\bb{L}^p}^p
\Big)
\le 0.
\]
Integrating over $(0,t)$ and taking $p$-th root {show}~\eqref{equ:u Lp} for $p\in [2,\infty)$. Letting $p\to \infty$ in~\eqref{equ:u Lp} completes the proof for the case $p=\infty$.
\end{proof}

In the next lemma, we assume a certain smallness condition on $\beta$ or $\norm{\bff{m}_0}{\bb{L}^\infty}$ to derive a uniform estimate for $\bff{s}$.

\begin{lemma}\label{lem:s L2}
Let $(\bff{m},\bff{s})$ be a smooth solution of \eqref{equ:sdllb} with initial data $\bff{m}_0\in \bb{L}^\infty$ and $\bff{s}_0\in \bb{L}^2$, such that $\beta D^\ast \norm{\bff{m}_0}{\bb{L}^\infty}^2 < D_\ast$. Then for all $t\in [0,T]$,
\begin{align}\label{equ:s L2}
	\norm{\bff{s}(t)}{\bb{L}^2}^2
	+
	\int_0^t \norm{\nabla \bff{s}(\tau)}{\bb{L}^2}^2 \dtau
	\leq
	C,
\end{align}
where the constant $C$ depends on the coefficients of the equation, $\norm{\bff{m}_0}{\bb{L}^\infty}$, and $\norm{\bff{s}_0}{\bb{L}^2}$. {Furthermore, we have}
\begin{align}\label{equ:limsup s zero}
	{\lim_{t\to\infty} \norm{\bff{s}(t)}{\bb{L}^2}^2 = 0.}
\end{align}
\end{lemma}

\begin{proof}
Taking the inner product of \eqref{equ:sdllb c} with $\bff{s}$, noting \eqref{equ:sdllb d} and our standing assumptions, gives
\begin{align*}
	&\frac12 \ddt \norm{\bff{s}}{\bb{L}^2}^2
	+
	\inpro{D_0 \nabla \bff{s}}{\nabla \bff{s}}
	+
	\inpro{D_0 \bff{s}}{\bff{s}}
	=
	\inpro{\bff{m}\otimes \bff{j}}{\nabla\bff{s}}
	+
	\beta \inpro{D_0(\bff{m}\otimes \bff{m}) \nabla \bff{s}}{\nabla\bff{s}}
\end{align*}
Noting the assumption on $D_0$, by Young's inequality we obtain for any $\epsilon>0$,
\begin{align}\label{equ:ineq s L2}
	&\frac12 \ddt \norm{\bff{s}}{\bb{L}^2}^2
	+
	D_\ast \norm{\nabla \bff{s}}{\bb{L}^2}^2
	+
	D_\ast \norm{\bff{s}}{\bb{L}^2}^2
	\nonumber\\
	&\leq
	C \norm{\bff{j}}{L^\infty(\Omega_T)}^2 \norm{\bff{m}}{\bb{L}^\infty}^2
	+
	\beta D^\ast \norm{\bff{m}}{\bb{L}^\infty}^2 \norm{\nabla \bff{s}}{\bb{L}^2}^2
	+
	\epsilon \norm{\nabla \bff{s}}{\bb{L}^2}^2
	\nonumber\\
	&\leq
	Ce^{-2t} \norm{\bff{m}_0}{\bb{L}^\infty}^2
	+ 
	\beta D^\ast e^{-2t} \norm{\bff{m}_0}{\bb{L}^\infty}^2 \norm{\nabla \bff{s}}{\bb{L}^2}^2
	+ 
	\epsilon \norm{\nabla \bff{s}}{\bb{L}^2}^2 ,
\end{align}
where in the last step we used the assumption on $\bff{j}$ and \eqref{equ:u Lp}. Now, if $\beta D^\ast \norm{\bff{m}_0}{\bb{L}^\infty}^2 < D_\ast$, then noting that $e^{-2t}\leq 1$ and choosing $\epsilon>0$ sufficiently small we can absorb the last two terms on the right-hand side. Integrating over $(0,t)$ then yields \eqref{equ:s L2}.

Furthermore, \eqref{equ:ineq s L2} and the assumption $\norm{\bff{m}_0}{\bb{L}^\infty}^2 < D_\ast/(\beta D^\ast)$ imply
\begin{align}\label{equ:asymp s}
	\ddt \norm{\bff{s}}{\bb{L}^2}^2 + 2D_\ast \norm{\bff{s}}{\bb{L}^2}^2 \leq
	\frac{Ce^{-2t} D_\ast}{\beta D^\ast}.
\end{align}
By a variant of the Gronwall inequality~\cite[Theorem 2.1.5]{Qin17}, we obtain
\begin{align*}
	\norm{\bff{s}(t)}{\bb{L}^2}^2
	&\leq
	\norm{\bff{s}_0}{\bb{L}^2}^2 e^{-2D_\ast t} + \int_0^t C e^{-2D_\ast (t-\tau)} \frac{e^{-2\tau} D_\ast}{\beta D^\ast} \,\dtau
	\\
	&=
	\norm{\bff{s}_0}{\bb{L}^2}^2 e^{-2D_\ast t} + \frac{CD_\ast}{\beta D^\ast (2D_\ast -2)} \left(e^{-2t}-e^{-2D_\ast t}\right).
\end{align*}
In particular, we have \eqref{equ:limsup s zero}. This completes the proof of the lemma.
\end{proof}

Further estimates on the norms of spatial derivatives of $\bff{m}$ are derived in the following lemmas.

\begin{lemma}\label{lem:nab m L2}
Let $(\bff{m},\bff{s})$ be a smooth solution of \eqref{equ:sdllb} with initial data $\bff{m}_0\in \bb{H}^1\cap \bb{L}^\infty$ and $\bff{s}_0\in \bb{L}^2$, such that $\beta D^\ast \norm{\bff{m}_0}{\bb{L}^\infty}^2 < D_\ast$. Then for all $t\in [0,T]$,
\begin{align}\label{equ:nab m L2}
	\norm{\nabla \bff{m}(t)}{\bb{L}^2}^2
	+
	\int_0^t \norm{\Delta \bff{m}(\tau)}{\bb{L}^2}^2 \dtau
	\leq
	C,
\end{align}
where the constant $C$ depends on the coefficients of the equation, $\norm{\bff{m}_0}{\bb{H}^1 \cap \bb{L}^\infty}$, and $\norm{\bff{s}_0}{\bb{L}^2}$.
\end{lemma}

\begin{proof}
Taking the inner product of \eqref{equ:sdllb a} with $-\Delta \bff{m}$ and integrating by parts, we have
\begin{align*}
	&\frac12 \ddt \norm{\nabla \bff{m}}{\bb{L}^2}^2
	+
	\norm{\Delta \bff{m}}{\bb{L}^2}^2
	+
	\norm{\nabla \bff{m}}{\bb{L}^2}^2
	+
	\norm{|\bff{m}| |\nabla \bff{m}|}{\bb{L}^2}^2
	+
	2\norm{\bff{m} \cdot \nabla \bff{m}}{\bb{L}^2}^2
	\\
	&=
	-
	\inpro{\bff{m}\times \bff{s}}{\Delta \bff{m}}
	\\
	&\leq
	\frac12 \norm{\Delta \bff{m}}{\bb{L}^2}^2
	+
	\frac12 \norm{\bff{m}}{\bb{L}^\infty}^2 \norm{\bff{s}}{\bb{L}^2}^2.
\end{align*}
Integrating both sides over $(0,t)$, noting \eqref{equ:u Lp} and \eqref{equ:s L2}, we obtain the required result.
\end{proof}

\begin{remark}\label{rem:phys}
Physically, \eqref{equ:u Lp} and \eqref{equ:limsup s zero} indicate that above the Curie temperature, both the local magnetisation and the spin accumulation decay to zero in a certain sense. Under the LLB framework considered here, $\abs{\bff{m}}$ can vary in space and time since $\bff{m}$ represents the \emph{local magnetisation} (or macrospin) at high temperatures. Indeed, above the Curie temperature, $\bff{m}(t)$ even decays exponentially to zero. We emphasise that this does not mean that the individual \emph{magnetic spin} vectors vanish (they are still of unit length), but rather that they become sufficiently disordered so that the averaged (or local) magnetisation $\bff{m}$ at any point tends to zero, corresponding to the loss of macroscopic magnetisation above the Curie point~\cite{AtxChuKaz07, KilFin12}. The loss of magnetisation and the enhanced scattering weaken the coupling between $\bff{m}$ and $\bff{s}$ in such a way that spin relaxation mechanisms dominate over spin current injection, resulting in an eventual decay of $\bff{s}$ to zero as well~\cite{BooChuCha20}.
\end{remark}

\begin{remark}
We comment on the physical plausibility of the assumption $\beta D^\ast \norm{\bff{m}_0}{\bb{L}^\infty}^2 < D_\ast$ used in Lemmas~\ref{lem:s L2} and \ref{lem:nab m L2}, and later in Theorem~\ref{the:exist} to establish the existence of a global weak solution. In the SDLLG system, a similar smallness condition on $\beta$ is typically imposed. Notably, if $\norm{\bff{m}_0}{\bb{L}^\infty}=1$, then our assumption simplifies to $\beta<D_\ast/D^\ast$, which is consistent with the condition intended in~\cite{AbeHrkPagPraRugSue14}. We note that the statement of~\cite[Lemma~5]{AbeHrkPagPraRugSue14} contains a minor inaccuracy, but once corrected, the resulting constraint for parabolicity reduces precisely to $\beta<D_\ast/D^\ast$ in our notation. If moreover $D_0$ is constant in space, then this assumption reduces to $\beta\in (0,1)$. 
Equivalently, for a given $\beta\in (0,1)$, the assumption holds if $\norm{\bff{m}_0}{\bb{L}^\infty}^2< D_\ast/(\beta D^\ast)$, which can occur naturally in a sufficiently `weak' ferromagnet (small local magnetisation almost everywhere) or at a sufficiently high temperature.
\end{remark}

\begin{lemma}\label{lem:Delta m L2}
Let $(\bff{m},\bff{s})$ be a smooth solution of \eqref{equ:sdllb} with initial data $\bff{m}_0\in \bb{H}^2$ and $\bff{s}_0\in \bb{L}^2$, such that 
\begin{align}\label{equ:small m0 CG}
	\norm{\bff{m}_0}{\bb{L}^\infty}^2 < 
	\min \big\{1/C_{\mathrm{G}}^2, D_\ast/(\beta D^\ast)\big\} =: K_0, 
\end{align}
where $C_\mathrm{G}:=C_1 C_2$ is a constant associated with the Gagliardo--Nirenberg inequalities \eqref{equ:C1} and \eqref{equ:C2}. Then for all $t\in [0,T]$,
\begin{align}\label{equ:Delta m L2}
		\norm{\Delta \bff{m}(t)}{\bb{L}^2}^2
		+
		\int_0^t \norm{\nabla\Delta \bff{m}(\tau)}{\bb{L}^2}^2 \dtau
		\leq
		C,
\end{align}
where the constant $C$ depends on the coefficients of the equation, $\norm{\bff{m}_0}{\bb{H}^2}$, and $\norm{\bff{s}_0}{\bb{L}^2}$.
\end{lemma}

\begin{proof}
Taking the inner product of \eqref{equ:sdllb a} with $\Delta^2 \bff{m}$ and integrating by parts as necessary, we have
\begin{align}\label{equ:J1 J4}
	&\frac12 \ddt \norm{\Delta \bff{m}}{\bb{L}^2}^2
	+
	\norm{\nabla\Delta \bff{m}}{\bb{L}^2}^2
	+
	\norm{\Delta \bff{m}}{\bb{L}^2}^2
	\nonumber\\
	&=
	\inpro{\nabla \bff{m}\times \Delta \bff{m}}{\nabla\Delta \bff{m}}
	+
	\inpro{\nabla(|\bff{m}|^2 \bff{m})}{\nabla\Delta \bff{m}}
	+
	\inpro{\nabla\bff{m}\times \bff{s}}{\nabla\Delta \bff{m}}
	+
	\inpro{\bff{m}\times \nabla\bff{s}}{\nabla\Delta \bff{m}}
	\nonumber\\
	&=: J_1+J_2+J_3+J_4.
\end{align}
We will estimate each term on the last line. Firstly, note that by the Gagliardo--Nirenberg inequalities in 3D,
\begin{align}\label{equ:C1}
	\norm{\nabla \bff{m}}{\bb{L}^6}
	&\leq
	C_1 \norm{\bff{m}}{\bb{L}^\infty}^{\frac23} \norm{\bff{m}}{\bb{H}^3}^{\frac13}
	\\
	\label{equ:C2}
	\norm{\Delta \bff{m}}{\bb{L}^3}
	&\leq
	C_2 \norm{\bff{m}}{\bb{L}^\infty}^{\frac13} \norm{\bff{m}}{\bb{H}^3}^{\frac23},
\end{align}
where $C_1$ and $C_2$ are constants depending on $\Omega$. Therefore, we infer that
\begin{align*}
	\abs{J_1}
	&\leq
	\norm{\nabla \bff{m}}{\bb{L}^6} \norm{\Delta \bff{m}}{\bb{L}^3} \norm{\nabla \Delta \bff{m}}{\bb{L}^2}
	\\
	&\leq
	C_\mathrm{G} \norm{\bff{m}}{\bb{L}^\infty} \norm{\bff{m}}{\bb{H}^3} \norm{\nabla\Delta \bff{m}}{\bb{L}^2}
	\\
	&\leq
	C_\mathrm{G} \norm{\bff{m}_0}{\bb{L}^\infty} \left(1+ \norm{\Delta \bff{m}}{\bb{L}^2}^2 \right)
	+
	\epsilon \norm{\nabla\Delta \bff{m}}{\bb{L}^2}^2
	+
	C_\mathrm{G} \norm{\bff{m}_0}{\bb{L}^\infty} \norm{\nabla\Delta \bff{m}}{\bb{L}^2}^2,
\end{align*}
for any $\epsilon>0$, where in the last step we used Young's inequality, \eqref{equ:u Lp}, and \eqref{equ:nab m L2}. For the term $J_2$, by Young's inequality we have
\begin{align*}
	\abs{J_2} 
	&\leq
	3 \norm{\bff{m}}{\bb{L}^\infty}^2 \norm{\nabla \bff{m}}{\bb{L}^2} \norm{\nabla\Delta \bff{m}}{\bb{L}^2}
	\leq
	C\norm{\bff{m}_0}{\bb{L}^\infty}^2 \norm{\nabla\bff{m}}{\bb{L}^2}^2
	+
	\epsilon \norm{\nabla\Delta \bff{m}}{\bb{L}^2}^2.
\end{align*}
Next, for the term $J_3$ we apply Young's and Agmon's inequalities to obtain
\begin{align*}
	\abs{J_3}
	&\leq
	\norm{\nabla \bff{m}}{\bb{L}^\infty} \norm{\bff{s}}{\bb{L}^2} \norm{\nabla\Delta \bff{m}}{\bb{L}^2}
	\\
	&\leq
	C \norm{\nabla \bff{m}}{\bb{H}^1}^{\frac12} \norm{\nabla\bff{m}}{\bb{H}^2}^{\frac12}  \norm{\bff{s}}{\bb{L}^2} \norm{\nabla\Delta \bff{m}}{\bb{L}^2}
	\\
	&\leq
	C\left(1+\norm{\Delta \bff{m}}{\bb{L}^2}^2\right) 
	+
	\epsilon \norm{\nabla\Delta \bff{m}}{\bb{L}^2}^2,
\end{align*}
where in the last step we used \eqref{equ:s L2} and \eqref{equ:nab m L2}. Similarly, for the term $J_4$ we have
\begin{align*}
	\abs{J_4}
	&\leq
	\norm{\bff{m}}{\bb{L}^\infty} \norm{\nabla \bff{s}}{\bb{L}^2} \norm{\nabla\Delta \bff{m}}{\bb{L}^2}
	\leq
	\norm{\bff{m}_0}{\bb{L}^\infty}^2 \norm{\nabla \bff{s}}{\bb{L}^2}^2
	+
	\epsilon \norm{\nabla\Delta \bff{m}}{\bb{L}^2}^2.
\end{align*}
Altogether, continuing from \eqref{equ:J1 J4} we obtain
\begin{align*}
	&\frac12 \ddt \norm{\Delta \bff{m}}{\bb{L}^2}^2
	+
	\norm{\nabla\Delta \bff{m}}{\bb{L}^2}^2
	+
	\norm{\Delta \bff{m}}{\bb{L}^2}^2
	\\
	&\leq
	C\left(1+ \norm{\Delta \bff{m}}{\bb{L}^2}^2 + \norm{\nabla \bff{s}}{\bb{L}^2}^2 \right)
	+
	C_\mathrm{G} \norm{\bff{m}_0}{\bb{L}^\infty} \norm{\nabla\Delta \bff{m}}{\bb{L}^2}^2
	+
	\epsilon \norm{\nabla\Delta \bff{m}}{\bb{L}^2}^2.
\end{align*}
Choosing $\epsilon>0$ sufficiently small and noting the assumption \eqref{equ:small m0 CG}, we can absorb the last two terms on the right-hand side above. Integrating with respect to time, noting \eqref{equ:s L2} and \eqref{equ:nab m L2} again, then yields \eqref{equ:Delta m L2}.
\end{proof}

Next, we derive estimates on the norms of spatial derivatives of $\bff{s}$ under further smoothness assumptions on $\bff{j}$ and $D_0$.

\begin{lemma}
Let $(\bff{m},\bff{s})$ be a smooth solution of \eqref{equ:sdllb} with initial data $\bff{m}_0\in \bb{H}^2$ and $\bff{s}_0\in \bb{H}^1$, such that \eqref{equ:small m0 CG} holds. Let $\bff{j}\in L^\infty\big(0,T;\bb{H}^1(\Omega;\bb{R}^d)\big)$ and $D_0\in W^{1,\infty}(\Omega)$ be given. Then for all $t\in [0,T]$,
\begin{align}\label{equ:nab s L2}
	\norm{\nabla \bff{s}(t)}{\bb{L}^2}^2
	+
	\int_0^t \norm{\Delta \bff{s}(\tau)}{\bb{L}^2}^2 \dtau
	\leq
	C,
\end{align}
where the constant $C$ depends on the coefficients of the equation, $\norm{\bff{m}_0}{\bb{H}^2}$, and $\norm{\bff{s}_0}{\bb{H}^1}$.
\end{lemma}

\begin{proof}
Taking the inner product of \eqref{equ:sdllb c} with $-\Delta \bff{s}$, we obtain
\begin{align*}
	\frac12 \ddt \norm{\nabla\bff{s}}{\bb{L}^2}^2
	&=
	\inpro{\nabla \cdot(\bff{m}\otimes \bff{j})}{\Delta \bff{s}}
	-
	\inpro{\nabla\cdot (D_0 \nabla \bff{s})}{\Delta \bff{s}}
	-
	\beta \inpro{\nabla \cdot \big(D_0 (\bff{m}\otimes \bff{m}) \nabla\bff{s}\big)}{\Delta \bff{s}}
	\\
	&\quad
	+
	\inpro{D_0 \bff{s}}{\Delta \bff{s}}
	+
	\inpro{D_0 \bff{s}\times\bff{m}}{\Delta \bff{s}}
	\\
	&=
	\inpro{\nabla \cdot(\bff{m}\otimes \bff{j})}{\Delta \bff{s}}
	-
	\inpro{D_0 \Delta \bff{s}}{\Delta\bff{s}}
	-
	\inpro{\nabla \bff{s} \cdot \nabla D_0}{\Delta \bff{s}}
	\\
	&\quad
	-
	\inpro{\nabla \cdot \big(D_0 (\bff{m}\otimes \bff{m}) \nabla\bff{s}\big)}{\Delta \bff{s}}
	+
	\inpro{D_0 \bff{s}}{\Delta \bff{s}}
	+
	\inpro{D_0 \bff{s}\times\bff{m}}{\Delta \bff{s}}.
\end{align*}
After rearranging the terms, we have
\begin{align}\label{equ:I1 to I5}
	\frac12 \ddt \norm{\nabla\bff{s}}{\bb{L}^2}^2
	+
	\inpro{D_0 \Delta \bff{s}}{\Delta\bff{s}}
	&=
	\inpro{\nabla \cdot(\bff{m}\otimes \bff{j})}{\Delta \bff{s}}
	-
	\inpro{\nabla \bff{s} \cdot \nabla D_0}{\Delta \bff{s}}
	+
	\inpro{D_0 \bff{s}}{\Delta \bff{s}}
	\nonumber\\
	&\quad
	-
	\beta\inpro{\nabla \cdot \big(D_0 (\bff{m}\otimes \bff{m}) \nabla\bff{s}\big)}{\Delta \bff{s}}
	+
	\inpro{D_0 \bff{s}\times\bff{m}}{\Delta \bff{s}}
	\nonumber\\
	&=: I_1+I_2+\cdots+I_5.
\end{align}
We need to estimate each term on the last line. The first term can be estimated in a straightforward manner by Young's inequality:
\begin{align*}
	\abs{I_1}
	&\leq
	\norm{\nabla \bff{m}}{\bb{L}^4} \norm{\bff{j}}{\bb{L}^4} \norm{\Delta \bff{s}}{\bb{L}^2}
	+
	\norm{\bff{m}}{\bb{L}^\infty} \norm{\nabla\bff{j}}{\bb{L}^2} \norm{\Delta \bff{s}}{\bb{L}^2}
	\\
	&\leq
	C+ \epsilon \norm{\Delta \bff{s}}{\bb{L}^2}^2,
\end{align*}
where we used the assumption on $\bff{j}$, the Sobolev embedding $\bb{H}^1\hookrightarrow \bb{L}^4$, \eqref{equ:Delta m L2}, and \eqref{equ:u Lp}. Similarly, by the assumption on $D_0$ and Young's inequality, we have
\begin{align*}
	\abs{I_2} 
	&\leq
	\norm{\nabla\bff{s}}{\bb{L}^2} \norm{\nabla D_0}{\bb{L}^\infty} \norm{\Delta \bff{s}}{\bb{L}^2}
	\leq
	C\norm{\nabla\bff{s}}{\bb{L}^2}^2 
	+
	\epsilon \norm{\Delta \bff{s}}{\bb{L}^2}^2,
	\\
	\abs{I_3}
	&\leq
	\norm{D_0}{L^\infty} \norm{\bff{s}}{\bb{L}^2} \norm{\Delta \bff{s}}{\bb{L}^2}
	\leq
	C\norm{\nabla\bff{s}}{\bb{L}^2}^2 
	+
	\epsilon \norm{\Delta \bff{s}}{\bb{L}^2}^2.
\end{align*}
For the term $I_4$, we expand the divergence term then apply H\"older's inequality to obtain
\begin{align*}
	\abs{I_4}
	&\leq
	\beta \norm{D_0}{W^{1,\infty}} \norm{\bff{m}}{\bb{L}^\infty}^2 \norm{\nabla \bff{s}}{\bb{L}^2} \norm{\Delta \bff{s}}{\bb{L}^2}
	\\
	&\quad
	+
	2\beta \norm{D_0}{L^\infty} \norm{\bff{m}}{\bb{L}^\infty} \norm{\nabla\bff{m}}{\bb{L}^4} \norm{\nabla \bff{s}}{\bb{L}^4} \norm{\Delta \bff{s}}{\bb{L}^2}
	\\
	&\quad
	+
	\beta \norm{D_0}{L^\infty} \norm{\bff{m}}{\bb{L}^\infty}^2 \norm{\Delta \bff{s}}{\bb{L}^2}^2
	\\
	&\leq
	C\norm{\nabla \bff{s}}{\bb{L}^2}^2 + \epsilon \norm{\Delta \bff{s}}{\bb{L}^2}^2
	+
	\beta D^\ast \norm{\bff{m}_0}{\bb{L}^\infty}^2 \norm{\Delta \bff{s}}{\bb{L}^2}^2,
\end{align*}
where in the last step we used the Gagliardo--Nirenberg inequalities, \eqref{equ:u Lp}, \eqref{equ:nab m L2}, and \eqref{equ:Delta m L2}. Finally, for the term $I_5$,
\begin{align*}
	\abs{I_5}
	&\leq
	\norm{D_0}{L^\infty} \norm{\bff{s}}{\bb{L}^2} \norm{\bff{m}}{\bb{L}^\infty} \norm{\Delta \bff{s}}{\bb{L}^2}
	\leq
	C\norm{\bff{s}}{\bb{L}^2}^2 + \epsilon \norm{\Delta \bff{s}}{\bb{L}^2}^2.
\end{align*}
Collecting the above estimates and continuing from \eqref{equ:I1 to I5}, we infer that for any $\epsilon>0$,
\begin{align*}
	\frac12 \ddt \norm{\nabla\bff{s}}{\bb{L}^2}^2
	+
	D_\ast \norm{\Delta \bff{s}}{\bb{L}^2}^2
	\leq
	C\norm{\nabla\bff{s}}{\bb{L}^2}^2
	+
	\epsilon \norm{\Delta \bff{s}}{\bb{L}^2}^2
	+
	\beta D^\ast \norm{\bff{m}_0}{\bb{L}^\infty}^2 \norm{\Delta \bff{s}}{\bb{L}^2}^2.
\end{align*}
If the quantity $\norm{\bff{m}_0}{\bb{L}^\infty}^2$ is small such that \eqref{equ:small m0 CG} holds, then we can absorb the last term on the right-hand side of the above inequality. Choosing $\epsilon>0$ sufficiently small and integrating over $(0,t)$, we obtain \eqref{equ:nab s L2}.
\end{proof}

\begin{lemma}\label{lem:Delta s L2}
Let $(\bff{m},\bff{s})$ be a smooth solution of \eqref{equ:sdllb} with initial data $\bff{m}_0\in \bb{H}^2$ and $\bff{s}_0\in \bb{H}^2$, such that \eqref{equ:small m0 CG} holds. Let $\bff{j}\in L^\infty\big(0,T;\bb{H}^2(\Omega;\bb{R}^d)\big)$ and $D_0\in W^{2,\infty}(\Omega)$ be given. Then for all $t\in [0,T]$,
\begin{align}\label{equ:Delta s L2}
	\norm{\Delta \bff{s}(t)}{\bb{L}^2}^2
	+
	\int_0^t \norm{\nabla \Delta \bff{s}(\tau)}{\bb{L}^2}^2 \dtau
	\leq
	C,
\end{align}
where the constant $C$ depends on the coefficients of the equation, $\norm{\bff{m}_0}{\bb{H}^2}$, and $\norm{\bff{s}_0}{\bb{H}^2}$.
\end{lemma}

\begin{proof}
Taking the inner product of \eqref{equ:sdllb c} with $\Delta^2 \bff{s}$ and integrating by parts as necessary, we have
\begin{align*}
	\frac12 \ddt \norm{\Delta \bff{s}}{\bb{L}^2}^2
	&=
	-\inpro{\nabla\cdot (\bff{m}\otimes \bff{j})}{\Delta^2 \bff{s}}
	+
	\inpro{\nabla \cdot (D_0\nabla \bff{s})}{\Delta^2 \bff{s}}
	\\
	&\quad
	-
	\inpro{D_0 \bff{s}}{\Delta^2 \bff{s}}
	-
	\inpro{D_0 \bff{s}\times \bff{m}}{\Delta^2 \bff{s}}
	-
	\beta \inpro{\nabla \cdot \big(D_0(\bff{m}\otimes\bff{m}) \nabla \bff{s}\big)}{\Delta^2 \bff{s}}
	\\
	&=
	\inpro{\nabla \big(\nabla\cdot(\bff{m}\otimes\bff{j})\big)}{\nabla\Delta \bff{s}}
	-
	\inpro{\nabla D_0 \otimes \Delta \bff{s}}{\nabla\Delta \bff{s}}
	-
	\inpro{D_0 \nabla\Delta \bff{s}}{\nabla\Delta \bff{s}}
	\\
	&\quad
	-
	\inpro{\nabla\big(\nabla\bff{s} \cdot \nabla D_0\big)}{\nabla\Delta \bff{s}}
	+
	\inpro{\nabla (D_0\bff{s})}{\nabla\Delta \bff{s}}
	+
	\inpro{\nabla (D_0 \bff{s}\times \bff{m})}{\nabla\Delta \bff{s}}
	\\
	&\quad
	+
	\beta \inpro{\nabla \big(D_0 \nabla\cdot (\bff{m}\otimes \bff{m}) \nabla \bff{s}\big)}{\nabla\Delta \bff{s}}
	+
	\beta \inpro{\nabla \big( (\bff{m}\otimes \bff{m} \cdot \nabla D_0 ) \nabla \bff{s}\big)}{\nabla\Delta \bff{s}}.
\end{align*}
Upon rearranging the terms, we obtain
\begin{align}\label{equ:J1 to J7}
	&\frac12 \ddt \norm{\Delta \bff{s}}{\bb{L}^2}^2
	+
	\inpro{D_0 \nabla\Delta \bff{s}}{\nabla\Delta \bff{s}}
	\nonumber\\
	&=
	\inpro{\nabla \big(\nabla\cdot(\bff{m}\otimes\bff{j})\big)}{\nabla\Delta \bff{s}}
	-
	\inpro{\nabla D_0 \otimes \Delta \bff{s}}{\nabla\Delta \bff{s}}
	-
	\inpro{\nabla\big(\nabla\bff{s} \cdot \nabla D_0\big)}{\nabla\Delta \bff{s}}
	\nonumber\\
	&\quad
	+
	\inpro{\nabla (D_0\bff{s})}{\nabla\Delta \bff{s}}
	+
	\inpro{\nabla (D_0 \bff{s}\times \bff{m})}{\nabla\Delta \bff{s}}
	\nonumber\\
	&\quad
	+
	\beta \inpro{\nabla \big(D_0 \nabla\cdot (\bff{m}\otimes \bff{m}) \nabla \bff{s}\big)}{\nabla\Delta \bff{s}}
	+
	\beta \inpro{\nabla \big( (\bff{m}\otimes \bff{m} \cdot \nabla D_0 ) \nabla \bff{s}\big)}{\nabla\Delta \bff{s}}
	\nonumber\\
	&=: J_1+J_2+\cdots+ J_7.
\end{align}
We now estimate each term on the last line. For the first term, we immediately have
\begin{align*}
	\abs{J_1}
	&\leq
	\norm{\bff{m}}{\bb{H}^2} \norm{\bff{j}}{\bb{H}^2} \norm{\nabla \Delta \bff{s}}{\bb{L}^2}
	\leq
	C + \epsilon \norm{\nabla \Delta \bff{s}}{\bb{L}^2}^2,
\end{align*}
where in the last step we used Young's inequality and \eqref{equ:Delta m L2}. For the terms $J_2$ and $J_3$, straightforward application of Young's inequality and the Leibniz rule gives
\begin{align*}
	\abs{J_2}
	&\leq
	\norm{\nabla D_0}{\bb{L}^\infty} \norm{\Delta \bff{s}}{\bb{L}^2} \norm{\nabla\Delta \bff{s}}{\bb{L}^2}
	\leq
	C \norm{\Delta \bff{s}}{\bb{L}^2}^2 + \epsilon \norm{\nabla \Delta \bff{s}}{\bb{L}^2}^2,
\end{align*}
and
\begin{align*} 
	\abs{J_3}
	&\leq
	\norm{\nabla \bff{s}\cdot \nabla D_0}{\bb{H}^1} \norm{\nabla\Delta \bff{s}}{\bb{L}^2}
	\\
	&\leq
	\norm{\nabla \bff{s}}{\bb{H}^1} \norm{D_0}{W^{1,\infty}} \norm{\nabla\Delta \bff{s}}{\bb{L}^2}
	+
	\norm{\nabla \bff{s}}{\bb{L}^2} \norm{D_0}{W^{2,\infty}} 
	\norm{\nabla\Delta \bff{s}}{\bb{L}^2}
	\\
	&\leq
	C \norm{\Delta \bff{s}}{\bb{L}^2}^2 + \epsilon \norm{\nabla \Delta \bff{s}}{\bb{L}^2}^2.
\end{align*}
For the terms $J_4$ and $J_5$, similar argument gives
\begin{align*}
	\abs{J_4}
	&\leq
	\left(\norm{D_0}{L^\infty} \norm{\nabla \bff{s}}{\bb{L}^2}
	+
	\norm{D_0}{W^{1,\infty}} \norm{\bff{s}}{\bb{L}^2} \right) \norm{\nabla\Delta \bff{s}}{\bb{L}^2}
	\\
	&\leq
	C\norm{\bff{s}}{\bb{H}^1}^2 + \epsilon \norm{\nabla\Delta \bff{s}}{\bb{L}^2}^2,
	\\
	\abs{J_5}
	&\leq
	\norm{D_0}{W^{1,\infty}} \left( \norm{\bff{s}}{\bb{H}^1} \norm{\bff{m}}{\bb{L}^\infty} + \norm{\bff{s}}{\bb{L}^4} \norm{\nabla \bff{m}}{\bb{L}^4} \right)
	\norm{\nabla\Delta \bff{s}}{\bb{L}^2}
		\\
	&\leq
	C\norm{\bff{s}}{\bb{H}^1}^2 + \epsilon \norm{\nabla\Delta \bff{s}}{\bb{L}^2}^2,
\end{align*}
where in the last step we also used the Sobolev embedding, \eqref{equ:u Lp}, and \eqref{equ:Delta m L2}. Similarly for the term $J_6$, by the Leibniz rule and H\"older's inequality we have
\begin{align*}
	\abs{J_6}
	&\leq
	\beta \norm{D_0}{W^{1,\infty}} \left( \norm{\bff{m}}{\bb{H}^2}^2 \norm{\nabla \bff{s}}{\bb{L}^\infty} + 2 \norm{\bff{m}}{\bb{L}^\infty} \norm{\nabla \bff{m}}{\bb{L}^\infty} \norm{\bff{s}}{\bb{H}^2} \right) 
	\norm{\nabla\Delta \bff{s}}{\bb{L}^2}
	\\
	&\leq
	C \Big( \norm{\nabla \bff{s}}{\bb{H}^1}^{\frac12} \norm{\nabla \bff{s}}{\bb{H}^2}^{\frac12} + \norm{\nabla\bff{m}}{\bb{L}^\infty} \norm{\bff{s}}{\bb{H}^2} \Big) \norm{\nabla \Delta \bff{s}}{\bb{L}^2}
	\\
	&\leq
	C\left(1+ \norm{\nabla \bff{m}}{\bb{L}^\infty}^2 \right) \left(1+ \norm{\Delta \bff{s}}{\bb{L}^2}^2 \right)
	+
	\epsilon \norm{\nabla\Delta \bff{s}}{\bb{L}^2}^2,
\end{align*}
where in the penultimate step we used Agmon's inequality, \eqref{equ:u Lp}, and \eqref{equ:Delta m L2}, while in the last step we used elliptic regularity~\cite{Gri11}, Young's inequality, \eqref{equ:s L2}, and \eqref{equ:nab s L2}. Finally, for the term $J_7$, by the same argument we infer that for any $\epsilon>0$,
\begin{align*}
	\abs{J_7}
	&\leq
	\beta \norm{D_0}{W^{1,\infty}} \left( \norm{\bff{m}}{\bb{L}^\infty} \norm{\nabla\bff{m}}{\bb{L}^\infty} 
	+ 
	\norm{\bff{m}}{\bb{L}^\infty}^2 \norm{\bff{s}}{\bb{H}^2} \right) \norm{\nabla\Delta \bff{s}}{\bb{L}^2} 
	\\
	&\quad
	+
	\beta \norm{D_0}{W^{2,\infty}} \norm{\bff{m}}{\bb{L}^\infty}^2 \norm{\nabla \bff{s}}{\bb{L}^2} \norm{\nabla\Delta \bff{s}}{\bb{L}^2} 
	\\
	&\leq
	C\left(1+ \norm{\nabla\bff{m}}{\bb{L}^\infty}\right)
	+
	C\norm{\Delta \bff{s}}{\bb{L}^2}^2
	+
	\epsilon \norm{\nabla\Delta \bff{s}}{\bb{L}^2}^2,
\end{align*}
where in the last step we used elliptic regularity, Young's inequality, \eqref{equ:u Lp}, \eqref{equ:s L2}, and \eqref{equ:nab s L2}. Altogether, continuing from \eqref{equ:J1 to J7} we have
\begin{align*}
	&\frac12 \ddt \norm{\Delta \bff{s}}{\bb{L}^2}^2
	+
	D_\ast \norm{\nabla \Delta \bff{s}}{\bb{L}^2}^2
	\leq
	C\left(1+ \norm{\nabla \bff{m}}{\bb{L}^\infty}^2 \right) \left(1+ \norm{\Delta \bff{s}}{\bb{L}^2}^2 \right)
	+
	\epsilon \norm{\nabla\Delta \bff{s}}{\bb{L}^2}^2.
\end{align*}
Choosing $\epsilon>0$ sufficiently small, noting \eqref{equ:Delta m L2} and the Sobolev embedding $\bb{H}^2\hookrightarrow \bb{L}^\infty$, we infer the required result by the Gronwall lemma.
\end{proof}

\begin{lemma}\label{lem:dt m s}
{Let $(\bff{m},\bff{s})$ be a smooth solution of \eqref{equ:sdllb} with initial data $\bff{m}_0\in \bb{H}^1\cap \bb{L}^\infty$ and $\bff{s}_0\in \bb{L}^2$, such that $\beta D^\ast \norm{\bff{m}_0}{\bb{L}^\infty}^2 < D_\ast$. Then for all $t\in [0,T]$,
\begin{align}\label{equ:dt m s weak}
	\norm{\bff{m}}{H^1_T(\widetilde{\bb{H}}^{-1})}
	+
	\norm{\bff{s}}{H^1_T(\widetilde{\bb{H}}^{-1})}
	\leq
	C,
\end{align}
where the constant $C$ depends on the coefficients of the equation, $\norm{\bff{m}_0}{\bb{H}^1 \cap \bb{L}^\infty}$, and $\norm{\bff{s}_0}{\bb{L}^2}$.}

If, in addition, $\bff{m}_0\in \bb{H}^2$ and $\bff{s}_0\in \bb{H}^2$ such that \eqref{equ:small m0 CG} holds, and moreover $\bff{j}\in L^\infty\big(0,T;\bb{H}^2(\Omega;\bb{R}^d)\big)$ and $D_0\in W^{2,\infty}(\Omega)$, then for all $t\in [0,T]$,
\begin{align}\label{equ:dt m s}
	\norm{\bff{m}}{H^1_T(\bb{H}^1)}
	+
	\norm{\bff{s}}{H^1_T(\bb{H}^1)}
	\leq
	C,
\end{align}
where the constant $C$ depends on the coefficients of the equation, $\norm{\bff{m}_0}{\bb{H}^2}$, and $\norm{\bff{s}_0}{\bb{H}^2}$.
\end{lemma}

\begin{proof}
First, we show \eqref{equ:dt m s weak}. We have by \eqref{equ:weak m} and H\"older's inequality,
\begin{align*}
	\norm{\partial_t \bff{m}}{\widetilde{\bb{H}}^{-1}}
	&=
	\sup_{\bff{\varphi}\in \bb{H}^1,\; \norm{\bff{\varphi}}{\bb{H}^1} \leq 1} \big| \inpro{\partial_t \bff{m}}{\bff{\varphi}} \big|
	\\
	&\leq
	C\norm{\bff{m}}{\bb{L}^\infty} \norm{\nabla \bff{m}}{\bb{L}^2}
	+
	C\norm{\nabla \bff{m}}{\bb{L}^2}
	+
	C\norm{\bff{m}}{\bb{L}^2}
	+
	C\norm{\bff{m}}{\bb{L}^6}^3
	+
	C\norm{\bff{m}}{\bb{L}^\infty} \norm{\bff{s}}{\bb{L}^2}
	\leq C,
\end{align*}
where in the last step we used \eqref{equ:u Lp}, \eqref{equ:s L2}, \eqref{equ:nab m L2}, and the embedding $\bb{H}^1\hookrightarrow \bb{L}^6$. This implies $\norm{\partial_t \bff{m}}{L^2_T(\widetilde{\bb{H}}^{-1})} \leq C$. Similarly, by \eqref{equ:weak s},
\begin{align*}
	\norm{\partial_t \bff{s}}{\widetilde{\bb{H}}^{-1}}
	&=
	\sup_{\bff{\varphi}\in \bb{H}^1,\; \norm{\bff{\varphi}}{\bb{H}^1} \leq 1} \big| \inpro{\partial_t \bff{s}}{\bff{\varphi}} \big|
	\\
	&\leq
	C \norm{\bff{m}}{\bb{L}^2} \norm{\bff{j}}{L^\infty(\Omega_T)}
	+
	C \norm{\nabla \bff{s}}{\bb{L}^2}
	+
	C \norm{\bff{m}}{\bb{L}^\infty}^2 \norm{\nabla \bff{s}}{\bb{L}^2}
	+
	C \norm{\bff{s}}{\bb{L}^2} \norm{\bff{m}}{\bb{L}^\infty}
	\\
	&\leq
	C+ C\norm{\nabla \bff{s}}{\bb{L}^2},
\end{align*}
where in the last step we used \eqref{equ:u Lp} and \eqref{equ:s L2}. This implies $\norm{\partial_t \bff{s}}{L^2_T(\widetilde{\bb{H}}^{-1})} \leq C$ by \eqref{equ:s L2}, thus completing the proof of \eqref{equ:dt m s weak}.
	
Next, we prove \eqref{equ:dt m s}.
From \eqref{equ:sdllb a}, we have
\begin{align*}
	\norm{\partial_t \bff{m}}{L^2_T(\bb{L}^2)}
	&\leq
	\norm{\bff{m}}{L^\infty_T(\bb{L}^\infty)} \norm{\bff{H}}{L^2_T(\bb{L}^2)}
	+
	\norm{\bff{H}}{L^2_T(\bb{L}^2)}
	+
	\norm{\bff{m}}{L^\infty_T(\bb{L}^\infty)}
	\norm{\bff{s}}{L^2_T(\bb{L}^2)}
	\leq C,
\end{align*}
where in the last step we used \eqref{equ:u Lp}, \eqref{equ:s L2}, and \eqref{equ:Delta m L2}. Similarly, from \eqref{equ:sdllb c} we infer
\begin{align*}
	\norm{\partial_t \bff{s}}{L^2_T(\bb{L}^2)}
	&\leq
	\norm{\bff{J}}{L^2_T(\bb{H}^1)}
	+
	\norm{D_0 \bff{s}}{L^2_T(\bb{L}^2)}
	+
	\norm{D_0 \bff{s}}{L^2_T(\bb{L}^2)}
	\norm{\bff{m}}{L^\infty_T(\bb{L}^\infty)}
	\leq C.
\end{align*}
Furthermore, we also have
\begin{align*}
	\norm{\nabla \partial_t \bff{m}}{L^2_T(\bb{L}^2)}
	&\leq
	\norm{\nabla \bff{m}}{L^2_T(\bb{L}^\infty)}
	\norm{\bff{H}}{L^\infty_T(\bb{L}^2)}
	+
	\norm{\bff{m}}{L^\infty_T(\bb{L}^\infty)}
	\norm{\nabla \bff{H}}{L^2_T(\bb{L}^2)}
	\\
	&\quad
	+
	\norm{\nabla \bff{H}}{L^2_T(\bb{L}^2)}
	+
	\norm{\nabla \bff{m}}{L^2_T(\bb{L}^\infty)}
	\norm{\bff{s}}{L^\infty_T(\bb{L}^2)}
	\\
	&\quad
	+
	\norm{\bff{m}}{L^\infty_T(\bb{L}^\infty)}
	\norm{\nabla \bff{s}}{L^2_T(\bb{L}^2)}
	\leq C,
\end{align*}
where in the last step we used \eqref{equ:u Lp}, \eqref{equ:s L2}, \eqref{equ:Delta m L2}, and the Sobolev embedding $\bb{H}^3\hookrightarrow \bb{W}^{1,\infty}$. By similar argument, from \eqref{equ:sdllb c} we infer that
\begin{align*}
	\norm{\nabla \partial_t \bff{s}}{L^2_T(\bb{L}^2)}
	&\leq
	\norm{\bff{J}}{L^2_T(\bb{H}^2)}
	+
	\norm{D_0 \bff{s}}{L^2_T(\bb{H}^1)}
	+
	\norm{D_0 \bff{s}}{L^2_T(\bb{H}^1)} \norm{\bff{m}}{L^\infty_T(\bb{L}^\infty)}
	\\
	&\quad
	+
	\norm{D_0 \bff{s}}{L^\infty_T(\bb{L}^2)}
	\norm{\bff{m}}{L^2_T(\bb{W}^{1,\infty})}
	\leq C,
\end{align*}
where in the last step we also used the fact that from \eqref{equ:sdllb d},
\begin{align*}
	\norm{\bff{J}}{L^2_T(\bb{H}^2)}
	&\leq
	\norm{\bff{m}}{L^\infty_T(\bb{H}^2)}
	\norm{\bff{j}}{L^2_T(\bb{H}^2)}
	+
	\norm{\bff{s}}{L^2_T(\bb{H}^3)}
	\\
	&\quad
	+
	\beta \norm{D_0}{\bb{H}^2}
	\norm{\bff{m}}{L^\infty_T(\bb{H}^2)}^2
	\norm{\bff{s}}{L^2_T(\bb{H}^3)}
	\leq C,
\end{align*}
by \eqref{equ:Delta m L2} and \eqref{equ:Delta s L2}. This completes the proof of the lemma.
\end{proof}

We can now state the main theorem of this section.

\begin{theorem}\label{the:exist}
{Let $T>0$. Let initial data $\bff{m}_0\in \bb{H}^1\cap \bb{L}^\infty$ and $\bff{s}_0\in \bb{H}^1$ be given, such that $\beta D^\ast \norm{\bff{m}_0}{\bb{L}^\infty}^2 < D_\ast$. Then there exists a global weak solution (in the sense of Definition~\ref{def:weak}) to the problem~\eqref{equ:sdllb}. This solution satisfies 
\begin{equation}\label{equ:m less m0}	
\norm{\bff{m}}{L^\infty_T(\bb{L}^\infty)} \leq \norm{\bff{m}_0}{\bb{L}^\infty}.
\end{equation}
	
In addition, if $\bff{m}_0\in \bb{H}^2$ and $\bff{s}_0\in \bb{H}^2$ such that \eqref{equ:small m0 CG} holds, and $\bff{j}\in L^\infty\big(0,T;\bb{H}^2(\Omega;\bb{R}^d)\big)$ and $D_0\in W^{2,\infty}(\Omega)$,} then there exists a unique global strong solution (in the sense of Definition~\ref{def:strong}) to the problem~\eqref{equ:sdllb}. This strong solution satisfies energy equality \eqref{equ:energy ineq}.
\end{theorem}

\begin{proof}
{First, we show the existence of a global weak solution. A straightforward application of the Banach--Alaoglu theorem and a compactness argument, utilising a priori estimates \eqref{equ:u Lp}, \eqref{equ:u grad Lp}, \eqref{equ:s L2}, \eqref{equ:nab m L2}, \eqref{equ:dt m s weak}, and the Aubin--Lions lemma yields the required global weak solution in the sense of Definition \ref{def:weak}. Since this part of the argument is standard (see e.g. \cite{Le16, LeSoeTra24}), further detail is omitted. Note that the continuity in time follows from the Lions--Magenes interpolation lemma (Theorem~II.5.13 of \cite{BoyFab13}):
\begin{align*}
	H^1(0,T;\widetilde{\bb{H}}^{-1}) \cap L^2(0,T;\bb{H}^1) \hookrightarrow C([0,T];\bb{L}^2).
\end{align*}	
Inequality \eqref{equ:m less m0} follows from \eqref{equ:u Lp}, noting that $e^{-t}\leq 1$ for $t\geq 0$.}

{If further regularity on the initial data, $\bff{j}$, and $D_0$ are assumed, then we obtain $\bff{m}\in L^\infty_T(\bb{H}^2) \cap L^2_T(\bb{H}^3)$ and $\bff{s}\in L^\infty_T(\bb{H}^2) \cap L^2_T(\bb{H}^3)$ by uniform estimates \eqref{equ:Delta m L2} and \eqref{equ:Delta s L2}. A priori estimate \eqref{equ:dt m s} further shows that $\bff{m}\in H^1_T(\bb{H}^1)$ and $\bff{s}\in H^1_T(\bb{H}^1)$.} By the embedding
\[
L^2(0,T;\bb{H}^3) \cap H^1(0,T;\bb{H}^1) \hookrightarrow C([0,T];\bb{H}^2),
\]
we obtain a strong solution with regularity specified in Definition~\ref{def:strong}. The energy equality~\eqref{equ:energy ineq} now follows by the argument leading to~\eqref{equ:energy arg}, noting that a strong solution $(\bff{m},\bff{s})$ satisfies \eqref{equ:sdllb} for almost every $(t,\bff{x})\in \Omega_T$ as well as the regularity of strong solution $\bff{m}$ and $\bff{s}$ obtained above.

It remains to show uniqueness of this strong solution. To this end, suppose that $(\bff{m},\bff{s})$ and $(\bff{m}',\bff{s}')$ are two strong solutions corresponding to initial data $(\bff{m}_0,\bff{s}_0)$ and $(\bff{m}_0',\bff{s}_0')$, respectively. Let $\bff{u}:= \bff{m}-\bff{m}'$, $\bff{r}:=\bff{s}-\bff{s}'$, $\bff{u}_0:= \bff{m}_0-\bff{m}_0'$, and $\bff{r}_0:=\bff{s}_0 -\bff{s}_0'$. Then for almost every $(t,\bff{x})\in \Omega_T$, the function $\bff{u}$ satisfies
\begin{align}\label{equ:u unique}
	\partial_t \bff{u}
	&=
	- \big(\bff{m}\times \Delta \bff{u}+ \bff{u}\times \Delta \bff{m}'\big)
	+ \Delta \bff{u}
	- \bff{u}
	- \big( |\bff{m}|^2 \bff{u}+ ((\bff{m}+\bff{m}')\cdot \bff{u})\bff{m}' \big)
	\nonumber\\
	&\quad
	- \big( \bff{m}\times \bff{r}+ \bff{u} \times \bff{s}' \big),
\end{align}
while for almost every $(t,\bff{x})\in \Omega_T$, the function $\bff{r}$ satisfies
\begin{align}\label{equ:r unique}
	\partial_t \bff{r}
	&=
	- \nabla\cdot (\bff{u}\otimes \bff{j})
	+ \nabla\cdot (D_0\nabla \bff{r})
	- \beta \nabla \cdot \left(D_0 \big[(\bff{m}\otimes \bff{m}) \nabla \bff{r} + (\bff{m}\otimes \bff{u}) \nabla \bff{s}' + (\bff{u}\otimes \bff{m}') \nabla \bff{s}'\big] \right)
	\nonumber\\
	&\quad
	- D_0 \bff{r}
	- D_0 \big(\bff{r}\times \bff{m}+ \bff{s}'\times \bff{u} \big),
\end{align}
with initial data $\bff{u}_0 \in\bb{H}^2$ and $\bff{r}_0\in \bb{H}^2$, as well as boundary data $\partial_{\bff{n}} \bff{u}=\partial_{\bff{n}} \bff{r}=\bff{0}$.

We will establish a continuous dependence estimate with respect to the initial data, which will imply uniqueness. Taking the inner product of \eqref{equ:u unique} with $\bff{u}$ gives
\begin{align*}
	&\frac12 \ddt \norm{\bff{u}}{\bb{L}^2}^2
	+
	\norm{\nabla \bff{u}}{\bb{L}^2}^2
	+
	\norm{\bff{u}}{\bb{L}^2}^2
	+
	\norm{|\bff{m}| |\bff{u}|}{\bb{L}^2}^2
	+
	\norm{\bff{m}' \cdot \bff{u}}{\bb{L}^2}^2
	\\
	&=
	- \inpro{\nabla \bff{m}\times \bff{u}}{\nabla \bff{u}}
	- \inpro{(\bff{m}\cdot\bff{u})\bff{m}'}{\bff{u}}
	- \inpro{\bff{m}\times \bff{r}}{\bff{u}}
	\\
	&\leq
	\frac12 \norm{\nabla \bff{u}}{\bb{L}^2}^2
	+
	\frac12 \norm{\nabla \bff{m}}{\bb{L}^\infty}^2 \norm{\bff{u}}{\bb{L}^2}^2
	+
	\frac12 \norm{|\bff{m}| |\bff{u}|}{\bb{L}^2}^2
	+
	\frac12 \norm{\bff{m}' \cdot \bff{u}}{\bb{L}^2}^2
	\\
	&\quad
	+
	\frac12 \norm{\bff{m}}{\bb{L}^\infty}^2 \norm{\bff{u}}{\bb{L}^2}^2
	+
	\frac12 \norm{\bff{r}}{\bb{L}^2}^2,
\end{align*}
where in the last step we used Young's inequality. Rearranging the terms, we obtain
\begin{align}\label{equ:ineq ddt u}
	\ddt \norm{\bff{u}}{\bb{L}^2}^2
	+
	\norm{\bff{u}}{\bb{H}^1}^2
	\leq
	\norm{\bff{m}}{\bb{W}^{1,\infty}}^2 \norm{\bff{u}}{\bb{L}^2}^2
	+
	\norm{\bff{r}}{\bb{L}^2}^2.
\end{align}
Next, taking the inner product of \eqref{equ:r unique} with $\bff{r}$ gives
\begin{align}\label{equ:I1 to I5 un}
	&\frac12 \ddt \norm{\bff{r}}{\bb{L}^2}^2
	+
	\inpro{D_0 \nabla \bff{r}}{\nabla\bff{r}}
	+
	\inpro{D_0\bff{r}}{\bff{r}}
	\nonumber\\
	&=
	- \inpro{\nabla \cdot (\bff{u}\otimes\bff{j})}{\bff{r}}
	- \inpro{D_0\bff{s}' \times \bff{u}}{\bff{r}}
	\nonumber\\
	&\quad
	+ \beta \inpro{D_0(\bff{m}\otimes \bff{m})\nabla\bff{r}}{\nabla\bff{r}}
	+ \beta \inpro{D_0(\bff{m}\otimes \bff{u}) \nabla\bff{s}'}{\nabla\bff{r}}
	+ \beta \inpro{D_0(\bff{u}\otimes\bff{m}') \nabla\bff{s}'}{\nabla\bff{r}}
	\nonumber\\
	&\quad
	=: I_1+I_2+\cdots+I_5.
\end{align}
We estimate each term in the following. Firstly, by H\"older's and Young's inequalities,
\begin{align*}
	\abs{I_1}
	&\leq
	\norm{\bff{j}}{L^\infty_T(\bb{L}^\infty)} \norm{\nabla \bff{u}}{\bb{L}^2} \norm{\bff{r}}{\bb{L}^2}
	+
	\norm{\bff{j}}{L^\infty_T(\bb{W}^{1,4})} \norm{\bff{u}}{\bb{L}^4} \norm{\bff{r}}{\bb{L}^2}
	\\
	&\leq
	\frac14 \norm{\bff{u}}{\bb{H}^1}^2
	+
	C \norm{\bff{j}}{L^\infty_T(\bb{H}^2)}^2 \norm{\bff{r}}{\bb{L}^2}^2,
\end{align*}
where in the last step we used the Sobolev embedding $\bb{H}^2\hookrightarrow \bb{W}^{1,4} \hookrightarrow \bb{L}^\infty$. For the second term, similarly we have
\begin{align*}
	\abs{I_2} 
	\leq
	\frac12 \norm{D_0}{L^\infty}^2 \norm{\bff{s}'}{\bb{L}^\infty}^2 \norm{\bff{u}}{\bb{L}^2}^2
	+
	\frac12 \norm{\bff{r}}{\bb{L}^2}^2
	\leq
	\frac12 D^\ast  \norm{\bff{s}'}{\bb{L}^\infty}^2 \norm{\bff{u}}{\bb{L}^2}^2
	+
	\frac12 \norm{\bff{r}}{\bb{L}^2}^2.
\end{align*}
For the term $I_3$, by the same argument we obtain
\begin{align*}
	\abs{I_3} 
	\leq	
	\beta D^\ast \norm{\bff{m}}{\bb{L}^\infty}^2 \norm{\nabla\bff{r}}{\bb{L}^2}^2
	\leq
	\beta D^\ast \norm{\bff{m}_0}{\bb{L}^\infty}^2  \norm{\nabla\bff{r}}{\bb{L}^2}^2,
\end{align*}
where in the last step we used \eqref{equ:u Lp}. For the terms $I_4$ and $I_5$, by Young's inequality and \eqref{equ:u Lp} again, we have for any $\epsilon>0$,
\begin{align*}
	\abs{I_4}
	&\leq
	\epsilon \beta \norm{\nabla\bff{r}}{\bb{L}^2}^2
	+
	\frac{\beta D^\ast}{4\epsilon} \norm{\bff{m}_0}{\bb{L}^\infty}^2 \norm{\nabla\bff{s}'}{\bb{L}^\infty}^2 \norm{\bff{u}}{\bb{L}^2}^2,
	\\
	\abs{I_5}
	&\leq
	\epsilon \beta  \norm{\nabla\bff{r}}{\bb{L}^2}^2
	+
	\frac{\beta D^\ast}{4\epsilon} \norm{\bff{m}_0'}{\bb{L}^\infty}^2 \norm{\nabla\bff{s}'}{\bb{L}^\infty}^2 \norm{\bff{u}}{\bb{L}^2}^2.
\end{align*}
Altogether, continuing from \eqref{equ:I1 to I5 un} we have for any $\epsilon>0$,
\begin{align*}
	\frac12 \ddt \norm{\bff{r}}{\bb{L}^2}^2
	+
	D_\ast \norm{\bff{r}}{\bb{H}^1}^2
	&\leq
	C \norm{\bff{r}}{\bb{L}^2}^2
	+
	\frac14 \norm{\bff{u}}{\bb{H}^1}^2
	+
	\frac12 D^\ast  \norm{\bff{s}'}{\bb{L}^\infty}^2 \norm{\bff{u}}{\bb{L}^2}^2
	\\
	&\quad
	+
	\big(\beta D^\ast \norm{\bff{m}_0}{\bb{L}^\infty}^2  + 2\epsilon \beta \big) \norm{\nabla\bff{r}}{\bb{L}^2}^2
	\\
	&\quad
	+
	\frac{\beta D^\ast}{4\epsilon} 
	\left(\norm{\bff{m}_0}{\bb{L}^\infty}^2 + \norm{\bff{m}_0'}{\bb{L}^\infty}^2 \right) \norm{\nabla\bff{s}'}{\bb{L}^\infty}^2 \norm{\bff{u}}{\bb{L}^2}^2.
\end{align*}
Now, by the assumption \eqref{equ:small m0 CG}, since $\norm{\bff{m}_0}{\bb{L}^\infty}^2 < D_\ast/(\beta D^\ast)$, we can choose 
\[
\epsilon= \frac{1}{4\beta} \left( D_\ast - \beta D^\ast \norm{\bff{m}_0}{\bb{L}^\infty}^2 \right) > 0,
\]
to absorb the term containing $\norm{\nabla\bff{r}}{\bb{L}^2}^2$ to the left-hand side. This yields
\begin{align}\label{equ:ineq ddt r}
	\ddt \norm{\bff{r}}{\bb{L}^2}^2
	\leq
	C \norm{\bff{r}}{\bb{L}^2}^2
	+
	\frac14 \norm{\bff{u}}{\bb{H}^1}^2
	+
	C \norm{\bff{s}'}{\bb{W}^{1,\infty}}^2 \norm{\bff{u}}{\bb{L}^2}^2.
\end{align}
Adding \eqref{equ:ineq ddt u} and \eqref{equ:ineq ddt r}, and rearranging the terms, we obtain
\begin{align*}
	\ddt \left( \norm{\bff{u}}{\bb{L}^2}^2 + \norm{\bff{r}}{\bb{L}^2}^2 \right)
	\leq
	C \left(1+ \norm{\bff{m}}{\bb{W}^{1,\infty}}^2 + \norm{\bff{s}'}{\bb{W}^{1,\infty}}^2 \right) 
	\left( \norm{\bff{u}}{\bb{L}^2}^2+ \norm{\bff{r}}{\bb{L}^2}^2 \right).
\end{align*}
Noting the regularity of the strong solution, by the Gronwall lemma and the Sobolev embedding $\bb{H}^3 \hookrightarrow \bb{W}^{1,\infty}$, we have the continuous dependence estimate
\begin{align*}
	\norm{\bff{u}}{\bb{L}^2}^2 + \norm{\bff{r}}{\bb{L}^2}^2
	\leq
	C \left(\norm{\bff{u}_0}{\bb{L}^2}^2 + \norm{\bff{r}_0}{\bb{L}^2}^2 \right),
\end{align*}
where $C$ depends on $T$, $\norm{\bff{m}_0}{\bb{H}^2}$, and $\norm{\bff{s}_0}{\bb{H}^2}$. This concludes the proof of the theorem.
\end{proof}

\section{Finite element approximation} \label{sec:fem}

Let $\bb{V}_h$ be the finite element space in \eqref{equ:Vh}, and let $k$ be the time-step size. Let $(\bff{m}_h^n, \bff{s}_h^n) \in \bb{V}_h  \times \bb{V}_h$ be the numerical approximation of $(\bff{m}(t_n), \bff{s}(t_n))$, where $t_n:= nk\in [0,T]$ and $n\in \{0,1,\ldots, N\}$, and $N:=\lfloor T/k \rfloor$. 
We denote $\bff{v}^n:= \bff{v}(t_n)$. For any discrete function $\bff{v}$, define for $n\in\bb{N}$,
\begin{align*}
	\dtt \bff{v}^n := \frac{\bff{v}^n-\bff{v}^{n-1}}{k}.
\end{align*}

Let $(\bff{m}_h^0, \bff{s}_h^0)$ be a suitable approximation of the initial data. For ease of presentation, suppose that $\bff{j}\in C\big([0,T]; \bb{H}^1\big)$ so that the expression $\bff{j}^n$ are well-defined for $n\in \{0,1,\ldots, N\}$. It is possible to replace $\bff{j}^n$ by a numerical approximation $\bff{j}_h^n$ without changing the analysis significantly.
A fully discrete linearised finite element scheme to solve \eqref{equ:sdllb} can be described as follows.

\begin{algorithm}\label{alg:scheme}
	Input: $(\bff{m}_h^0, \bff{s}_h^0), \{\bff{j}^n\}_{0\leq n\leq N}$.
	\\
	For $n=1,2,\ldots,N$, iterate:
	\begin{enumerate}
		\item compute $\bff{m}_h^n\in \bb{V}_h$ such that
		\begin{align}\label{equ:scheme m}
			\inpro{\dtt \bff{m}_h^n}{\bff{\phi}_h}
			&=
			\inpro{\bff{m}_h^{n-1} \times \nabla \bff{m}_h^n}{\nabla\bff{\phi}_h}
			-
			\inpro{\nabla \bff{m}_h^n}{\nabla \bff{\phi}_h}
			-
			\inpro{\bff{m}_h^n}{\bff{\phi}_h}
			\nonumber\\
			&\quad
			-
			\inpro{|\bff{m}_h^{n-1}|^2 \bff{m}_h^n}{\bff{\phi}_h}
			-
			\inpro{\bff{m}_h^n \times \bff{s}_h^{n-1}}{\bff{\phi}_h},
			\quad \forall \bff{\phi}_h\in \bb{V}_h;
		\end{align}
		\item compute $\bff{s}_h^n \in \bb{V}_h$ such that
		\begin{align}\label{equ:scheme s}
			\inpro{\dtt \bff{s}_h^n}{\bff{\phi}_h}
			&=
			\inpro{\bff{m}_h^{n-1} \otimes \bff{j}^n}{\nabla\bff{\psi}_h}
			-
			\inpro{D_0 \nabla \bff{s}_h^n}{\nabla \bff{\psi}_h}
			-
			\inpro{D_0 \bff{s}_h^n}{\bff{\psi}_h}
			\nonumber\\
			&\quad
			+
			\beta\inpro{D_0 \big(\bff{m}_h^{n-1}\otimes \bff{m}_h^{n-1}\big) \nabla \bff{s}_h^n}{\nabla \bff{\psi}_h}
			-
			\inpro{D_0 \bff{s}_h^n\times \bff{m}_h^{n-1}}{\bff{\psi}_h},
			\quad \forall \bff{\psi}_h\in \bb{V}_h;
		\end{align}
	\end{enumerate}
	Output: a sequence of discrete functions $\big\{(\bff{m}_h^n, \bff{s}_h^n)\big\}_{1\leq n\leq N}$.
\end{algorithm}

We remark that while \eqref{equ:sdllb} is a nonlinearly coupled system of quasilinear PDEs, the above scheme only requires solving two completely decoupled linear systems per time-step. Note that, in contrast to the scheme presented in \cite{AbeHrkPagPraRugSue14} for the SDLLG equation which solves for $\bff{m}_h^n$ and then $\bff{s}_h^n$ sequentially, Algorithm~\ref{alg:scheme} for the SDLLB equation allows $\bff{m}_h^n$ and $\bff{s}_h^n$ to be computed \emph{in parallel}, owing to the explicit treatment of $\bff{m}_h^{n-1}$ in \eqref{equ:scheme s}. 
This decoupling provides a significant computational advantage, particularly in large-scale simulations or parallel computing environments.

To derive an error estimate in this section, for simplicity we {take $(\bff{m}_h^0,\bff{s}_h^0)= (P_h \bff{m}_0, P_h\bff{s}_0)$, where $P_h$ is the $\bb{L}^2$ projection operator defined in \eqref{equ:L2 proj}.} Note that other initial inputs are possible as long as they approximate the initial data with sufficient accuracy. We further assume adequate regularity for the exact solution $(\bff{m},\bff{s})$ to \eqref{equ:sdllb} to derive an optimal order of convergence, namely
\begin{equation}\label{equ:assum}
		\bff{m},\bff{s} \in L^\infty_T(\bb{H}^3) \cap W^{1,\infty}_T(\bb{H}^{r+1}) \cap W^{2,\infty}_T(\bb{L}^2),
\end{equation}
where $r$ is the degree of polynomials used in $\bb{V}_h$.

We now define several bilinear forms and discuss their properties to facilitate subsequent analysis.

\begin{definition}\label{def:bilinear}
Given $\bff{\phi},\bff{\psi}\in \bb{H}^1\cap \bb{L}^\infty$ and $\beta>0$, together with $D_0$ and $\bff{j}$ satisfying the assumptions stated in Section~\ref{sec:assum}, we define bilinear forms on $\bb{H}^1\times \bb{H}^1$:
\begin{align*}
	\mathcal{D}_1(\bff{v},\bff{w}) &:= \inpro{\nabla \bff{v}}{\nabla \bff{w}} + \inpro{\bff{v}}{\bff{w}},
	\\
	\mathcal{C}_1(\bff{\phi};\bff{v},\bff{w}) &:= -\inpro{\bff{\phi}\times \nabla \bff{v}}{\nabla \bff{w}},
	\\
	\mathcal{B}_1(\bff{\phi},\bff{\psi}; \bff{v},\bff{w}) &:= \inpro{(\bff{\phi}\cdot \bff{\psi})\bff{v}}{\bff{w}},
	\\
	\mathcal{L}_1(\bff{\phi};\bff{v},\bff{w}) &:= -\inpro{\bff{v}\times \bff{\phi}}{\bff{w}},
	\\
	\mathcal{D}_2(\bff{v},\bff{w}) &:= \inpro{D_0 \nabla \bff{v}}{\nabla \bff{w}} + \inpro{D_0 \bff{v}}{\bff{w}},
	\\
	\mathcal{C}_2(\bff{\phi};\bff{v},\bff{w}) &:= \inpro{D_0\bff{v} \times\bff{\phi}}{\bff{w}},
	\\
	\mathcal{B}_2(\bff{\phi},\bff{\psi}; \bff{v},\bff{w}) &:= -\beta \inpro{D_0 (\bff{\phi}\otimes \bff{\psi}) \nabla \bff{v}}{\nabla\bff{w}},
	\\
	\mathcal{L}_2(\bff{\phi};\bff{v},\bff{w}) &:= {-\inpro{\bff{v}\otimes \bff{\phi}}{\nabla\bff{w}}.}
\end{align*}
Furthermore, let
\begin{align}\label{equ:bilinear A}
	\mathcal{A}_i(\bff{\phi};\bff{v},\bff{w}) := \mathcal{D}_i(\bff{v},\bff{w})
	+
	\mathcal{C}_i(\bff{\phi}; \bff{v},\bff{w})
	+
	\mathcal{B}_i(\bff{\phi},\bff{\phi};\bff{v},\bff{w}),
	\text{ for } i=1,2.
\end{align}
\end{definition}

With the above notations, the weak formulations \eqref{equ:weak m} and \eqref{equ:weak s} for the continuous problem now read as:
\begin{align}
	\label{equ:weak m cont}
	\inpro{\partial_t \bff{m}}{\bff{\phi}} 
	+
	\mathcal{A}_1(\bff{m};\bff{m},\bff{\phi})
	+
	\mathcal{L}_1(\bff{s};\bff{m},\bff{\phi})
	&= 0,
	\quad \forall \bff{\phi}\in \bb{H}^1,
	\\
	\label{equ:weak s cont}
	\inpro{\partial_t \bff{s}}{\bff{\psi}} 
	+
	\mathcal{A}_2(\bff{m};\bff{s},\bff{\psi})
	+
	\mathcal{L}_2(\bff{j};\bff{m},\bff{\psi})
	&=0,
	\quad \forall \bff{\psi}\in \bb{H}^1.
\end{align}
The weak formulations \eqref{equ:scheme m} and \eqref{equ:scheme s} for the discrete scheme can be written as:
\begin{align}
	\label{equ:weak m disc}
	\inpro{\dtt \bff{m}_h^n}{\bff{\phi}_h}
	+
	\mathcal{A}_1(\bff{m}_h^{n-1};\bff{m}_h^n, \bff{\phi}_h)
	+
	\mathcal{L}_1(\bff{s}_h^{n-1}; \bff{m}_h^n, \bff{\phi}_h)
	&= 0,
	\quad \forall \bff{\phi}_h\in \bb{V}_h,
	\\
	\label{equ:weak s disc}
	\inpro{\dtt \bff{s}_h^n}{\bff{\psi}_h}
	+
	\mathcal{A}_2(\bff{m}_h^{n-1};\bff{s}_h^n, \bff{\psi}_h)
	+
	\mathcal{L}_2(\bff{j}^n; \bff{m}_h^{n-1}, \bff{\psi}_h)
	&=0,
	\quad \forall \bff{\psi}_h\in \bb{V}_h.
\end{align}

We collect some important estimates on the previously defined bilinear forms in the following lemma.

\begin{lemma}\label{lem:coer bdd A}
The following statements hold true for the bilinear forms defined in Definition~\ref{def:bilinear}:
\begin{enumerate}[(i)]
	\item $\mathcal{A}_1(\bff{\phi}; \bff{\cdot},\bff{\cdot})$ and $\mathcal{A}_2(\bff{\phi}; \bff{\cdot},\bff{\cdot})$ are bounded, i.e. there exist constants $\beta_1,\beta_2>0$ depending only on the coefficients of the problem \eqref{equ:sdllb}, $\norm{\bff{\phi}}{\bb{L}^\infty}$, and $\norm{\bff{\phi}}{\bb{H}^1}$, such that
	\begin{align}
		\label{equ:A1 bdd}
		\abs{\mathcal{A}_1(\bff{\phi};\bff{v},\bff{w})} 
		&\leq
		\beta_1 \norm{\bff{v}}{\bb{H}^1} \norm{\bff{w}}{\bb{H}^1},
		\quad \forall \bff{v},\bff{w}\in \bb{H}^1,
		\\
		\label{equ:A2 bdd}
		\abs{\mathcal{A}_2(\bff{\phi};\bff{v},\bff{w})} 
		&\leq
		\beta_2 \norm{\bff{v}}{\bb{H}^1} \norm{\bff{w}}{\bb{H}^1},
		\quad \forall \bff{v},\bff{w}\in \bb{H}^1;
	\end{align}
	\item $\mathcal{A}_1(\bff{\phi}; \bff{\cdot},\bff{\cdot})$ is coercive, i.e. there exist a constant $\mu_1>0$ \emph{independent} of $\bff{\phi}$, such that
	\begin{align}\label{equ:A1 coercive}
		\abs{\mathcal{A}_1(\bff{\phi};\bff{v},\bff{v})} 
		&\geq
		\mu_1 \norm{\bff{v}}{\bb{H}^1}^2,
		\quad \forall \bff{v}\in \bb{H}^1;
	\end{align}
	\item If $\norm{\bff{\phi}}{\bb{L}^\infty}^2 \leq D_\ast/(2\beta D^\ast)$, then $\mathcal{A}_2(\bff{\phi}; \bff{\cdot},\bff{\cdot})$ is also coercive, i.e. there exist a constant $\mu_2>0$ \emph{independent} of $\bff{\phi}$, such that
	\begin{align}\label{equ:A2 coercive}
		\abs{\mathcal{A}_2(\bff{\phi};\bff{v},\bff{v})} 
		&\geq
		\mu_2 \norm{\bff{v}}{\bb{H}^1}^2,
		\quad \forall \bff{v}\in \bb{H}^1;
	\end{align}
	\item There exists a constant $C>0$ such that {for any $\bff{\phi}\in \bb{W}^{1,4}$, $\bff{v}\in \bb{H}^1$, and $\bff{w}\in \bb{H}^2_{\bff{n}}$},
	\begin{align}\label{equ:C ineq W14}
		\left|\mathcal{C}_1 (\bff{\phi}; \bff{v},\bff{w}) \right|
		&\leq
		C \norm{\bff{\phi}}{\bb{W}^{1,4}} \norm{\bff{v}}{\bb{L}^2} \norm{\bff{w}}{\bb{H}^2},
		\\
		\label{equ:C Delta w w}
		\left|\mathcal{C}_1 (\bff{\phi}; \Delta \bff{w},\bff{w}) \right|
		&\leq
		C \norm{\bff{\phi}}{\bb{W}^{1,4}} \norm{\bff{w}}{\bb{W}^{1,4}} \norm{\Delta \bff{w}}{\bb{L}^2};
	\end{align}
	\item There exists a constant $C>0$ such that {for any $D_0\in W^{1,4}(\Omega)$, $\bff{\phi},\bff{\psi} \in \bb{W}^{1,4}$, $\bff{v}\in \bb{H}^1$, and $\bff{w}\in \bb{H}^2_{\bff{n}}$},
	\begin{align}\label{equ:C2 ineq W14}
		\left|\mathcal{B}_2 (\bff{\phi}, \bff{\psi}; \bff{v},\bff{w}) \right|
		&\leq
		C \left( \norm{\bff{\phi}}{\bb{W}^{1,4}} \norm{\bff{\psi}}{\bb{L}^\infty} + \norm{\bff{\phi}}{\bb{L}^\infty} \norm{\bff{\psi}}{\bb{W}^{1,4}}  \right)  \norm{\bff{v}}{\bb{L}^2} \norm{\bff{w}}{\bb{H}^2};
	\end{align}
	\item {For any $\bff{\phi}\in \bb{L}^\infty$, $\bff{v}\in \bb{L}^2$,  and $\bff{w}\in \bb{H}^1$,}
	\begin{align}\label{equ: H1 H1 ineq}
		\abs{\mathcal{L}_2(\bff{\phi};\bff{v},\bff{w})}
		&\leq
		\norm{\bff{\phi}}{\bb{L}^\infty} \norm{\bff{v}}{\bb{L}^2} \norm{\nabla\bff{w}}{\bb{L}^2}.
	\end{align}
\end{enumerate}
\end{lemma}

\begin{proof}
Inequalities \eqref{equ:A1 bdd} and \eqref{equ:A2 bdd} follow immediately by H\"older's inequality, while inequality \eqref{equ:A1 coercive} is obvious. Finally, we have
\begin{align*}
	\mathcal{A}_2(\bff{\phi};\bff{v},\bff{v})
	&=
	\inpro{D_0 \nabla \bff{v}}{\nabla \bff{v}}
	+
	\inpro{D_0 \bff{v}}{\bff{v}}
	-
	\beta \inpro{D_0 (\bff{\phi}\otimes \bff{\phi}) \nabla \bff{v}}{\nabla\bff{v}}
	\\
	&\geq
	D_\ast \norm{\bff{v}}{\bb{L}^2}^2
	+
	\big(D_\ast -\beta D^\ast \norm{\bff{\phi}}{\bb{L}^\infty}^2 \big) \norm{\nabla \bff{v}}{\bb{L}^2}^2
	\geq
	\mu_2 \norm{\bff{v}}{\bb{H}^1}^2,
\end{align*}
where $\mu_2:= D_\ast/2$, thus showing~\eqref{equ:A2 coercive}. 
Next, we show~\eqref{equ:C ineq W14} and~\eqref{equ:C Delta w w}. Integrating by parts, we infer that
\begin{align*}
	\mathcal{C}_1 (\bff{\phi}; \bff{v},\bff{w})
	&=
	\inpro{\bff{\phi}\times \bff{v}}{\Delta \bff{w}}
	+
	\inpro{\nabla \bff{\phi}\times \bff{v}}{\nabla \bff{w}}.
\end{align*}
By H\"older's inequality, we then have
\begin{align*}
	\abs{\mathcal{C}_1 (\bff{\phi}; \bff{v},\bff{w})}
	&\leq
	\norm{\bff{\phi}}{\bb{L}^\infty} \norm{\bff{v}}{\bb{L}^2} \norm{\Delta \bff{w}}{\bb{L}^2}
	+
	\norm{\nabla\bff{\phi}}{\bb{L}^4} \norm{\bff{v}}{\bb{L}^2} \norm{\nabla\bff{w}}{\bb{L}^4},
\end{align*}
from which~\eqref{equ:C ineq W14} follows by the embeddings $\bb{H}^2\hookrightarrow \bb{W}^{1,4} \hookrightarrow \bb{L}^\infty$.
Inequality~\eqref{equ:C Delta w w} follows by similar argument.
Next, using integration by parts, we also have
\begin{align*}
	\mathcal{B}_2(\bff{\phi}, \bff{\psi}; \bff{v},\bff{w})
	=
	\beta \inpro{D_0 (\bff{\phi}\otimes \bff{\psi})\bff{v}}{\Delta \bff{w}}
	+
	\beta \inpro{\bff{v} \nabla \big(D_0(\bff{\phi}\otimes \bff{\psi})\big)^\top}{\nabla \bff{w}}.
\end{align*}
Therefore, by H\"older's inequality,
\begin{align*}
	\abs{\mathcal{B}_2(\bff{\phi}, \bff{\psi}; \bff{v},\bff{w})}
	&\leq
	\beta D^\ast \norm{\bff{\phi}}{\bb{L}^\infty} \norm{\bff{\psi}}{\bb{L}^\infty} \norm{\bff{v}}{\bb{L}^2} \norm{\Delta \bff{w}}{\bb{L}^2}
	\\
	&\quad
	+
	\beta \norm{\bff{v}}{\bb{L}^2} \norm{D_0(\bff{\phi}\otimes \bff{\psi})}{\bb{W}^{1,4}} \norm{\nabla \bff{w}}{\bb{L}^4},
\end{align*}
which implies~\eqref{equ:C2 ineq W14} by the assumptions and the embeddings $\bb{H}^2\hookrightarrow \bb{W}^{1,4} \hookrightarrow \bb{L}^\infty$.
Finally, {by H\"older's inequality we have
\begin{align*}
	\abs{\mathcal{L}_2(\bff{\phi};\bff{v},\bff{w})}
	&\leq
	\norm{\bff{v}\otimes \bff{\phi}}{\bb{L}^2} \norm{\nabla \bff{w}}{\bb{L}^2}
	\leq
	\norm{\bff{\phi}}{\bb{L}^\infty} \norm{\bff{v}}{\bb{L}^2} \norm{\nabla\bff{w}}{\bb{L}^2},
\end{align*}
showing \eqref{equ: H1 H1 ineq},} thus completing the proof of the lemma.
\end{proof}

To aid in the analysis, we next introduce the following elliptic projection operators, a technique originating in~\cite{Whe73}.

\begin{definition}\label{def:elliptic proj}
Let $t\in [0,T]$. Let $(\bff{m},\bff{s})$ be a strong solution to \eqref{equ:sdllb} in the sense of Definition~\ref{def:strong}, with initial data $(\bff{m}_0,\bff{s}_0)$ satisfying
\begin{equation}\label{equ:small init ellip}
	\norm{\bff{m}_0}{\bb{L}^\infty}^2 \leq D_\ast/(2\beta D^\ast).
\end{equation}
{The $\mathcal{A}_1$-elliptic projection operator $\Pi_h :=\Pi_{h,
	\bff{m}(t)} : \bb{H}^1 \to \bb{V}_h$ is defined by
\begin{align}\label{equ:A1 proj}
	\bff{v}\mapsto \Pi_h\bff{v}
	\quad\text{satisfying}\quad 
	\mathcal{A}_1(\bff{m}(t); \bff{v}- \Pi_h \bff{v}, \bff{\phi}_h)=0, \quad \forall \bff{\phi}_h\in \bb{V}_h.
\end{align}
The $\mathcal{A}_2$-elliptic projection operator $\Lambda_h :=\Lambda_{h,
	\bff{m}(t)} : \bb{H}^1 \to \bb{V}_h$ is defined by
\begin{align}\label{equ:A2 proj}
	\bff{v}\mapsto \Lambda_h\bff{v}
	\quad\text{satisfying}\quad 
	\mathcal{A}_2(\bff{m}(t); \bff{v}- \Lambda_h \bff{v}, \bff{\psi}_h)=0, \quad \forall \bff{\psi}_h\in \bb{V}_h.
\end{align}
}
\end{definition}

Note that the bilinear forms defining {$\Pi_h\bff{v}$ and $\Lambda_h \bff{v}$} in \eqref{equ:A1 proj} and \eqref{equ:A2 proj} depend on the given exact solution $(\bff{m},\bff{s})$ of the problem~\eqref{equ:sdllb}. In light of the regularity assumption on $(\bff{m},\bff{s})$, estimate \eqref{equ:m less m0} and the assumption \eqref{equ:small init ellip}, as well as the coercivity and boundedness properties \eqref{equ:A1 bdd}, \eqref{equ:A2 bdd}, \eqref{equ:A1 coercive}, and \eqref{equ:A2 coercive}, the elliptic projections are well-defined by the {Lax--Milgram theorem}.

{Next, for any $\bff{v}\in L_T^\infty(\bb{H}^1)$}, let $\bff{\rho}$ and $\bff{\eta}$ be defined by
\begin{align}\label{equ:rho eta def}
	\bff{\rho}(t) := \Pi_h \bff{v}(t)- \bff{v}(t)
	\quad \text{and} \quad
	\bff{\eta}(t) := \Lambda_h \bff{v}(t)- \bff{v}(t).
\end{align}
{Then}
\begin{align}\label{equ:elliptic proj}
	\mathcal{A}_1(\bff{m}(t); \bff{\rho}(t), \bff{\phi}_h)
	=
	\mathcal{A}_2(\bff{m}(t); \bff{\eta}(t), \bff{\psi}_h)
	=
	0, \quad \forall \bff{\phi}_h, \bff{\psi}_h\in \bb{V}_h.
\end{align}
Some estimates on $\bff{\rho}$ and $\bff{\eta}$ are derived below.

\begin{proposition}\label{pro:est rho}
	Let $(\bff{m},\bff{s})$ be the solution to~\eqref{equ:sdllb} with regularity given by \eqref{equ:assum}, and initial data satisfying \eqref{equ:small init ellip}.
	Let $\bff{\rho}$ and $\bff{\eta}$ be as defined in \eqref{equ:rho eta def}, {and suppose that $\bff{v}\in L^\infty_T(\bb{H}^{r+1})$.} Then there exists a constant $C>0$ such that for any $t\in [0,T]$,
	\begin{align}
		\label{equ:rho t est}
		\norm{\bff{\rho}(t)}{\bb{L}^2}
		+
		h \norm{\nabla \bff{\rho}(t)}{\bb{L}^2}
		&\leq
		{C \norm{\bff{v}}{L^\infty_T(\bb{H}^{r+1})} h^{r+1},}
		\\
		\label{equ:eta t est}
		\norm{\bff{\eta}(t)}{\bb{L}^2}
		+
		h \norm{\nabla \bff{\eta}(t)}{\bb{L}^2}
		&\leq
		{C \norm{\bff{v}}{L^\infty_T(\bb{H}^{r+1})} h^{r+1},}
	\end{align}
	where the constant $C$ is independent of $h$, but may depend on $T$ {and $\norm{\bff{m}}{L^\infty_T(\bb{H}^2)}$}.
	
	Furthermore, suppose that $\bff{v}\in L^\infty_T(\bb{H}^{r+1}) \cap H^1_T(\bb{H}^2)$, then there exists a constant $C>0$ such that for any $t\in [0,T]$,
	\begin{align}
		\label{equ:dt rho t est}
		\norm{\partial_t \bff{\rho}(t)}{\bb{L}^2}
		+
		h \norm{\nabla \partial_t \bff{\rho}(t)}{\bb{L}^2}
		&\leq
		{C\big(1+\norm{\bff{v}}{L^\infty_T(\bb{H}^{r+1})}\big) h^{r+1},}
		\\
		\label{equ:dt eta t est}
		\norm{\partial_t \bff{\eta}(t)}{\bb{L}^2}
		+
		h \norm{\nabla \partial_t \bff{\eta}(t)}{\bb{L}^2}
		&\leq
		{C\big(1+\norm{\bff{v}}{L^\infty_T(\bb{H}^{r+1})}\big) h^{r+1},}
	\end{align}
	where the constant $C$ is independent of $h$, but may depend on $T$ {and $\norm{\bff{m}}{W^{1,\infty}_T(\bb{H}^{r+1})}$}.
\end{proposition}

\begin{proof}
	Note that by~\eqref{equ:u Lp}, we have~$\norm{\bff{m}}{L^\infty_T(\bb{L}^\infty)}^2 \leq D_\ast/(2\beta D^\ast)$.
	First, we prove~\eqref{equ:rho t est}.
	For all $\bff{\chi}\in \bb{V}_h$, by the coercivity and the boundedness of $\mathcal{A}_1$ in Lemma~\ref{lem:coer bdd A}, and the definition of $\bff{\rho}$,
	\begin{align*}
		\mu_1 \norm{\bff{\rho}(t)}{\bb{H}^1}^2
		&\leq
		\mathcal{A}_1(\bff{m}(t); \bff{\rho}(t), \bff{\rho}(t))
		\\
		&=
		\mathcal{A}_1(\bff{m}(t); \bff{\rho}(t), \Pi_h\bff{v}(t)- \bff{\chi})
		-
		\mathcal{A}_1(\bff{m}(t); \bff{\rho}(t), \bff{v}(t)-\bff{\chi})
		\\
		&\leq
		\beta_1 \norm{\bff{\rho}(t)}{\bb{H}^1} \norm{\bff{v}(t)-\bff{\chi}}{\bb{H}^1},
	\end{align*}
	since the first term in the second step is zero. Here, $\beta_1$ depends on $\norm{\bff{m}}{L^\infty_T(\bb{H}^2)}$. Therefore, by~\eqref{equ:fin approx},
	\begin{align}\label{equ:eta H1}
		\norm{\bff{\rho}(t)}{\bb{H}^1}
		\leq
		(\beta_1/\mu_1) \inf_{\bff{\chi}\in \bb{V}_h} \norm{\bff{v}(t)-\bff{\chi}}{\bb{H}^1}
		\leq
		Ch^r \norm{\bff{v}}{L^\infty_T(\bb{H}^{r+1})},
	\end{align}
	where $C$ depends on $\norm{\bff{m}}{L^\infty_T(\bb{H}^2)}$.
	To show the $\bb{L}^2$-estimate, we use a duality argument. For each $\bff{m}(t)\in \bb{H}^2$, let $\bff{\psi}(t)\in \bb{H}^2_{\bff{n}}$ satisfy
	\begin{align}\label{equ:A u zeta}
		\mathcal{A}_1(\bff{m}(t); \bff{\zeta}, \bff{\psi}(t)) = \inpro{\bff{\rho}(t)}{\bff{\zeta}},
		\quad \forall \bff{\zeta}\in \bb{H}^1.
	\end{align}
	For any $t\in [0,T]$, such $\bff{\psi}(t)$ exists by Lemma~\ref{lem:dual exist} under the assumed conditions on $\Omega$. Moreover,
	\begin{align}\label{equ:psi rho}
		\norm{\bff{\psi}(t)}{\bb{H}^2} \leq
		C\norm{\bff{\rho}(t)}{\bb{L}^2},
	\end{align}
	where $C$ depends on $\norm{\bff{m}}{L^\infty_T(\bb{H}^2)}$.
	Therefore, taking $\bff{\zeta}= \bff{\rho}(t)$ in~\eqref{equ:A u zeta} and noting~\eqref{equ:elliptic proj}, we have for all $\bff{\chi}\in \bb{V}_h$,
	\begin{align*}
		\norm{\bff{\rho}(t)}{\bb{L}^2}^2
		&=
		\mathcal{A}_1(\bff{m}(t); \bff{\rho}(t), \bff{\psi}(t))
		=
		\mathcal{A}_1(\bff{m}(t); \bff{\rho}(t), \bff{\psi}(t)-\bff{\chi})
		\\
		&\leq
		\beta_1 \norm{\bff{\rho}(t)}{\bb{H}^1} \inf_{\bff{\chi}\in \bb{V}_h} \norm{\bff{\psi}(t)-\bff{\chi}}{\bb{H}^1}
		\leq
		{Ch^{r+1} \norm{\bff{v}}{L^\infty_T(\bb{H}^{r+1})} \norm{\bff{\rho}(t)}{\bb{L}^2},}
	\end{align*}
	{where $C$ depends on $\norm{\bff{m}}{L^\infty_T(\bb{H}^2)}$}, and in the last step we used~\eqref{equ:eta H1}, \eqref{equ:fin approx}, and~\eqref{equ:psi rho}. {This then implies~\eqref{equ:rho t est}.} The proof of~\eqref{equ:eta t est} follows in a similar manner.
	
	Next, we show~\eqref{equ:dt rho t est} for $\bff{v}\in
	L^\infty_T(\bb{H}^{r+1}) \cap H^1_T(\bb{H}^2)$. For ease of
	presentation, we will omit the dependence of the functions on $t$.
	{By using successively the coercivity of $\mathcal{A}_1$, the
		definition of $\bff{\rho}$ in \eqref{equ:rho eta def}, and the
		fact that $P_h \partial_t \Pi_h \bff{v}= \partial_t \Pi_h
		\bff{v}$, with $P_h$ being the $\bb{L}^2$~projection operator
	onto $\bb{V}_h$, we obtain}
	\begin{align}\label{equ:leq A dt eta}
		\mu_1 \norm{\partial_t \bff{\rho}}{\bb{H}^1}^2
		\leq
		\mathcal{A}_1(\bff{m}; \partial_t \bff{\rho}, \partial_t \bff{\rho})
		&=
		\mathcal{A}_1\left(\bff{m}; \partial_t \bff{\rho}, P_h(\partial_t \bff{\rho})\right) 
		+
		\mathcal{A}_1 \left(\bff{m}; \partial_t \bff{\rho}, \partial_t \bff{\rho}- P_h(\partial_t \bff{\rho})\right)
		\nonumber\\
		&=
		{\mathcal{A}_1\left(\bff{m}; \partial_t \bff{\rho}, P_h(\partial_t \bff{\rho})\right) 
		+
		{\mathcal{A}_1 \left(\bff{m}; \partial_t \bff{\rho}, \partial_t \Pi_h \bff{v}(t)- P_h \partial_t \Pi_h \bff{v}(t) \right)}
		}
		\nonumber\\
		&\quad
		+
		{\mathcal{A}_1 \left(\bff{m}; \partial_t \bff{\rho},
		P_h(\partial_t \bff{v})-\partial_t \bff{v}\right)}
		\nonumber\\
		&=
		{\mathcal{A}_1\left(\bff{m}; \partial_t \bff{\rho}, P_h(\partial_t \bff{\rho})\right) 
		+
		\mathcal{A}_1\left(\bff{m}; \partial_t \bff{\rho}, P_h(\partial_t \bff{v}) -\partial_t \bff{v}\right).}
	\end{align}
	We will estimate each term on the last line of \eqref{equ:leq A dt eta}. To this end, noting~\eqref{equ:bilinear A} and differentiating~\eqref{equ:elliptic proj} with respect to $t$, we have for all $\bff{\phi}_h \in \bb{V}_h$,
	\begin{align}\label{equ:A dt eta}
		\mathcal{A}_1 (\bff{m}; \partial_t \bff{\rho},\bff{\phi}_h)
		+
		2 \mathcal{B}_1(\bff{m},\partial_t \bff{m}; \bff{\rho}, \bff{\phi}_h)
		+
		\mathcal{C}_1(\partial_t \bff{m}; \bff{\rho}, \bff{\phi}_h)
		= 0.
	\end{align}
	Thus, for the first term on the right-hand side of~\eqref{equ:leq A dt eta}, by the boundedness of $\mathcal{B}_1, \mathcal{C}_1$, and $P_h$ we obtain
	\begin{align}\label{equ:A dt leq 1}
		\big| \mathcal{A}_1(\bff{m}; \partial_t \bff{\rho}, P_h(\partial_t \bff{\rho})) \big| 
		&\leq
		\big| 2 \mathcal{B}_1(\bff{m},\partial_t \bff{m}; \bff{\rho}, P_h(\partial_t \bff{\rho})) \big|
		+
		\big| \mathcal{C}_1(\partial_t \bff{m}; \bff{\rho}, P_h(\partial_t \bff{\rho})) \big|
		\nonumber\\
		&\leq
		C \norm{\bff{\rho}}{\bb{H}^1} \norm{P_h (\partial_t \bff{\rho})}{\bb{H}^1}
		\leq
		Ch^r \norm{\bff{v}}{L^\infty_T(\bb{H}^{r+1})} \norm{\partial_t \bff{\rho}}{\bb{H}^1},
	\end{align}
	{where $C$ depends on $\norm{\bff{m}}{W^{1,\infty}_T(\bb{H}^2)}$},
	and in the last step we also used~\eqref{equ:rho t est}
	and~\eqref{equ:Ph H1 stab}. For the second term on the right-hand side
	of~\eqref{equ:leq A dt eta}, by the boundedness of $\mathcal{A}_1$ and \eqref{equ:Ph approx} we have
	\begin{align}\label{equ:A dt leq 2}
		\big| \mathcal{A}_1(\bff{m}; \partial_t \bff{\rho}, P_h(\partial_t \bff{v})- \partial_t \bff{v}) \big| 
		\leq
		C \norm{\partial_t \bff{\rho}}{\bb{H}^1} \norm{P_h(\partial_t \bff{v})- \partial_t \bff{v}}{\bb{H}^1}
		\leq
		Ch^r \norm{\partial_t \bff{\rho}}{\bb{H}^1},
	\end{align}
	{where $C$ depends on $\norm{\bff{m}}{W^{1,\infty}_T(\bb{H}^{r+1})}$.}
	The estimates~\eqref{equ:A dt leq 1} and \eqref{equ:A dt leq 2}, together with~\eqref{equ:leq A dt eta} imply
	\begin{align}\label{equ:dt eta H1}
		\norm{\partial_t \bff{\rho}(t)}{\bb{H}^1}
		\leq
		C\big(1+\norm{\bff{v}}{L^\infty_T(\bb{H}^{r+1})}\big) h^r.
	\end{align}
	To estimate $\norm{\partial_t \bff{\rho}}{\bb{L}^2}$, we use duality argument as before. For each $\bff{m}(t) \in \bb{H}^2$, let $\bff{\psi}(t)\in \bb{H}^2_{\bff{n}}$ satisfy
	\begin{align}\label{equ:A pa t zeta}
		\mathcal{A}_1(\bff{m}(t); \bff{\zeta}, \bff{\psi}(t)) = \inpro{\partial_t \bff{\rho}(t)}{\bff{\zeta}}, \quad \forall \bff{\zeta}\in \bb{H}^1. 
	\end{align}
	The existence of $\bff{\psi}(t)$ is given by Lemma~\ref{lem:dual exist} (to be proven after the conclusion of this proof), and furthermore we have
	\begin{align}\label{equ:psi dt eta}
		\norm{\bff{\psi}(t)}{\bb{H}^2} \leq C \norm{\partial_t \bff{\rho}(t)}{\bb{L}^2},
	\end{align}
	where $C$ depends on $\norm{\bff{m}}{L^\infty_T(\bb{H}^2)}$.
	Taking $\bff{\zeta}= \partial_t \bff{\rho}(t)$ in~\eqref{equ:A pa t zeta}, we have
	\begin{align*}
		\mathcal{A}_1(\bff{m}(t); \partial_t \bff{\rho}(t), \bff{\psi}(t))
		=
		\norm{\partial_t \bff{\rho}(t)}{\bb{L}^2}^2.
	\end{align*}
	This equation and~\eqref{equ:A dt eta} yield for all $\bff{\phi}_h\in \bb{V}_h$,
	\begin{align}\label{equ:dt rho abc}
		\norm{\partial_t \bff{\rho}}{\bb{L}^2}^2
		&=
		\mathcal{A}_1 (\bff{m}; \partial_t \bff{\rho},\bff{\psi}-\bff{\phi}_h)
		-
		2 \mathcal{B}_1(\bff{m},\partial_t \bff{m}; \bff{\rho}, \bff{\phi}_h)
		-
		\mathcal{C}_1(\partial_t \bff{m}; \bff{\rho}, \bff{\phi}_h)
		\nonumber\\
		&=
		\mathcal{A}_1 (\bff{m}; \partial_t \bff{\rho},\bff{\psi}-\bff{\phi}_h)
		+
		2\mathcal{B}_1 (\bff{m},\partial_t \bff{m}; \bff{\rho}, \bff{\psi}-\bff{\phi}_h)
		-
		2 \mathcal{B}_1(\bff{m},\partial_t \bff{m}; \bff{\rho}, \bff{\psi})
		\nonumber\\
		&\quad
		+
		\mathcal{C}_1(\partial_t \bff{m}; \bff{\rho},\bff{\psi}-\bff{\phi}_h)
		-
		\mathcal{C}_1(\partial_t \bff{m}; \bff{\rho}, \bff{\psi}).
	\end{align}
	It remains to bound each term in the last expression using \eqref{equ:A1 bdd} and H\"older inequality, resulting in
	\begin{align*}
		\norm{\partial_t \bff{\rho}}{\bb{L}^2}^2
		&\leq
		C \left(1+\norm{\bff{m}}{\bb{H}^2}^2\right) \norm{\partial_t \bff{\rho}}{\bb{H}^1} \norm{\bff{\psi}-\bff{\phi}_h}{\bb{H}^1}
		+
		2 \norm{\bff{m}}{\bb{L}^\infty} \norm{\partial_t \bff{m}}{\bb{L}^\infty} \norm{\bff{\rho}}{\bb{L}^2} \norm{\bff{\psi}-\bff{\phi}_h}{\bb{L}^2}
		\\
		&\quad
		+
		2 \norm{\bff{m}}{\bb{L}^\infty} \norm{\partial_t \bff{m}}{\bb{L}^\infty} \norm{\bff{\rho}}{\bb{L}^2} \norm{\bff{\psi}}{\bb{L}^2}
		+
		C \norm{\partial_t \bff{m}}{\bb{L}^\infty} \norm{\bff{\rho}}{\bb{H}^1} \norm{\bff{\psi}-\bff{\phi}_h}{\bb{H}^1}
		\\
		&\quad
		+
		C \norm{\partial_t \bff{m}}{\bb{W}^{1,4}} \norm{\bff{\rho}}{\bb{L}^2} \norm{\bff{\psi}}{\bb{H}^2},
	\end{align*}
	where in the last step we also used the Sobolev embedding and~\eqref{equ:C ineq W14}.
	We now choose $\bff{\phi}_h= P_h \bff{\psi}$. Successively using~\eqref{equ:dt eta H1}, \eqref{equ:Ph approx}, and~\eqref{equ:rho t est}, we have
	\begin{align}\label{equ:dt rho h L2}
		\norm{\partial_t \bff{\rho}(t)}{\bb{L}^2}^2
		&\leq
		C \big(1+\norm{\bff{v}}{L^\infty_T(\bb{H}^{r+1})}\big) (h^{r+1}+h^{r+3}) \norm{\bff{\psi}(t)}{\bb{H}^2}
		\nonumber\\
		&\leq
		C\big(1+\norm{\bff{v}}{L^\infty_T(\bb{H}^{r+1})}\big) h^{r+1} \norm{\partial_t \bff{\rho}(t)}{\bb{L}^2},
	\end{align}
	where in the last step we used~\eqref{equ:psi dt eta}, {and $C$ depends on $\norm{\bff{m}}{W^{1,\infty}_T(\bb{H}^{r+1})}$}. This implies
	\begin{align*}
		\norm{\partial_t \bff{\rho}(t)}{\bb{L}^2}
		\leq
		{C \big(1+\norm{\bff{v}}{L^\infty_T(\bb{H}^{r+1})}\big) h^{r+1}.}
	\end{align*}
	This inequality and~\eqref{equ:dt eta H1} imply~\eqref{equ:dt rho t est}. 
	
	Finally, we prove~\eqref{equ:dt eta t est}. By the coercivity of $\mathcal{A}_2$ in Lemma~\ref{lem:coer bdd A} and the definition of $\bff{\eta}$, in the same manner as \eqref{equ:leq A dt eta} we have
	\begin{align*}
		\mu_2 \norm{\partial_t \bff{\eta}}{\bb{H}^1}^2
		\leq
		\mathcal{A}_2(\bff{m}; \partial_t \bff{\eta}, \partial_t \bff{\eta})
		=
		\mathcal{A}_2\left(\bff{m}; \partial_t \bff{\eta}, P_h(\partial_t \bff{\eta})\right) 
		+
		\mathcal{A}_2\left(\bff{m}; \partial_t \bff{\eta}, P_h(\partial_t \bff{v}) -\partial_t \bff{v}\right).
	\end{align*}
	By the same argument leading to \eqref{equ:A dt eta} and \eqref{equ:dt eta H1}, we have for all $\bff{\psi}_h\in \bb{V}_h$,
	\begin{align}\label{equ:A2 dt eta}
		\mathcal{A}_2 (\bff{m}; \partial_t \bff{\eta},\bff{\psi}_h)
		+
		2 \mathcal{B}_2(\bff{m},\partial_t \bff{m}; \bff{\eta}, \bff{\psi}_h)
		+
		\mathcal{C}_2(\partial_t \bff{m}; \bff{\eta}, \bff{\psi}_h)
		= 0.
	\end{align}
	and infer that
	\begin{align}\label{equ:dt2 eta H1}
		\norm{\partial_t \bff{\eta}(t)}{\bb{H}^1}
		\leq
		C\big(1+\norm{\bff{v}}{L^\infty_T(\bb{H}^{r+1})}\big) h^r,
	\end{align}
	{where $C$ depends on $\norm{\bff{m}}{W^{1,\infty}_T(\bb{H}^{r+1})}$.}
	To bound $\norm{\partial_t \bff{\eta}}{\bb{L}^2}$, we again use duality. For each $\bff{m}(t) \in \bb{H}^2$, let $\widetilde{\bff{\psi}}(t)\in \bb{H}^2_{\bff{n}}$ satisfy
	\begin{align}\label{equ:A2 pa t zeta}
		\mathcal{A}_2(\bff{m}(t); \bff{\zeta}, \widetilde{\bff{\psi}}(t)) = \inpro{\partial_t \bff{\eta}(t)}{\bff{\zeta}}, \quad \forall \bff{\zeta}\in \bb{H}^1. 
	\end{align}
	The existence of $\widetilde{\bff{\psi}}(t)$ is given by Lemma~\ref{lem:dual exist}, and furthermore we have
	\begin{align}\label{equ:psi2 dt eta}
		\norm{\widetilde{\bff{\psi}}(t)}{\bb{H}^2} \leq C \norm{\partial_t \bff{\eta}(t)}{\bb{L}^2},
	\end{align}
	{where $C$ depends on $\norm{\bff{m}}{L^\infty_T(\bb{H}^2)}$.}
	Taking $\bff{\zeta}= \partial_t \bff{\eta}(t)$ in~\eqref{equ:A2 pa t zeta}, we have
	\begin{align*}
		\mathcal{A}_2(\bff{m}(t); \partial_t \bff{\eta}(t), \widetilde{\bff{\psi}}(t))
		=
		\norm{\partial_t \bff{\eta}(t)}{\bb{L}^2}^2.
	\end{align*}
	As in \eqref{equ:dt rho abc}, this equation and~\eqref{equ:A2 dt eta} yield for all $\bff{\psi}_h \in \bb{V}_h$,
	\begin{align*}
		\norm{\partial_t \bff{\eta}}{\bb{L}^2}^2
		&=
		\mathcal{A}_2 (\bff{m}; \partial_t \bff{\eta},\widetilde{\bff{\psi}}-\bff{\psi}_h)
		+
		2\mathcal{B}_2 (\bff{m},\partial_t \bff{m}; \bff{\eta}, \widetilde{\bff{\psi}}-\bff{\psi}_h)
		-
		2 \mathcal{B}_2(\bff{m},\partial_t \bff{m}; \bff{\eta}, \widetilde{\bff{\psi}})
		\\
		&\quad
		+
		\mathcal{C}_2(\partial_t \bff{m}; \bff{\eta}, \widetilde{\bff{\psi}}-\bff{\psi}_h)
		-
		\mathcal{C}_2(\partial_t \bff{m}; \bff{\eta}, \widetilde{\bff{\psi}}).
	\end{align*}
	By~\eqref{equ:A2 bdd}, H\"older inequality, and the Sobolev embedding, we obtain
	\begin{align*}
		\norm{\partial_t \bff{\eta}}{\bb{L}^2}^2
		&\leq
		C\left(1+\norm{\bff{m}}{\bb{H}^2}^2\right) \norm{\partial_t \bff{\eta}}{\bb{H}^1} \norm{\widetilde{\bff{\psi}}-\bff{\psi}_h}{\bb{H}^1}
		+
		2 \norm{\bff{m}}{\bb{L}^\infty} \norm{\partial_t \bff{m}}{\bb{L}^\infty} \norm{\bff{\eta}}{\bb{H}^1} \norm{\widetilde{\bff{\psi}}-\bff{\psi}_h}{\bb{H}^1}
		\\
		&\quad
		+
		2 \big( \norm{\bff{m}}{\bb{W}^{1,4}} \norm{\partial_t \bff{m}}{\bb{L}^\infty} + \norm{\bff{m}}{\bb{L}^\infty} \norm{\partial_t \bff{m}}{\bb{W}^{1,4}} \big) \norm{\bff{\eta}}{\bb{L}^2} \norm{\widetilde{\bff{\psi}}}{\bb{H}^2}
		\\
		&\quad
		+
		C \norm{\partial_t \bff{m}}{\bb{L}^\infty} \norm{\bff{\eta}}{\bb{L}^2} \norm{\widetilde{\bff{\psi}}-\bff{\psi}_h}{\bb{L}^2}
		+
		C \norm{\partial_t \bff{m}}{\bb{L}^\infty} \norm{\bff{\eta}}{\bb{L}^2} \norm{\widetilde{\bff{\psi}}}{\bb{L}^2}.
	\end{align*}
	We now choose $\bff{\chi}= P_h \widetilde{\bff{\psi}}$. By the same argument as in~\eqref{equ:dt rho h L2}, we infer that
	\begin{align*}
		\norm{\partial_t \bff{\eta}(t)}{\bb{L}^2}^2
		\leq
		C\big(1+\norm{\bff{v}}{L^\infty_T(\bb{H}^{r+1})}\big) h^{r+1} \norm{\partial_t \bff{\eta}(t)}{\bb{L}^2},
	\end{align*}
	where in the last step we used~\eqref{equ:psi2 dt eta}, {and $C$ depends on $\norm{\bff{m}}{W^{1,\infty}_T(\bb{H}^{r+1})}$}. This implies
	\begin{align*}
		\norm{\partial_t \bff{\eta}(t)}{\bb{L}^2}
		\leq
		C\big(1+\norm{\bff{v}}{L^\infty_T(\bb{H}^{r+1})}\big) h^{r+1}.
	\end{align*}
	This, together with~\eqref{equ:dt2 eta H1}, implies~\eqref{equ:dt eta t est}.
	The proof of this proposition will then be complete once we show the following regularity result in Lemma~\ref{lem:dual exist}.
\end{proof}

\begin{lemma}\label{lem:dual exist}
	Assume that $(\bff{m},\bff{s})$ is a strong solution of~\eqref{equ:sdllb} with initial data satisfying \eqref{equ:small init ellip}. For any $\bff{\varphi}\in \bb{L}^2$, $\widetilde{\bff{\varphi}}\in \bb{L}^2$, and for each $t\in [0,T]$, there exists $\bff{\psi}(t)\in \bb{H}^2_{\bff{n}}$ and $\widetilde{\bff{\psi}}(t)\in \bb{H}^2_{\bff{n}}$ such that
	\begin{align}\label{equ:A1 ut zeta}
		\mathcal{A}_1(\bff{m}(t); \bff{\zeta}, \bff{\psi}(t))
		&=
		\inpro{\bff{\varphi}}{\bff{\zeta}}, \quad
		\forall \bff{\zeta}\in \bb{H}^1,
		\\
		\label{equ:A2 ut zeta}
		\mathcal{A}_2(\bff{m}(t); \widetilde{\bff{\zeta}}, \widetilde{\bff{\psi}}(t))
		&=
		\inpro{\widetilde{\bff{\varphi}}}{\widetilde{\bff{\zeta}}}, \quad
		\forall \bff{\widetilde{\zeta}}\in \bb{H}^1.
	\end{align}
	Moreover, there exist constants $C_1$ and $C_2$ such that
	\begin{align}\label{equ:psi1 H2}
		\norm{\bff{\psi}(t)}{\bb{H}^2} \leq C_1 \norm{\bff{\varphi}}{\bb{L}^2},
		\\
		\label{equ:psi2 H2}
		\norm{\widetilde{\bff{\psi}}(t)}{\bb{H}^2} \leq C_2 \norm{\bff{\widetilde{\bff{\varphi}}}}{\bb{L}^2},
	\end{align}
	where the constants $C_1, C_2$ depend on $\Omega$, $T$, and $\norm{\bff{m}}{L^\infty_T(\bb{H}^2)}$.
\end{lemma}

\begin{proof}
	{Let $(\bff{m},\bff{s})$ be a given strong solution with initial data satisfying \eqref{equ:small init ellip}. The existence of $\bff{\psi}(t)\in \bb{H}^1$ satisfying \eqref{equ:A1 ut zeta} and of $\widetilde{\bff{\psi}}(t)\in \bb{H}^1$ satisfying \eqref{equ:A2 ut zeta} follows from the Lax--Milgram theorem, noting the coercivity estimates \eqref{equ:A1 coercive} and \eqref{equ:A2 coercive}, the boundedness of $\mathcal{A}_1$ and $\mathcal{A}_2$ in \eqref{equ:A1 bdd} and \eqref{equ:A2 bdd}, as well as the regularity of the solution $\bff{m}$. Taking $\bff{\zeta}=\bff{\psi}(t)$ in \eqref{equ:A1 ut zeta} and using the coercivity of $\mathcal{A}_1$, we have
	\begin{align*}
		\mu_1 \norm{\bff{\psi}(t)}{\bb{H}^1}^2
		\leq
		\mathcal{A}_1(\bff{m}(t); \bff{\psi}(t), \bff{\psi}(t))
		=
		\inpro{\bff{\varphi}}{\bff{\psi}(t)}
		\leq
		\norm{\bff{\varphi}}{\bb{L}^2} \norm{\bff{\psi}(t)}{\bb{L}^2},
	\end{align*}
	which implies
	\begin{align}\label{equ:psi n H1}
		\norm{\bff{\psi}(t)}{\bb{H}^1} \leq C \norm{\bff{\varphi}}{\bb{L}^2}.
	\end{align}
	Similarly, making use of the coercivity of $\mathcal{A}_2$ and \eqref{equ:m less m0}, we also have
	\begin{align}\label{equ:psi tilde n H1}
		\norm{\widetilde{\bff{\psi}}(t)}{\bb{H}^1} \leq C \norm{\widetilde{\bff{\varphi}}}{\bb{L}^2}.
	\end{align}
	
	Next, we establish \eqref{equ:psi1 H2}; the proof of \eqref{equ:psi2 H2} follows in a similar manner. To this end, we carry out a formal energy estimate, which can be justified rigorously through the standard finite-dimensional Galerkin approximation based on the eigenfunctions of the Neumann Laplacian.}
	Taking $\bff{\zeta}= -\Delta \bff{\psi}$, integrating by parts as necessary, and applying H\"older's inequality together with~\eqref{equ:C Delta w w}, we obtain
	\begin{align*}
		D_\ast \norm{\Delta \bff{\psi}}{\bb{L}^2}^2
		+
		D_\ast \norm{\nabla \bff{\psi}}{\bb{L}^2}^2
		&\leq
		\abs{\mathcal{B}_1(\bff{m},\bff{m}; \Delta \bff{\psi}, \bff{\psi})}
		+
		\abs{\mathcal{C}_1(\bff{m}; \Delta \bff{\psi}, \bff{\psi})}
		+
		\abs{\inpro{\bff{\varphi}}{\Delta \bff{\psi}}}
		\\
		&\leq
		D^\ast \norm{\bff{m}}{\bb{L}^\infty}^2 \norm{\Delta \bff{\psi}}{\bb{L}^2} \norm{\bff{\psi}}{\bb{L}^2}
		\\
		&\quad
		+
		C \norm{\bff{m}}{\bb{W}^{1,4}} \norm{\bff{\psi}}{\bb{W}^{1,4}} \norm{\Delta \bff{\psi}}{\bb{L}^2}
		+
		\norm{\bff{\varphi}}{\bb{L}^2} \norm{\Delta \bff{\psi}}{\bb{L}^2}
		\\
		&\leq
		C \norm{\bff{\psi}}{\bb{H}^1}^2
		+
		\frac{\alpha}{2} \norm{\Delta \bff{\psi}}{\bb{L}^2}^2
		+
		\norm{\bff{\varphi}}{\bb{L}^2}^2,
	\end{align*}
	where we in the last step we also used the Gagliardo--Nirenberg inequality, Sobolev embedding, and Young's inequality. This, together with~\eqref{equ:psi n H1} and the elliptic regularity result~\cite{Gri11}, implies~\eqref{equ:psi1 H2}. This completes the proof of Lemma~\ref{lem:dual exist} (thus also of Proposition~\ref{pro:est rho}).
\end{proof}

Next, we also need the following technical result on the stability of the elliptic projections defined previously. To show this lemma, we need to assume~\eqref{equ:deg r}.

\begin{lemma}\label{lem:stab elliptic}
	Let $(\bff{m},\bff{s})$ be a strong solution of~\eqref{equ:sdllb}
	{and let} $\Pi_h$ and $\Lambda_h$ be {the projection operators}
	defined by~\eqref{equ:A1 proj} and \eqref{equ:A2 proj}, respectively.
	Then {for any $\bff{v}\in L^\infty_T(\bb{H}^2)$ and $t\in [0,T]$,}
	\begin{align}\label{equ:stab Linfty Pi}
		\norm{\Pi_h \bff{v}(t)}{\bb{L}^{\infty}}
		&\leq
		C_{\Pi} \norm{\bff{v}}{L^\infty_T(\bb{H}^2)},
	\end{align}
	{where $C_\Pi$ depends on $T$ and $\norm{\bff{m}}{L^\infty_T(\bb{H}^2)}$, but is independent of $h$ or $\bff{v}$.}
	 
	Now, {suppose that} the polynomial degree $r$ in $\bb{V}_h$
	satisfies~\eqref{equ:deg r}. Then for all {$\bff{v}\in
	L^\infty_T(\bb{H}^3)$ and $t\in [0,T]$,}
	\begin{align}\label{equ:stab ellip Pi}
		\norm{\Pi_h \bff{v}(t)}{\bb{W}^{1,\infty}}
		&\leq
		C \norm{\bff{v}}{L^\infty_T(\bb{H}^3)},
		\\
		\label{equ:stab ellip Lambda}
		\norm{\Lambda_h \bff{v}(t)}{\bb{W}^{1,\infty}}
		&\leq
		C\norm{\bff{v}}{L^\infty_T(\bb{H}^3)}.
	\end{align}
	{The constant $C$ depends on $T$ and $\norm{\bff{m}}{L^\infty_T(\bb{H}^2)}$, but is independent of $h$ or $\bff{v}$.}
\end{lemma}

\begin{proof}
	First, we show~\eqref{equ:stab Linfty Pi}. Recall that $\bff{\rho}$ and $\bff{\eta}$ are defined by \eqref{equ:rho eta def}. By the triangle inequality, the inverse estimate~\eqref{equ:inverse infty H1}, \eqref{equ:rho t est}, and the stability of the finite element projector $P_h$ in \eqref{equ:Ph Lp stab}, we obtain
	\begin{align*}
		\norm{\Pi_h \bff{v}(t)}{\bb{L}^{\infty}}
		&\leq
		\norm{\Pi_h \bff{v}(t)- P_h(\bff{v}(t))}{\bb{L}^{\infty}} 
		+
		\norm{P_h(\bff{v}(t))}{\bb{L}^{\infty}}
		\\
		&\leq
		C\ell_h \big(\norm{\bff{\rho}(t)}{\bb{H}^1} + \norm{P_h(\bff{v}(t))- \bff{v}(t)}{\bb{H}^1} \big) 
		+
		\norm{\bff{v}(t)}{\bb{L}^{\infty}}
		\\
		&\leq
		C\ell_h \big(Ch \norm{\bff{v}}{L^\infty_T(\bb{H}^2)} \big) 
		+
		C\norm{\bff{v}}{L^\infty_T(\bb{L}^\infty)}
		\\
		&\leq
		C(1+h^r \ell_h) \norm{\bff{v}}{L^\infty_T(\bb{H}^2)},
	\end{align*}
	where {$C$ depends on $\norm{\bff{m}}{L^\infty_T(\bb{H}^2)}$}, and $\ell_h$ is defined in~\eqref{equ:ell h}, thus showing \eqref{equ:stab Linfty Pi}.
	
	Next, we prove~\eqref{equ:stab ellip Pi}. By a similar argument, we have
	\begin{align*}
		\norm{\Pi_h \bff{v}(t)}{\bb{W}^{1,\infty}}
		&\leq
		\norm{\Pi_h \bff{v}(t)- P_h(\bff{v}(t))}{\bb{W}^{1,\infty}} 
		+
		\norm{P_h(\bff{v}(t))}{\bb{W}^{1,\infty}}
		\\
		&\leq
		Ch^{-d/2} \big(\norm{\bff{\rho}(t)}{\bb{H}^1} + \norm{P_h(\bff{v}(t))- \bff{v}(t)}{\bb{H}^1} \big) 
		+
		C\norm{\bff{v}(t)}{\bb{W}^{1,\infty}}
		\\
		&\leq
		C(1+h^{r-d/2}) \norm{\bff{v}}{L^\infty_T(\bb{H}^3)}.
	\end{align*}
	where {$C$ depends on $\norm{\bff{m}}{L^\infty_T(\bb{H}^2)}$} and we used the embedding $\bb{H}^3\hookrightarrow \bb{W}^{1,\infty}$. Inequality \eqref{equ:stab ellip Pi} follows by noting the fact that $r\geq d/2$ by the assumption~\eqref{equ:deg r}, while inequality \eqref{equ:stab ellip Lambda} can be shown similarly. This completes the proof of the lemma.
\end{proof}

{We now proceed with the analysis of numerical scheme proposed in Algorithm~\eqref{alg:scheme}. 
To this end, the approximation errors are decomposed as:
\begin{align}
	\label{equ:theta rho}
	\bff{m}_h^n-\bff{m}^n &= (\bff{m}_h^n- \Pi_h \bff{m}^n)+ (\Pi_h \bff{m}^n- \bff{m}^n) =: \bff{\theta}^n + \bff{\rho}^n,
	\\
	\label{equ:xi eta}
	\bff{s}_h^n-\bff{s}^n &= (\bff{s}_h^n- \Lambda_h \bff{s}^n)+ (\Lambda_h \bff{s}^n- \bff{s}^n) =: \bff{\xi}^n + \bff{\eta}^n,
\end{align}
where $\Pi_h \bff{m}^n$ and $\Lambda_h \bff{s}^n$ are, respectively, the $\mathcal{A}_1$-elliptic projection of the solution $\bff{m}(t_n)$ and the $\mathcal{A}_2$-elliptic projection of the solution $\bff{s}(t_n)$ introduced in Definition~\ref{def:elliptic proj}.	
	
The error analysis proceeds by induction, showing the existence of solution and then the error estimate alternatingly. We outline the general strategy now. First, we show a criterion for the unique existence of $(\bff{m}_h^n,\bff{s}_h^n)$ in terms of the numerical solution at the previous time-step, namely:
\begin{equation}\label{equ:mhn1 less D}
	\norm{\bff{m}_h^{n-1}}{\bb{L}^\infty}^2 < D_\ast/(2\beta D^\ast).
\end{equation}

Next, we prove that if $(\bff{m}_h^j,\bff{s}_h^j)$ exists for all $j\leq n$ for some $n\leq N$, then these discrete solutions necessarily satisfy a certain stability property for all $j\leq n$. Under the assumption that the criterion \eqref{equ:mhn1 less D} holds (and thus $(\bff{m}_h^n,\bff{s}_h^n)$ exists), we derive auxiliary error estimates for $\bff{m}_h^n-\Pi_h \bff{m}^n$ and $\bff{s}_h^n-\Lambda_h \bff{s}^n$. These auxiliary error estimates are then used to prove, under a suitable restriction on the time-step size, that $\norm{\bff{m}_h^n}{\bb{L}^\infty}^2 < D_\ast/(2\beta D^\ast)$.
This guarantees the unique existence of $(\bff{m}_h^{n+1},\bff{s}_h^{n+1})$. Proceeding inductively, we thus establish the unique existence of $(\bff{m}_h^j, \bff{s}_h^j)$ along with the corresponding auxiliary error estimates for all $j\leq N$. Finally, these auxiliary estimates, together with the error decompositions in \eqref{equ:theta rho} and \eqref{equ:xi eta}, and the projection error estimates in Proposition~\ref{pro:est rho} yield the desired full error estimates.}

We now proceed to implement this strategy, starting with the following proposition.

\begin{proposition}\label{pro:exist}
{Let $n\geq 1$} and $(\bff{m}_h^{n-1},\bff{s}_h^{n-1})\in \bb{V}_h\times \bb{V}_h$ be given such that condition \eqref{equ:mhn1 less D} holds.
Then there exists a unique solution $(\bff{m}_h^n, \bff{s}_h^n)\in \bb{V}_h \times \bb{V}_h$ to the numerical scheme at time step $n$, as described in Algorithm~\ref{alg:scheme}.
\end{proposition}

\begin{proof}
Given $(\bff{m}_h^{n-1}, \bff{s}_h^{n-1})$ and $\bff{j}^n$, our numerical scheme is seeking {for $(\bff{v},\bff{w})\in \bb{V}_h\times \bb{V}_h$ which satisfies} 
\begin{align*}
	&\inpro{\bff{v}}{\bff{\phi}_h}
	+
	k\mathcal{A}_1(\bff{m}_h^{n-1};\bff{v}, \bff{\phi}_h)
	+
	k\mathcal{L}_1(\bff{s}_h^{n-1}; \bff{v}, \bff{\phi}_h)
	= \inpro{\bff{m}_h^{n-1}}{\bff{\phi}_h},
	\quad \forall \bff{\phi}_h\in \bb{V}_h,
	\\
	&\inpro{\bff{w}}{\bff{\psi}_h}
	+
	k\mathcal{A}_2(\bff{m}_h^{n-1};\bff{w}, \bff{\psi}_h)
	=
	\inpro{\bff{s}_h^{n-1}}{\bff{\psi}_h}
	-
	k\mathcal{L}_2(\bff{j}^n; \bff{m}_h^{n-1}, \bff{\psi}_h),
	\quad \forall \bff{\psi}_h\in \bb{V}_h.
\end{align*}
It is clear that for any $\bff{v}\in \bb{V}_h$,
\begin{align*}
	\inpro{\bff{v}}{\bff{v}}+ k\mathcal{A}_1(\bff{m}_h^{n-1};\bff{v}, \bff{v})
	+
	k\mathcal{L}_1(\bff{s}_h^{n-1}; \bff{v}, \bff{v})
	&\geq 
	\norm{\bff{v}}{\bb{H}^1}^2.
\end{align*}
Moreover, if \eqref{equ:mhn1 less D} holds, then for any $\bff{w}\in \bb{V}_h$,
\begin{align*}
	\inpro{\bff{w}}{\bff{w}}
	+
	k\mathcal{A}_2(\bff{m}_h^{n-1};\bff{w}, \bff{w})
	\geq
	(1+k D_\ast) \norm{\bff{w}}{\bb{L}^2}^2
	+
	k(D_\ast-\beta D^\ast \delta) \norm{\nabla\bff{w}}{\bb{L}^2}^2
	\geqs \norm{\bff{w}}{\bb{H}^1}^2.
\end{align*}
Boundedness of the bilinear forms is a consequence of \eqref{equ:A1 bdd} and \eqref{equ:A2 bdd}.
The existence of a unique solution $(\bff{m}_h^n, \bff{s}_h^n)\in \bb{V}_h\times \bb{V}_h$ then follows from the Lax--Milgram theorem.
\end{proof}

{Assuming the existence of $(\bff{m}_h^j, \bff{s}_h^j)$ for all $j\leq n$, the scheme described in Algorithm~\ref{alg:scheme} enjoys the following stability property.}

\begin{lemma}\label{lem:num stab}
Suppose that for some $n\leq N$,
\begin{equation}\label{equ:max jn m criter}
	\max_{1\leq j\leq n} \norm{\bff{m}_h^{j-1}}{\bb{L}^\infty}^2 
	< D_\ast/(2\beta D^\ast),
\end{equation}
and thus $(\bff{m}_h^j, \bff{s}_h^j)$ exists for all $j\leq n$ (by Proposition~\ref{pro:exist}).
Then for any $j\leq n$,
\begin{align}\label{equ:num stab}
	\norm{\bff{m}_h^j}{\bb{L}^2}^2
	+
	\norm{\bff{s}_h^j}{\bb{L}^2}^2
	+
	k \sum_{i=1}^j \left(\norm{\bff{m}_h^i}{\bb{H}^1}^2 + \norm{\bff{s}_h^i}{\bb{H}^1}^2 \right)
	\leq
	C_T,
\end{align}
where $C_T$ is a constant depending on $T$, but is independent of $n$, $j$, $h$, or $k$.
\end{lemma}

\begin{proof}
We replace the index $n$ with $j$, and set $\bff{\phi}_h=\bff{m}_h^j$ and $\bff{\psi}_h=\bff{s}_h^j$ in~\eqref{equ:weak m disc} and \eqref{equ:weak s disc}, respectively, then add the resulting equations. Noting the elementary vector identities
\begin{align*}
(\bff{a}\times \bff{b})\cdot \bff{a} =0
\;\text{ and }\;
2(\bff{a}-\bff{b})\cdot \bff{a} =  \abs{\bff{a}}^2-\abs{\bff{b}}^2 + \abs{\bff{a}-\bff{b}}^2,
\end{align*}
and applying \eqref{equ:A1 coercive} and \eqref{equ:A2 coercive}, we obtain after discarding some non-negative terms, that for any $j\leq n$,
\begin{align*}
	&\frac{1}{2k} \left( \norm{\bff{m}_h^j}{\bb{L}^2}^2 - \norm{\bff{m}_h^{j-1}}{\bb{L}^2}^2 \right) 
	+
	\frac{1}{2k} \left( \norm{\bff{s}_h^j}{\bb{L}^2}^2 - \norm{\bff{s}_h^{j-1}}{\bb{L}^2}^2 \right) 
	+
	\mu_1 \norm{\bff{m}_h^j}{\bb{H}^1}^2
	+
	\mu_2 \norm{\bff{s}_h^j}{\bb{H}^1}^2
	\\
	&\leq
	\abs{\mathcal{L}_2(\bff{j}^j; \bff{m}_h^{j-1}, \bff{s}_h^j)}
	\\
	&\leq
	C\norm{\bff{j}^j}{\bb{L}^\infty} \norm{\bff{m}_h^{j-1}}{\bb{L}^2} \norm{\bff{s}_h^j}{\bb{H}^1}
	\leq
	C\norm{\bff{m}_h^{j-1}}{\bb{L}^2}^2
	+
	\frac{\mu_2}{2} \norm{\bff{s}_h^j}{\bb{H}^1}^2.
\end{align*}
In the last line, we used \eqref{equ: H1 H1 ineq} and Young's inequality. The required result then follows by the discrete Gronwall lemma.
\end{proof}

Some inequalities for nonlinear terms which are needed to derive auxiliary error estimates will be shown next.

\begin{lemma}\label{lem:nonlinear est}
{Suppose that $(\bff{m},\bff{s})$ is a strong solution to~\eqref{equ:sdllb} satisfying the regularity assumption~\eqref{equ:assum}, and let $(\bff{m}_h^n,\bff{s}_h^n)$ be the sequence defined by Algorithm~\ref{alg:scheme}.} Let the bilinear forms be those introduced in Definition~\ref{def:bilinear}. Given $\epsilon>0$ and $\bff{\chi}\in \bb{H}^1$, the following estimates hold: 
\begin{align}
	\label{equ:B1 est}
	&\abs{\mathcal{B}_1(\bff{m}_h^{n-1},\bff{m}_h^{n-1}; \Pi_h \bff{m}^n, \bff{\chi})- \mathcal{B}_1(\bff{m}^n,\bff{m}^n; \Pi_h \bff{m}^n, \bff{\chi})}
	\nonumber\\
	&\quad\leq 
	C\left(1+ \norm{\bff{m}_h^{n-1}}{\bb{L}^4}^2 \right) \left(\norm{\bff{\theta}^{n-1}}{\bb{L}^2}^2+ h^{2(r+1)}+k^2 \right)
	+
	\epsilon \norm{\bff{\chi}}{\bb{H}^1}^2,
	\\
	\label{equ:C1 est}
	&\abs{\mathcal{C}_1(\bff{m}_h^{n-1}; \Pi_h \bff{m}^n, \bff{\chi})- \mathcal{C}_1(\bff{m}^n; \Pi_h \bff{m}^n, \bff{\chi})}
	\nonumber\\
	&\quad\leq
	C\left(1+ \norm{\bff{m}_h^{n-1}}{\bb{L}^4}^2 \right) \left(\norm{\bff{\theta}^{n-1}}{\bb{L}^2}^2+ h^{2(r+1)}+k^2 \right)
	+
	\epsilon \norm{\bff{\chi}}{\bb{H}^1}^2,
	\\
	\label{equ:B2 est}
	&\abs{\mathcal{B}_2(\bff{m}_h^{n-1},\bff{m}_h^{n-1}; \Lambda_h \bff{s}^n, \bff{\chi})- \mathcal{B}_2(\bff{m}^n,\bff{m}^n; \Lambda_h \bff{s}^n, \bff{\chi})}
	\nonumber\\
	&\quad\leq
	C\left(1+ \norm{\bff{m}_h^{n-1}}{\bb{L}^\infty}^2 \right) \left(\norm{\bff{\theta}^{n-1}}{\bb{L}^2}^2+ h^{2(r+1)}+k^2 \right)  
	+
	\epsilon \norm{\bff{\chi}}{\bb{H}^1}^2, 
	\\
	\label{equ:C2 est}
	&\abs{\mathcal{C}_2(\bff{m}_h^{n-1}; \Lambda_h \bff{s}^n, \bff{\chi})- \mathcal{C}_2(\bff{m}^n; \Lambda_h \bff{s}^n, \bff{\chi})}
	\nonumber\\
	&\quad\leq
	C \left(\norm{\bff{\theta}^{n-1}}{\bb{L}^2}^2+ h^{2(r+1)}+k^2 \right)
	+
	\epsilon \norm{\bff{\chi}}{\bb{L}^2}^2.
\end{align}
Furthermore, with $\bff{\theta}^n$ and $\bff{\xi}^n$ as defined in~\eqref{equ:theta rho} and \eqref{equ:xi eta}, respectively, we also have
\begin{align}
	\label{equ:L1 est}
	\abs{\mathcal{L}_1(\bff{s}_h^{n-1};\bff{m}_h^n,\bff{\theta}^n) - \mathcal{L}_1(\bff{s}^n;\bff{m}^n,\bff{\theta}^n)}
	&\leq
	C \left(1+ \norm{\bff{s}_h^{n-1}}{\bb{L}^4}^2 \right) (h^{2(r+1)}+k^2)
	\nonumber\\
	&\quad
	+
	C\norm{\bff{\xi}^{n-1}}{\bb{L}^2}^2
	+
	\epsilon \norm{\bff{\theta}^n}{\bb{H}^1}^2,
	\\
	\label{equ:L2 est}
	\abs{\mathcal{L}_2(\bff{j}^n;\bff{m}_h^{n-1},\bff{\xi}^n) - \mathcal{L}_2(\bff{j}^n;\bff{m}^n,\bff{\xi}^n)}
	&\leq
	C\norm{\bff{\theta}^{n-1}}{\bb{L}^2}^2 + C(h^{2(r+1)}+k^2)
	+
	\epsilon \norm{\bff{\xi}^n}{\bb{H}^1}^2,
\end{align}
where $C$ depends on $\epsilon$, but is independent of $n$, $h$, or $k$.
\end{lemma}

\begin{proof}
We will refer to the following identities several times in the proof:
\begin{align}
	\label{equ:mhn1 min mn}
	\bff{m}_h^{n-1}- \bff{m}^n 
	&= \bff{\theta}^{n-1} + \bff{\rho}^{n-1} - k\cdot \mathrm{d}_t \bff{m}^n,
	\\
	\label{equ:shn1 min sn}
	\bff{s}_h^{n-1}- \bff{s}^n 
	&= \bff{\xi}^{n-1} + \bff{\eta}^{n-1} - k\cdot \mathrm{d}_t \bff{s}^n.
\end{align}
First, we prove~\eqref{equ:B1 est}.
By H\"older's inequality, \eqref{equ:mhn1 min mn}, and~\eqref{equ:rho t est}, we have
\begin{align}\label{equ:un2 L43}
	\norm{|\bff{m}_h^{n-1}|^2 - |\bff{m}^n|^2}{\bb{L}^{4/3}}
	&=
	\norm{(\bff{m}_h^{n-1}+\bff{m}^n) \cdot (\bff{m}_h^{n-1}-\bff{m}^n)}{\bb{L}^{4/3}}
	\nonumber\\
	&\leq
	\norm{\bff{m}_h^{n-1}+\bff{m}^n}{\bb{L}^4} \norm{\bff{\theta}^{n-1} + \bff{\rho}^{n-1} - k\cdot \mathrm{d}_t \bff{m}^n}{\bb{L}^2}
	\nonumber\\
	&\leq
	C\left(1+ \norm{\bff{m}_h^{n-1}}{\bb{L}^4} \right) \left(\norm{\bff{\theta}^{n-1}}{\bb{L}^2} + h^{r+1}+k \right).
\end{align}
We then obtain for any $\epsilon>0$,
\begin{align*}
	&\abs{\mathcal{B}_1(\bff{m}_h^{n-1},\bff{m}_h^{n-1}; \Pi_h \bff{m}^n, \bff{\chi})- \mathcal{B}_1(\bff{m}^n,\bff{m}^n; \Pi_h \bff{m}^n, \bff{\chi})}
	\\
	&=
	\abs{ \inpro{(|\bff{m}_h^{n-1}|^2- |\bff{m}^{n}|^2) \Pi_h \bff{m}^n}{\bff{\chi}} }
	\\
	&\leq
	\norm{|\bff{m}_h^{n-1}|^2 - |\bff{m}^n|^2}{\bb{L}^{4/3}} \norm{\Pi_h \bff{m}^n}{\bb{L}^\infty} \norm{\bff{\chi}}{\bb{L}^4}
	\\
	&\leq
	C\left(1+ \norm{\bff{m}_h^{n-1}}{\bb{L}^4}^2 \right) \left(\norm{\bff{\theta}^{n-1}}{\bb{L}^2}^2+ h^{2(r+1)}+k^2 \right)
	+
	\epsilon \norm{\bff{\chi}}{\bb{H}^1}^2,
\end{align*}
where in the last step we used~\eqref{equ:un2 L43}, Young's inequality, and Lemma~\ref{lem:stab elliptic} (noting the Sobolev embedding $\bb{H}^1\hookrightarrow \bb{L}^4$). This proves~\eqref{equ:B1 est}.

Next, we prove~\eqref{equ:C1 est}. By \eqref{equ:mhn1 min mn}, Young's inequality, and Lemma~\ref{lem:stab elliptic}, we have for any $\epsilon>0$,
\begin{align*}
	\abs{\mathcal{C}_1(\bff{m}_h^{n-1}; \Pi_h \bff{m}^n, \bff{\chi}) - \mathcal{C}_1(\bff{m}^n; \Pi_h \bff{m}^n, \bff{\chi}) }
	&\leq
	\abs{ \inpro{(\bff{m}_h^{n-1}-\bff{m}^n) \times \nabla \Pi_h \bff{m}^n}{\nabla \bff{\chi}} }
	\\
	&\leq
	C \norm{\bff{m}_h^{n-1}- \bff{m}^n}{\bb{L}^2} \norm{\Pi_h \bff{m}^n}{\bb{W}^{1,\infty}} \norm{\bff{\chi}}{\bb{H}^1}
	\\
	&\leq
	C(h^{2(r+1)} + k^2) + C\norm{\bff{\theta}^{n-1}}{\bb{L}^2}^2
	+
	\epsilon \norm{\bff{\chi}}{\bb{H}^1}^2,
\end{align*}
where in the last step we used~\eqref{equ:rho t est}.

Inequality~\eqref{equ:B2 est} can be estimated by similar argument as follows.
\begin{align*}
	&\abs{\mathcal{B}_2(\bff{m}_h^{n-1},\bff{m}_h^{n-1}; \Lambda_h \bff{s}^n, \bff{\chi})- \mathcal{B}_2(\bff{m}^n,\bff{m}^n; \Lambda_h \bff{s}^n, \bff{\chi})}
	\\
	&\leq
	\abs{\mathcal{B}_2(\bff{m}_h^{n-1}-\bff{m}^n, \bff{m}_h^{n-1}; \Lambda_h \bff{s}^n, \bff{\chi})}
	+
	\abs{\mathcal{B}_2(\bff{m}^n, \bff{m}_h^{n-1}-\bff{m}^n; \Lambda_h \bff{s}^n, \bff{\chi})}
	\\
	&\leq
	C \norm{\bff{m}_h^{n-1}-\bff{m}^n}{\bb{L}^2} \norm{\bff{m}_h^{n-1}}{\bb{L}^\infty} \norm{\Lambda_h \bff{s}^n}{\bb{W}^{1,\infty}} \norm{\bff{\chi}}{\bb{H}^1}
	\\
	&\quad
	+
	C\norm{\bff{m}^n}{\bb{L}^\infty} \norm{\bff{m}_h^{n-1}-\bff{m}^n}{\bb{L}^2} \norm{\Lambda_h \bff{s}^n}{\bb{W}^{1,\infty}} \norm{\bff{\chi}}{\bb{H}^1}
	\\
	&\leq
	C\left(1+ \norm{\bff{m}_h^{n-1}}{\bb{L}^\infty}^2 \right) \left(\norm{\bff{\theta}^{n-1}}{\bb{L}^2}^2+ h^{2(r+1)}+k^2 \right)  
	+
	\epsilon \norm{\bff{\chi}}{\bb{H}^1}^2.
\end{align*}
The proof of~\eqref{equ:C2 est} proceeds in a similar way as that of~\eqref{equ:C1 est}.

By similar argument, noting~\eqref{equ:shn1 min sn} and the fact that $(\bff{a}\times\bff{b})\cdot \bff{a}=0$, we have
\begin{align*}
	&\abs{\mathcal{L}_1(\bff{s}_h^{n-1};\bff{m}_h^n,\bff{\theta}^n) - \mathcal{L}_1(\bff{s}^n;\bff{m}^n,\bff{\theta}^n)}
	\\
	&\leq
	\abs{\mathcal{L}_1(\bff{s}_h^{n-1};\bff{m}_h^n-\bff{m}^n,\bff{\theta}^n) }
	+
	\abs{\mathcal{L}_1(\bff{s}_h^{n-1}-\bff{s}^n; \bff{m}^n,\bff{\theta}^n) }
	\\
	&=
	\abs{\inpro{\bff{\rho}^n\times \bff{s}_h^{n-1}}{\bff{\theta}^n}}
	+
	\abs{\inpro{\bff{m}^n \times (\bff{s}_h^{n-1}-\bff{s}^n)}{\bff{\theta}^n}}
	\\
	&\leq
	\norm{\bff{\rho}^n}{\bb{L}^2} \norm{\bff{s}_h^{n-1}}{\bb{L}^4} \norm{\bff{\theta}^n}{\bb{L}^4}
	+
	\norm{\bff{m}^n}{\bb{L}^\infty} \norm{\bff{s}_h^{n-1}-\bff{s}^n}{\bb{L}^2} \norm{\bff{\theta}^n}{\bb{L}^2}
	\\
	&\leq
	C \left(1+ \norm{\bff{s}_h^{n-1}}{\bb{L}^4}^2 \right) (h^{2(r+1)}+k^2)
	+
	C\norm{\bff{\xi}^{n-1}}{\bb{L}^2}^2
	+
	\epsilon \norm{\bff{\theta}^n}{\bb{H}^1}^2,
\end{align*}
which shows~\eqref{equ:L1 est}. Finally, we show~\eqref{equ:L2 est}. By \eqref{equ: H1 H1 ineq}, H\"older's and Young's inequalities, noting~\eqref{equ:rho t est} and~\eqref{equ:mhn1 min mn}, we obtain for any $\epsilon>0$,
\begin{align*}
	\abs{\mathcal{L}_2(\bff{j}^n;\bff{m}_h^{n-1},\bff{\xi}^n) - \mathcal{L}_2(\bff{j}^n;\bff{m}^n,\bff{\xi}^n)}
	&\leq 
	\norm{\bff{j}^n}{\bb{L}^\infty} \norm{\bff{m}_h^{n-1}-\bff{m}^n}{\bb{L}^2} \norm{ \bff{\xi}^n}{\bb{H}^1}
	\\
	&\leq
	C\norm{\bff{\theta}^{n-1}}{\bb{L}^2}^2 + C(h^{2(r+1)}+k^2)
	+
	\epsilon \norm{\bff{\xi}^n}{\bb{H}^1}^2,
\end{align*}
which shows~\eqref{equ:L2 est}. This completes the proof of the lemma.
\end{proof}

We are now ready to show auxiliary error estimates for the numerical scheme proposed in Algorithm~\ref{alg:scheme}. In the proof, we use the following inequalities:
\begin{align}\label{equ:d rho n}
	\norm{\dtt \bff{\rho}^n}{\bb{L}^2} 
	&= 
	\norm{\frac{1}{k} \int_{t_{n-1}}^{t_n} \partial_t \bff{\rho}(t) \,\dt}{\bb{L}^2} 
	\leq 
	C h^{r+1},
	\\
	\label{equ:diff un dt un}
	\norm{\dtt \bff{m}^n - \partial_t \bff{m}^n}{\bb{L}^2} 
	&= 
	\norm{\frac{1}{2k} \int_{t_{n-1}}^{t_n} (t- t_{n-1}) \, \partial_{t}^2 \bff{m}(t)\,\dt }{\bb{L}^2} 
	\leq 
	Ck.
\end{align}
Similar estimates hold for $\norm{\dtt \bff{\eta}^n}{\bb{L}^2}$ and $\norm{\dtt \bff{s}^n - \partial_t \bff{s}^n}{\bb{L}^2}$.

\begin{proposition}
Let $n\leq N$, and let $\bff{\theta}^n$ and $\bff{\xi}^n$ be as defined in \eqref{equ:theta rho} and~\eqref{equ:xi eta}, respectively. {Assume that \eqref{equ:max jn m criter} holds and write $\delta:= D_\ast/ (2\beta D^\ast)$. }
Then
\begin{align}\label{equ:theta xi err}
	\norm{\bff{\theta}^n}{\bb{L}^2}^2
	+
	\norm{\bff{\xi}^n}{\bb{L}^2}^2
	+
	k \sum_{j=1}^n \left( \norm{\nabla \bff{\theta}^j}{\bb{L}^2}^2 + \norm{\nabla \bff{\xi}^j}{\bb{L}^2}^2 \right)
	\leq
	\widetilde{C} e^{\widetilde{C} (1+\delta)T} (h^{2(r+1)}+k^2),
\end{align}
where $\widetilde{C}$ is a constant independent of $n$, $h$ or $k$.
\end{proposition}

\begin{proof}
Noting the definitions of $\bff{\theta}^n$ and $\bff{\rho}^n$, the scheme~\eqref{equ:weak m disc}, and the weak formulation~\eqref{equ:weak m cont}, we have for all $\bff{\chi}\in \bb{V}_h$,
\begin{align}\label{equ:A1 theta n}
	&\inpro{\mathrm{d}_t \bff{\theta}^n}{\bff{\chi}}
	+
	\mathcal{A}_1(\bff{m}_h^{n-1};\bff{\theta}^n, \bff{\chi})
	\nonumber\\
	&=
	\inpro{\mathrm{d}_t \bff{m}_h^n}{\bff{\chi}}
	+
	\mathcal{A}_1(\bff{m}_h^{n-1};\bff{m}_h^n, \bff{\chi})
	-
	\inpro{\mathrm{d}_t \Pi_h \bff{m}^n}{\bff{\chi}}
	-
	\mathcal{A}_1(\bff{m}_h^{n-1}; \Pi_h \bff{m}^n, \bff{\chi})
	\nonumber\\
	&=
	\mathcal{L}_1(\bff{s}_h^{n-1}; \bff{m}_h^n, \bff{\chi})
	-
	\inpro{\mathrm{d}_t \Pi_h\bff{m}^n- \partial_t \bff{m}^n}{\bff{\chi}}
	-
	\inpro{\partial_t \bff{m}^n}{\bff{\chi}}
	\nonumber\\
	&\quad
	-
	\left(\mathcal{A}_1(\bff{m}_h^{n-1}; \Pi_h \bff{m}^n, \bff{\chi})- \mathcal{A}_1(\bff{m}^n;\Pi_h \bff{m}^n,\bff{\chi})\right) 
	-
	\mathcal{A}_1(\bff{m}^n; \Pi_h\bff{m}^n, \bff{\chi})
	\nonumber\\
	&=
	\left(\mathcal{L}_1(\bff{s}^n; \bff{m}^n,\bff{\chi}) -\mathcal{L}_1 (\bff{s}_h^{n-1}; \bff{m}_h^n, \bff{\chi}) \right) 
	-
	\inpro{\mathrm{d}_t \bff{\rho}^n}{\bff{\chi}}
	-
	\inpro{\mathrm{d}_t \bff{m}^n-\partial_t \bff{m}^n}{\bff{\chi}}
	\nonumber\\
	&\quad
	-
	\left(\mathcal{A}_1(\bff{m}_h^{n-1}; \Pi_h \bff{m}^n, \bff{\chi})- \mathcal{A}_1(\bff{m}^n;\Pi_h \bff{m}^n,\bff{\chi})\right).
\end{align}
We now take $\bff{\chi}=\bff{\theta}^n$ in~\eqref{equ:A1 theta n}. Applying~\eqref{equ:A1 coercive}, Young's inequality, \eqref{equ:rho t est}, \eqref{equ:dt rho t est}, \eqref{equ:d rho n}, \eqref{equ:diff un dt un}, and Lemma~\ref{lem:nonlinear est}, we have
\begin{align}\label{equ:12k theta n}
	&\frac{1}{2k} \left(\norm{\bff{\theta}^n}{\bb{L}^2}^2 - \norm{\bff{\theta}^{n-1}}{\bb{L}^2}^2 \right)
	+
	\frac{1}{2k} \norm{\bff{\theta}^n- \bff{\theta}^{n-1}}{\bb{L}^2}^2
	+
	\mu_1 \norm{\bff{\theta}^n}{\bb{H}^1}^2
	\nonumber\\
	&\leq
	\abs{\mathcal{L}_1(\bff{s}^n; \bff{m}^n,\bff{\theta}^n) -\mathcal{L}_1 (\bff{s}_h^{n-1}; \bff{m}_h^n, \bff{\theta}^n)}
	+
	\abs{\inpro{\mathrm{d}_t \bff{\rho}^n}{\bff{\theta}^n}}
	\nonumber\\
	&\quad
	+
	\abs{\inpro{\mathrm{d}_t \bff{m}^n-\partial_t \bff{m}^n}{\bff{\theta}^n}}
	+
	\abs{\mathcal{A}_1(\bff{m}_h^{n-1}; \Pi_h \bff{m}^n, \bff{\theta}^n)- \mathcal{A}_1(\bff{m}^n;\Pi_h \bff{m}^n,\bff{\theta}^n)}
	\nonumber\\
	&\leq
	C\left(1+ \norm{\bff{m}_h^{n-1}}{\bb{L}^4}^2 \right) \left(\norm{\bff{\theta}^{n-1}}{\bb{L}^2}^2+ h^{2(r+1)}+k^2 \right)
	+
	C \left(1+ \norm{\bff{s}_h^{n-1}}{\bb{L}^4}^2 \right) (h^{2(r+1)}+k^2)
	\nonumber\\
	&\quad
	+
	C\norm{\bff{\xi}^{n-1}}{\bb{L}^2}^2
	+
	\epsilon \norm{\bff{\theta}^n}{\bb{H}^1}^2.
\end{align}
Analogous to~\eqref{equ:A1 theta n}, noting the definitions of $\bff{\xi}^n$ and $\bff{\eta}^n$, we obtain
\begin{align*}
	&\inpro{\mathrm{d}_t \bff{\xi}^n}{\bff{\chi}}
	+
	\mathcal{A}_2(\bff{m}_h^{n-1};\bff{\xi}^n, \bff{\chi})
	\nonumber\\
	&=
	\left(\mathcal{L}_2(\bff{j}^n; \bff{m}^n,\bff{\chi}) -\mathcal{L}_2 (\bff{j}^n; \bff{m}_h^{n-1}, \bff{\chi}) \right) 
	-
	\inpro{\mathrm{d}_t \bff{\eta}^n}{\bff{\chi}}
	-
	\inpro{\mathrm{d}_t \bff{s}^n-\partial_t \bff{s}^n}{\bff{\chi}}
	\nonumber\\
	&\quad
	-
	\left(\mathcal{A}_2(\bff{m}_h^{n-1}; \Lambda_h \bff{s}^n, \bff{\chi})- \mathcal{A}_2(\bff{m}^n;\Lambda_h \bff{s}^n,\bff{\chi})\right).
\end{align*}
We apply the same argument as in~\eqref{equ:12k theta n}. Taking $\bff{\chi}=\bff{\xi}^n$ and using~\eqref{equ:A2 coercive}, noting the assumption~\eqref{equ:max jn m criter}, we obtain
\begin{align}\label{equ:12k xi n}
	&\frac{1}{2k} \norm{\bff{\xi}^n}{\bb{L}^2}^2
	+
	\mu_2 \norm{\bff{\xi}^n}{\bb{H}^1}^2
	\nonumber\\
	&\leq
	\abs{\mathcal{L}_2(\bff{j}^n; \bff{m}^n,\bff{\xi}^n) -\mathcal{L}_1 (\bff{j}^n; \bff{m}_h^{n-1}, \bff{\xi}^n)}
	+
	\abs{\inpro{\mathrm{d}_t \bff{\eta}^n}{\bff{\xi}^n}}
	\nonumber\\
	&\quad
	+
	\abs{\inpro{\mathrm{d}_t \bff{s}^n-\partial_t \bff{s}^n}{\bff{\xi}^n}}
	+
	\abs{\mathcal{A}_2(\bff{m}_h^{n-1}; \Lambda_h \bff{s}^n, \bff{\xi}^n)- \mathcal{A}_2(\bff{m}^n;\Lambda_h \bff{s}^n,\bff{\xi}^n)}
	\nonumber\\
	&\leq
	C\left(1+ \norm{\bff{m}_h^{n-1}}{\bb{L}^\infty}^2 \right) \left(\norm{\bff{\theta}^{n-1}}{\bb{L}^2}^2+ h^{2(r+1)}+k^2 \right)  
	+
	\epsilon \norm{\bff{\xi}^n}{\bb{H}^1}^2.
\end{align}
Adding \eqref{equ:12k theta n} and \eqref{equ:12k xi n} yields
\begin{align*}
	&\frac{1}{2k} \left( \norm{\bff{\theta}^n}{\bb{L}^2}^2 - \norm{\bff{\theta}^{n-1}}{\bb{L}^2}^2 \right) 
	+\frac{1}{2k} \left(\norm{\bff{\xi}^n}{\bb{L}^2}^2 - \norm{\bff{\xi}^{n-1}}{\bb{L}^2}^2 \right) 
	+
	\mu_1 \norm{\bff{\theta}^n}{\bb{H}^1}^2
	+
	\mu_2 \norm{\bff{\xi}^n}{\bb{H}^1}^2
	\\
	&\leq
	C \left(1+ \norm{\bff{s}_h^{n-1}}{\bb{L}^4}^2 \right) (h^{2(r+1)}+k^2)
	+
	C\left(1+ \norm{\bff{m}_h^{n-1}}{\bb{L}^\infty}^2 \right) \left(\norm{\bff{\theta}^{n-1}}{\bb{L}^2}^2+ h^{2(r+1)}+k^2 \right)  
	\\
	&\quad
	+
	C\norm{\bff{\xi}^{n-1}}{\bb{L}^2}^2
	+
	\epsilon \norm{\bff{\theta}^n}{\bb{H}^1}^2
	+
	\epsilon \norm{\bff{\xi}^n}{\bb{H}^1}^2
	\\
	&\leq
	C\left(1+\delta+ \norm{\bff{s}_h^{n-1}}{\bb{L}^4}^2 \right) \left(h^{2(r+1)}+k^2\right) 
	+
	C(1+\delta) \norm{\bff{\theta}^{n-1}}{\bb{L}^2}^2
	+
	C\norm{\bff{\xi}^{n-1}}{\bb{L}^2}^2
	\\
	&\quad
	+
	\epsilon \norm{\bff{\theta}^n}{\bb{H}^1}^2
	+
	\epsilon \norm{\bff{\xi}^n}{\bb{H}^1}^2,
\end{align*}
where in the last step we used the assumption~\eqref{equ:max jn m criter}. Choosing $\epsilon>0$ sufficiently small, summing over $m\in \{1,2,\ldots,n\}$, and applying the discrete Gronwall lemma, we obtain
\begin{align*}
	&\norm{\bff{\theta}^n}{\bb{L}^2}^2
	+
	\norm{\bff{\xi}^n}{\bb{L}^2}^2
	+
	k \sum_{j=1}^n \left( \norm{\nabla \bff{\theta}^j}{\bb{L}^2}^2 + \norm{\nabla \bff{\xi}^j}{\bb{L}^2}^2 \right)
	\\
	&\leq
	C \left[\norm{\bff{\theta}^0}{\bb{L}^2}^2 + \norm{\bff{\xi}^0}{\bb{L}^2}^2 + k \sum_{j=1}^n \left(1+\delta+ \norm{\bff{s}_h^{j-1}}{\bb{H}^1}^2 \right) (h^{2(r+1)}+k^2) \right] 
	\exp\left(C(1+\delta)T\right)
	\\
	&\leq
	C e^{C(1+\delta)T} (h^{2(r+1)}+k^2),
\end{align*}
where we used the Sobolev embedding $\bb{H}^1\hookrightarrow \bb{L}^4$ and the stability estimate \eqref{equ:num stab}. In the last step, we also used the fact that by the triangle inequality, \eqref{equ:Ph approx}, and \eqref{equ:rho t est},
\begin{align*}
	\frac12 \left(\norm{\bff{\theta}^0}{\bb{L}^2}^2 + \norm{\bff{\xi}^0}{\bb{L}^2}^2\right)
	&\leq
	\norm{P_h \bff{m}_0- \bff{m}_0}{\bb{L}^2}^2 + \norm{\bff{m}_0- \Pi_h \bff{m}^0}{\bb{L}^2}^2
	+
	\norm{P_h \bff{s}_0- \bff{s}_0}{\bb{L}^2}^2 + \norm{\bff{s}_0- \Pi_h \bff{s}^0}{\bb{L}^2}^2
	\\
	&\leq
	Ch^{2(r+1)}.
\end{align*}
This completes the proof of the theorem.
\end{proof}

Finally, we now turn to the proof of the main theorem.

\begin{theorem}\label{the:error}
Let $n\leq N$. Let $(\bff{m},\bff{s})$ be a strong solution to the problem~\eqref{equ:sdllb} with regularity given by~\eqref{equ:assum}. Suppose that the initial data (or $\beta$) is sufficiently small such that
\begin{equation}\label{equ:small m}
	\norm{\bff{m}}{L^\infty_T(\bb{H}^2)}^2 \leq \frac{D_\ast}{18 C_\Pi^2 \beta D^\ast} \;\text{ and }\; 
	\norm{\bff{m}_h^0}{\bb{L}^\infty}^2 \leq \frac{D_\ast}{2 \beta D^\ast}.
\end{equation}
Assume further that $h$ and $k$ are sufficiently small such that $k=O(h^{-d/2})$. {Then the scheme defined by Algorithm~\ref{alg:scheme} is well-posed and stable (in the sense of Lemma~\ref{lem:num stab}). Furthermore, the following error estimates hold:}
\begin{equation}\label{equ:error}
\begin{alignedat}{1}
	\norm{\bff{m}_h^n-\bff{m}(t_n)}{\bb{L}^2}^2
	+
	\norm{\bff{s}_h^n-\bff{s}(t_n)}{\bb{L}^2}^2
	&\leq 
	C\left(h^{2(r+1)} +k^2 \right),
	\\
	k\sum_{j=1}^n \left( \norm{\nabla\bff{m}_h^j-\nabla\bff{m}(t_j)}{\bb{L}^2}^2
	+
	\norm{\nabla\bff{s}_h^j- \nabla\bff{s}(t_j)}{\bb{L}^2}^2 \right) 
	&\leq 
	C\left(h^{2r} +k^2 \right),
\end{alignedat}
\end{equation}
where $C$ depends on $T$, but is independent of $n$, $h$ or $k$.
\end{theorem}

\begin{proof}
Let $\delta:= D_\ast/ (2\beta D^\ast)$. We will prove the theorem by induction.

Firstly, if \eqref{equ:small m} holds, then a unique solution $(\bff{m}_h^1,\bff{s}_h^1)$ to the numerical scheme exists (by Proposition~\ref{pro:exist}), satisfying the auxiliary error estimate \eqref{equ:theta xi err} for $n=1$. By \eqref{equ:rho t est} and \eqref{equ:eta t est}, together with \eqref{equ:theta xi err} and the triangle inequality, we then have \eqref{equ:error} for $n=1$.

For the inductive step, suppose that for some $n\leq N$,
\begin{align*}
	\max_{0\leq j\leq n-1} \norm{\bff{m}_h^j}{\bb{L}^\infty}^2 \leq \delta.
\end{align*}
Under this assumption, we will show that the conclusions of the theorem hold at time step $n$ and $\norm{\bff{m}_h^n}{\bb{L}^\infty}^2 \leq \delta$, thus completing the induction. Indeed, the existence of a unique solution $(\bff{m}_h^n,\bff{s}_h^n)$ follows from Proposition~\ref{pro:exist}. The auxiliary error estimate~\eqref{equ:theta xi err} also holds. Therefore, by \eqref{equ:rho t est}, \eqref{equ:eta t est}, and the triangle inequality, we obtain~\eqref{equ:error}. 

Furthermore, by the definition of $\bff{\theta}^n$, the triangle inequality, the inverse estimate~\eqref{equ:inverse}, the stability estimate~\eqref{equ:stab Linfty Pi}, and the auxiliary error estimate~\eqref{equ:theta xi err},
\begin{align*}
	\norm{\bff{m}_h^n}{\bb{L}^\infty}
	&\leq
	\norm{\bff{\theta}^n}{\bb{L}^\infty}
	+
	\norm{\Pi_h \bff{m}^n}{\bb{L}^\infty}
	\\
	&\leq
	C_{\textrm{inv}} h^{-\frac{d}{2}} \norm{\bff{\theta}^n}{\bb{L}^2}
	+
	C_\Pi \norm{\bff{m}}{L^\infty_T(\bb{H}^2)}
	\\
	&\leq
	\left(C_{\textrm{inv}} h^{-\frac{d}{2}}\right) \widetilde{C}e^{\widetilde{C}(1+\delta)T} \left( h^{r+1}+k \right) 
	+
	C_\Pi \norm{\bff{m}}{L^\infty_T(\bb{H}^2)}
	\\
	&\leq
	C_{\textrm{inv}} \widetilde{C} e^{\widetilde{C}(1+\delta)T} h^{r+1-\frac{d}{2}} + C_{\textrm{inv}} \widetilde{C} e^{\widetilde{C}(1+\delta)T} kh^{-\frac{d}{2}}
	+ C_\Pi \norm{\bff{m}}{L^\infty_T(\bb{H}^2)}.
\end{align*}
Suppose that $h$ and $k$ are sufficiently small (noting~\eqref{equ:deg r}) such that
\begin{align*}
	h^{r+1-\frac{d}{2}} 
	\leq
	\frac{\delta^{\frac12}}{3 C_{\textrm{inv}} \widetilde{C} e^{\widetilde{C}(1+\delta)T}}
	\;\text{ and }\;
	kh^{-\frac{d}{2}}
	\leq
	\frac{\delta^{\frac12}}{3 C_{\textrm{inv}} \widetilde{C} e^{\widetilde{C}(1+\delta)T}}.
\end{align*}
Combining this with the first inequality in assumption~\eqref{equ:small m}, we obtain $\norm{\bff{m}_h^n}{\bb{L}^\infty} \leq \delta^{\frac12}$. This concludes the induction step, thus completing proof of the theorem.
\end{proof}

We note that condition \eqref{equ:small m} can be ensured by choosing sufficiently small initial data $(\bff{m}_0,\bff{s}_0)$ or a small $\beta$, in view of estimate \eqref{equ:Delta m L2}, as well as our choice of $\bff{m}_h^0$ and the stability property \eqref{equ:Ph Lp stab}.

\section{Numerical experiments} \label{sec:num exp}

We perform some numerical simulations for the numerical scheme in Algorithm~\ref{alg:scheme} using the open-source package~\textsc{FEniCS}. The results are presented in this section. Since the exact solution of the equation is not known, we use extrapolation to verify the order of convergence experimentally. To this end, let $\left(\bff{m}_h^n, \bff{s}_h^n\right)$ be the finite element solution with spatial step size $h$ and time-step size $k=\lfloor T/n\rfloor$. We define the extrapolated spatial order of convergence
\begin{equation*}
	h\text{-rate}(\cdot) :=  \log_2 \left[\frac{\max_n \norm{\bff{e}_{2h}(\cdot)}{\bb{L}^2}}{\max_n \norm{\bff{e}_{h}(\cdot)}{\bb{L}^2}}\right],
\end{equation*}
where $\bff{e}_h(\bff{m}) := \bff{m}_{h}^n-\bff{m}_{h/2}^n$ and $\bff{e}_h(\bff{s}) := \bff{s}_{h}^n-\bff{s}_{h/2}^n$. Taking the time-step size $k$ to be very small, we expect the error in the numerical scheme to be dominated by that due to spatial discretisation.

\subsection{Simulation 1 (disk magnet, above the Curie temperature)}
Let $\Omega$ be a unit disk. We take $k=1.0\times 10^{-7}$. The coefficients in~\eqref{equ:sdllb} are $\gamma=2.3\times 10^6, \alpha=1.0\times 10^5, \gamma'=3.0\times 10^{-3}, \alpha'= 1.0\times 10^{-6}, \kappa=1.0\times 10^{-9}, \mu=1.0\times 10^4, D_0=1.0\times 10^{-2}, \tau_{\textrm{sf}}=1.0\times 10^{-7}, \tau_{\textrm{J}}=5\times 10^{-8}, \beta=0.1, \beta'=1.0\times 10^{-5}$, which are of typical order of magnitude (in SI units) for a micromagnetic simulation of a ferromagnet \cite{AbeRugBru15}. The current density is given by $\bff{j}=(0,2.0\times 10^8)^\top$. The initial magnetisation vector $\bff{m}_0$ and spin vector $\bff{s}_0$ is given by
\begin{align*}
	\bff{m}_0(x,y)&= \big(-0.1y, 0.1x, 0.1(1-x^2-y^2) \big),
	\\
	\bff{s}_0(x,y)&= \big(0.1y, -0.1x, -0.1(1-x^2-y^2)\big).
\end{align*}
Snapshots of the magnetisation vector field $\bff{m}$ and the spin accumulation vector field $\bff{s}$ at selected times are shown in Figure~\ref{fig:snapshots field 2d 1} and Figure~\ref{fig:snapshots spin 2d 1}, respectively. The colours indicate the relative magnitude of the vectors. Plots of $\bff{e}_h$ against $1/h$ are displayed in Figure~\ref{fig:order u 1} and Figure~\ref{fig:order H 1}, which show the expected order of convergence. 

{The local magnetisation vector field $\bff{m}$ is observed to precess regularly around both the effective field and the spin accumulation field. Over a longer time scale, the magnitude of $\bff{m}$ is expected to decay to zero, as predicted by \eqref{equ:u Lp}. The spin accumulation field $\bff{s}$ also appears to decay to zero, as suggested by \eqref{equ:limsup s zero}, seemingly at a faster rate.}

\subsection{Simulation 2 (square magnet, above the Curie temperature)}
Let $\Omega= [-1,1]\times [-1,1]$. We take $k=1.0\times 10^{-9}$. The material coefficients are taken to be the same as in Simulation 1. The current density is given by $\bff{j}=(0,1.0\times 10^7)^\top$. The initial magnetisation vector $\bff{m}_0$ and spin vector $\bff{s}_0$ is given by
\begin{align*}
	\bff{m}_0(x,y)&= \big(-0.1y, 0.1x, 0.1\sin(2\pi x) \big),
	\\
	\bff{s}_0(x,y)&= \big(0.1x, -0.1y, 0.1xy\big).
\end{align*}
Snapshots of the magnetisation vector field $\bff{m}$ and the spin accumulation vector field $\bff{s}$ at selected times are shown in Figure~\ref{fig:snapshots field 2d 2} and Figure~\ref{fig:snapshots spin 2d 2}, respectively. The colours indicate the relative magnitude of the vectors. Plots of $\bff{e}_h$ against $1/h$ are given in Figure~\ref{fig:order u 2} and Figure~\ref{fig:order H 2}, which show the expected order of convergence. 

{The spin accumulation field and the effective field both induce torques on $\bff{m}$. The magnitudes of $\bff{m}$ and $\bff{s}$ are observed to decay to zero in this simulation as well.}

\subsection{Simulation 3 (below the Curie temperature)}
At moderate temperatures below the Curie temperature, the Ginzburg--Landau theory dictates that $\mu<0$. In the literature, equation~\eqref{equ:sdllb a} with $\gamma'=0$ is also known as the Landau--Lifshitz--Baryakhtar equation in the limiting case of vanishing higher-order damping term~\cite{Bar84, SoeTra23, WanDvo15}. Our numerical scheme and its error analysis also applies to this case with minor modifications.

Let $\Omega= [-1,1]\times [-1,1]$. We take $k=1.0\times 10^{-9}$. The coefficients in~\eqref{equ:sdllb} are now taken to be $\gamma=2.0\times 10^6, \alpha=2.0\times 10^5, \gamma'=1.0\times 10^{-3}, \alpha'= 1.2\times 10^{-6}, \kappa=1.0\times 10^{-9}, \mu=-1.0\times 10^4, D_0=2.0\times 10^{-2}, \tau_{\textrm{sf}}=2.0\times 10^{-7}, \tau_{\textrm{J}}=1.0\times 10^{-7}, \beta=0.2, \beta'=1.0\times 10^{-5}$. The current density is given by $\bff{j}=(1.0\times 10^6, 0)^\top$. The initial magnetisation vector $\bff{m}_0$ and spin vector $\bff{s}_0$ is given by
\begin{align*}
	\bff{m}_0(x,y)&= \big(0.2\sin(2\pi y), 0.2\sin(2\pi x), 0.05 \big),
	\\
	\bff{s}_0(x,y)&= \big(0.1\cos(2\pi x), -0.1\cos(\pi x), 0.1xy\big).
\end{align*}
Snapshots of the magnetisation vector field $\bff{m}$ and the spin accumulation vector field $\bff{s}$ at selected times are shown in Figure~\ref{fig:snapshots field 2d 3} and Figure~\ref{fig:snapshots spin 2d 3}, respectively. The colours indicate the relative magnitude of the vectors. Plots of $\bff{e}_h$ against $1/h$ are displayed in Figure~\ref{fig:order u 3} and Figure~\ref{fig:order H 3}, which show the expected order of convergence. 

{In this moderate temperature regime, Figure \ref{fig:snapshots field 2d 3} appears to show varying degrees of torque and precession in the local magnetisation vectors. However, the magnitude of $\bff{m}$ does not decay to zero, consistent with the observations in~\cite{Soe25b, Soe24}; in the present case, the additional torque induces nonuniformity in the magnetisation. Notably, Figure \ref{fig:snapshots spin 2d 3} qualitatively indicates that the magnitude of the spin accumulation $\bff{s}$ still decays to zero.}

\section{Conclusion}

We have shown the existence and uniqueness of global solution to the spin-diffusion Landau--Lifshitz--Bloch equation, which is a system of coupled quasilinear PDEs modelling the evolution of the magnetisation vector field in the presence of spin-polarised currents at high temperatures. We then propose a finite element scheme which only involves solving two decoupled linear systems and perform a rigorous error analysis, assuming adequate regularity of the exact solution and sufficiently small initial data. Some numerical results are presented which confirm our theoretical analysis. 

For future research, we will consider the problem posed in a multi-layer domain which leads to a system with discontinuous coefficients. Higher-order time-marching scheme involving the backward difference formulae (as in~\cite{AkrFeiKovLub21, GuiWanChe24, HuZhaAn22}) will also be considered in a future work.

\section*{Funding information}
The author is supported by the Australian Government Research
Training Program (RTP) Scholarship awarded at the University of New South Wales, Sydney.
Financial support from the Australian Research Council under grant number DP200101866 is gratefully acknowledged.

\section*{Conflict of interest declaration}
The authors declare no conflict of interest.

\section*{Data availability statement}
Data sharing not applicable to this article as no datasets were generated or analysed during the current study.

\section*{Acknowledgements}
The author would like to thank the referees for their careful reading and valuable comments, which have significantly improved the quality of the paper.

\begin{figure}[!htb]
	\centering
	\begin{subfigure}[b]{0.28\textwidth}
		\centering
		\includegraphics[width=\textwidth]{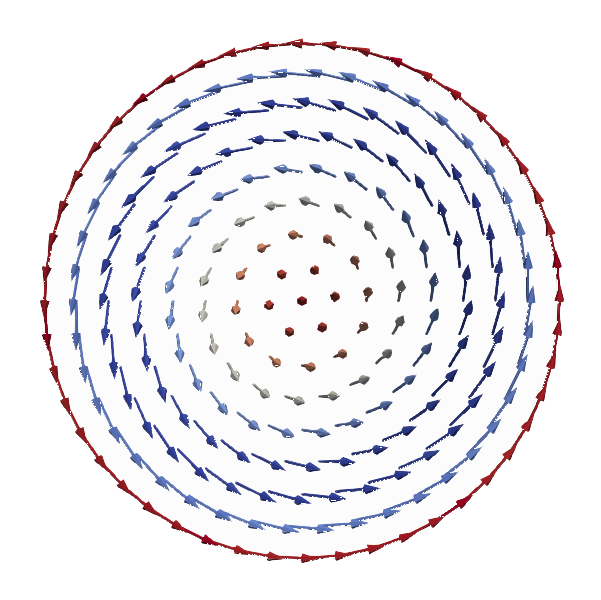}
		\caption{$t=0$}
	\end{subfigure}
	\begin{subfigure}[b]{0.28\textwidth}
		\centering
		\includegraphics[width=\textwidth]{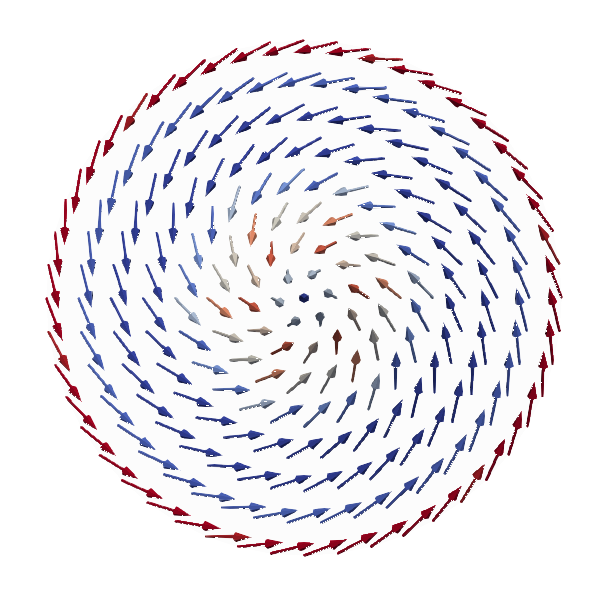}
		\caption{$t=2.5\times 10^{-7}$}
	\end{subfigure}
	\begin{subfigure}[b]{0.28\textwidth}
		\centering
		\includegraphics[width=\textwidth]{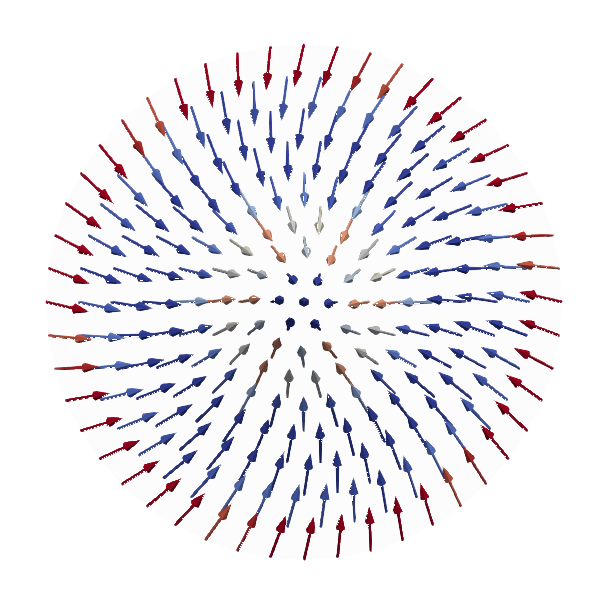}
		\caption{$t=1.0 \times 10^{-6}$}
	\end{subfigure}
	\begin{subfigure}[b]{0.1\textwidth}
		\centering
		\includegraphics[width=\textwidth]{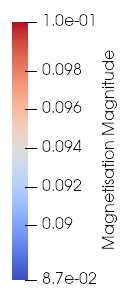}
	\end{subfigure}
	\begin{subfigure}[b]{0.28\textwidth}
		\centering
		\includegraphics[width=\textwidth]{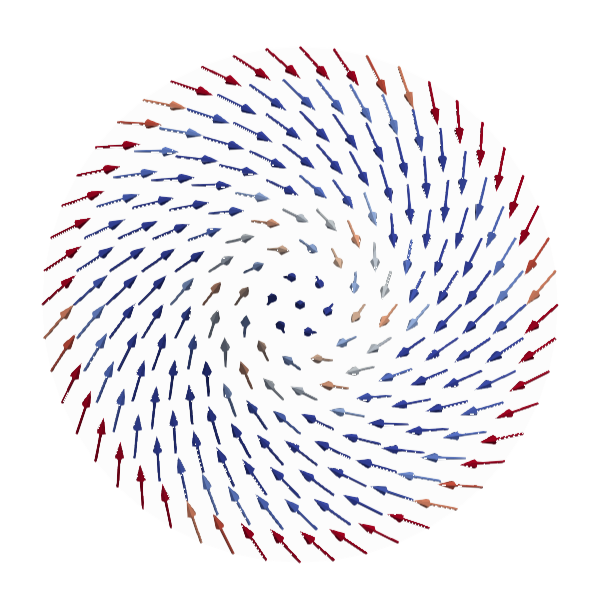}
		\caption{$t=2.0\times 10^{-6}$}
	\end{subfigure}
	\begin{subfigure}[b]{0.28\textwidth}
		\centering
		\includegraphics[width=\textwidth]{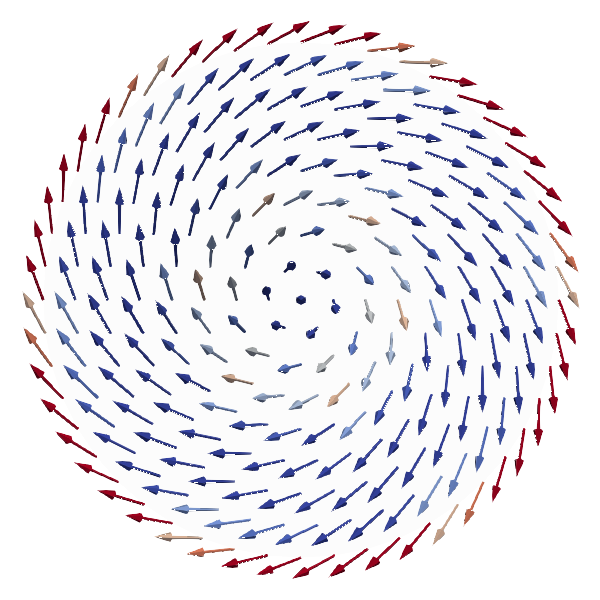}
		\caption{$t=4.0\times 10^{-6}$}
	\end{subfigure}
	\begin{subfigure}[b]{0.28\textwidth}
		\centering
		\includegraphics[width=\textwidth]{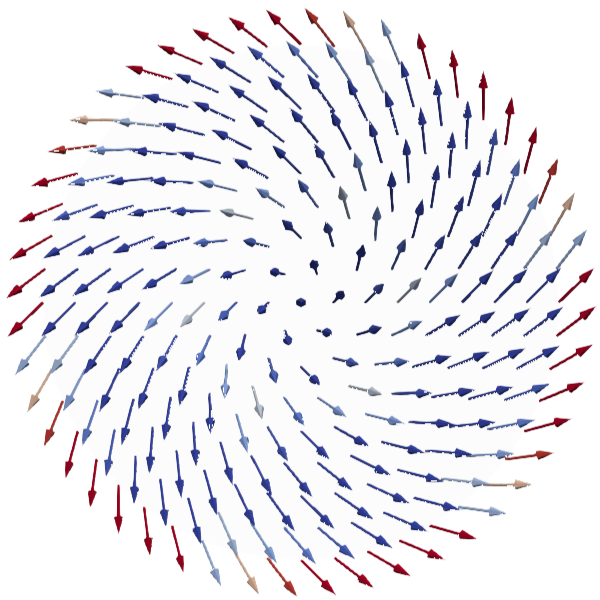}
		\caption{$t=5.0\times 10^{-6}$}
	\end{subfigure}
	\begin{subfigure}[b]{0.1\textwidth}
		\centering
		\includegraphics[width=\textwidth]{md_legend.png}
	\end{subfigure}
	\caption{Snapshots of the magnetisation vector field $\bff{m}$ (projected onto $\bb{R}^2$) for simulation 1.}
	\label{fig:snapshots field 2d 1}
\end{figure}

\begin{figure}[!htb]
	\centering
	\begin{subfigure}[b]{0.28\textwidth}
		\centering
		\includegraphics[width=\textwidth]{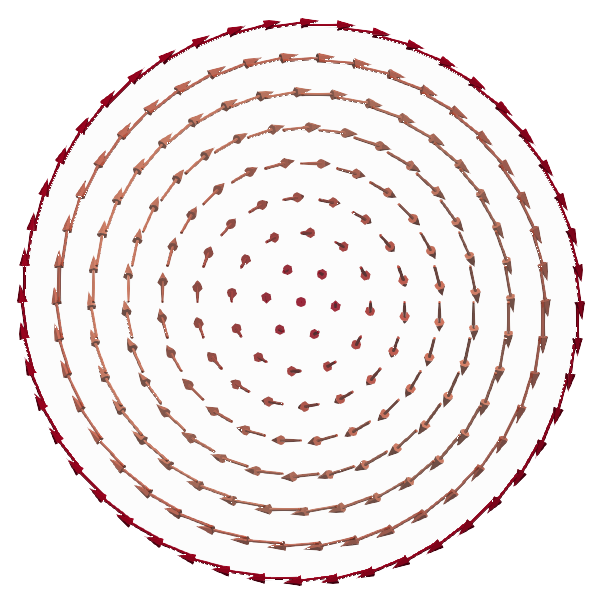}
		\caption{$t=0$}
	\end{subfigure}
	\begin{subfigure}[b]{0.28\textwidth}
		\centering
		\includegraphics[width=\textwidth]{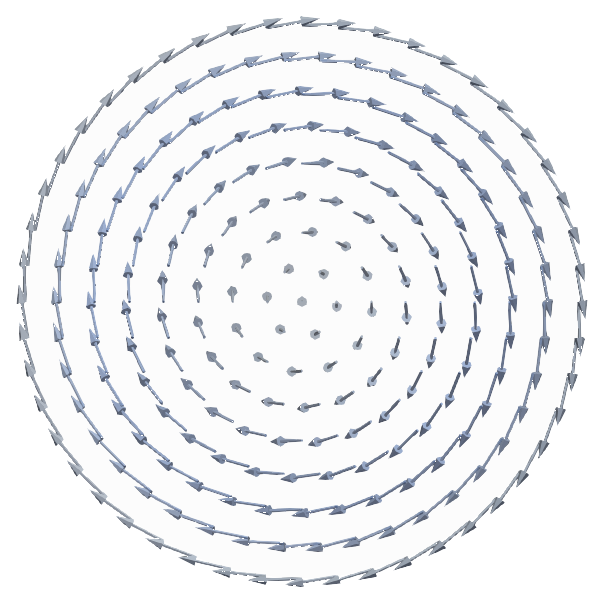}
		\caption{$t=2.5\times 10^{-6}$}
	\end{subfigure}
	\begin{subfigure}[b]{0.28\textwidth}
		\centering
		\includegraphics[width=\textwidth]{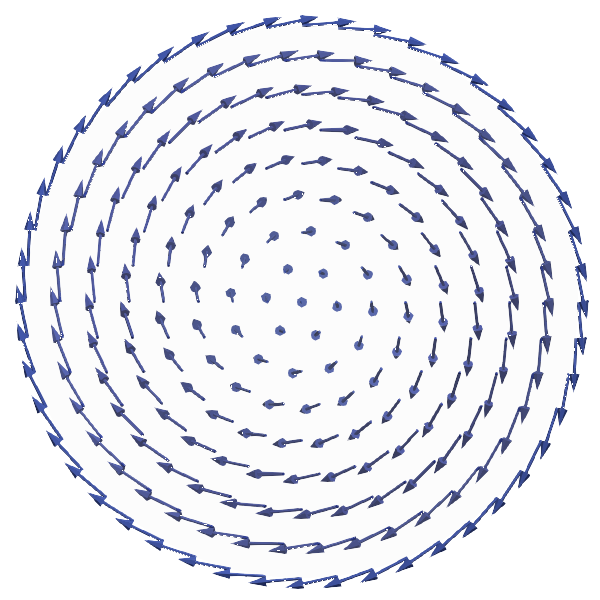}
		\caption{$t=5.0 \times 10^{-6}$}
	\end{subfigure}
	\begin{subfigure}[b]{0.1\textwidth}
		\centering
		\includegraphics[width=\textwidth]{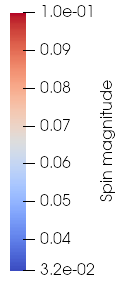}
	\end{subfigure}
	\caption{Snapshots of the spin accumulation vector field $\bff{s}$ (projected onto $\bb{R}^2$) for simulation 1.}
	\label{fig:snapshots spin 2d 1}
\end{figure}

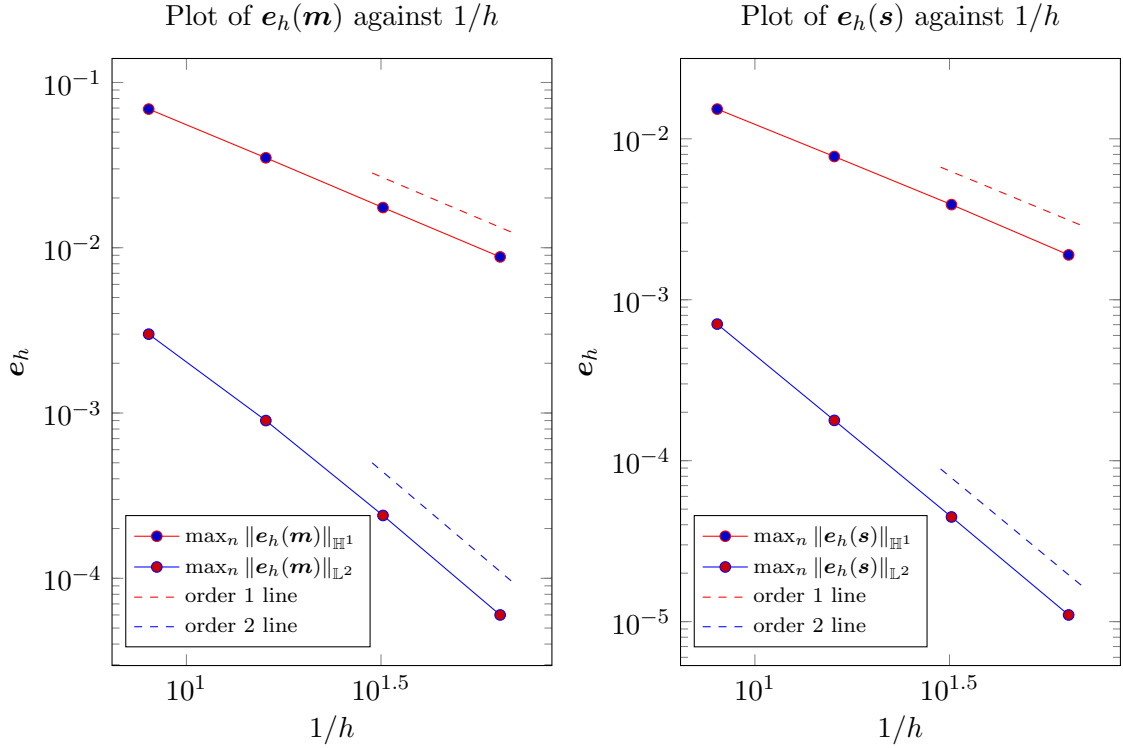
\begin{figure}[!htb]
	\begin{subfigure}[b]{0.45\textwidth}
		\centering
		\begin{tikzpicture}
			\begin{axis}[
				title=Plot of $\bff{e}_h(\bff{m})$ against $1/h$,
				height=1.3\textwidth,
				width=1\textwidth,
				xlabel= $1/h$,
				ylabel= $\bff{e}_h$,
				xmode=log,
				ymode=log,
				legend pos=south west,
				legend cell align=left,
				]
				\addplot+[mark=*,red] coordinates {(8,0.069)(16,0.035)(32,0.0175)(64,0.0088)};
				\addplot+[mark=*,blue] coordinates {(8,0.003)(16,0.0009)(32,0.00024)(64,0.00006)};
				\addplot+[dashed,no marks,red,domain=30:70]{0.85/x};
				\addplot+[dashed,no marks,blue,domain=30:70]{0.45/x^2};
				\legend{\scriptsize{$\max_n \norm{\bff{e}_h(\bff{m})}{\bb{H}^1}$}, \scriptsize{$\max_n \norm{\bff{e}_h(\bff{m})}{\bb{L}^2}$}, \scriptsize{order 1 line}, \scriptsize{order 2 line}}
			\end{axis}
		\end{tikzpicture}
		\caption{Error order of $\bff{m}$ for simulation 1.}
		\label{fig:order u 1}
	\end{subfigure}
	\begin{subfigure}[b]{0.45\textwidth}
		\centering
		\begin{tikzpicture}
			\begin{axis}[
				title=Plot of $\bff{e}_h(\bff{s})$ against $1/h$,
				height=1.3\textwidth,
				width=1\textwidth,
				xlabel= $1/h$,
				ylabel= $\bff{e}_h$,
				xmode=log,
				ymode=log,
				legend pos=south west,
				legend cell align=left,
				]
				\addplot+[mark=*,red] coordinates {(8,0.0153)(16,0.00776)(32,0.0039)(64,0.0019)};
				\addplot+[mark=*,blue] coordinates {(8,0.000706)(16,0.000178)(32,0.0000447)(64,0.000011)};
				\addplot+[dashed,no marks,red,domain=30:70]{0.2/x};
				\addplot+[dashed,no marks,blue,domain=30:70]{0.08/x^2};
				\legend{\scriptsize{$\max_n \norm{\bff{e}_h(\bff{s})}{\bb{H}^1}$}, \scriptsize{$\max_n \norm{\bff{e}_h(\bff{s})}{\bb{L}^2}$}, \scriptsize{order 1 line}, \scriptsize{order 2 line}}
			\end{axis}
		\end{tikzpicture}
		\caption{Error order of $\bff{s}$ for simulation 1.}
		\label{fig:order H 1}
	\end{subfigure}
	\caption{Order of convergence of $\bff{m}$ and $\bff{s}$ for simulation 1.}
\end{figure}

\begin{figure}[!htb]
	\centering
	\begin{subfigure}[b]{0.28\textwidth}
		\centering
		\includegraphics[width=\textwidth]{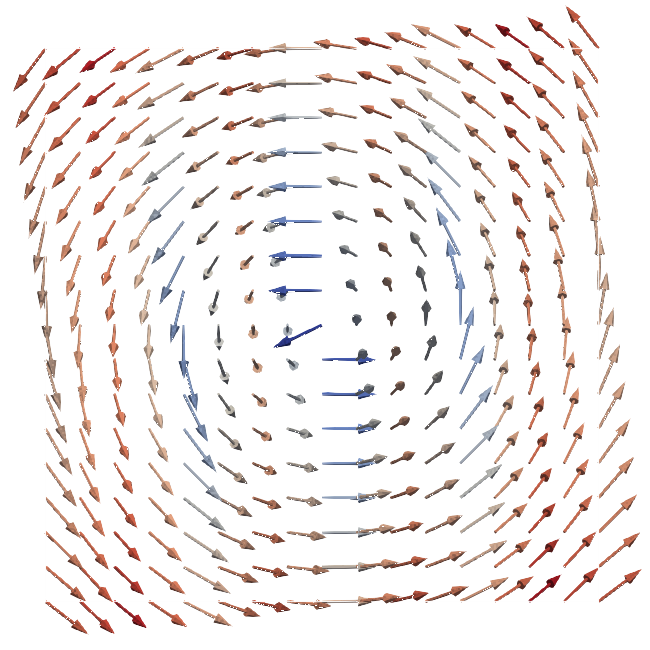}
		\caption{$t=0$}
	\end{subfigure}
	\begin{subfigure}[b]{0.28\textwidth}
		\centering
		\includegraphics[width=\textwidth]{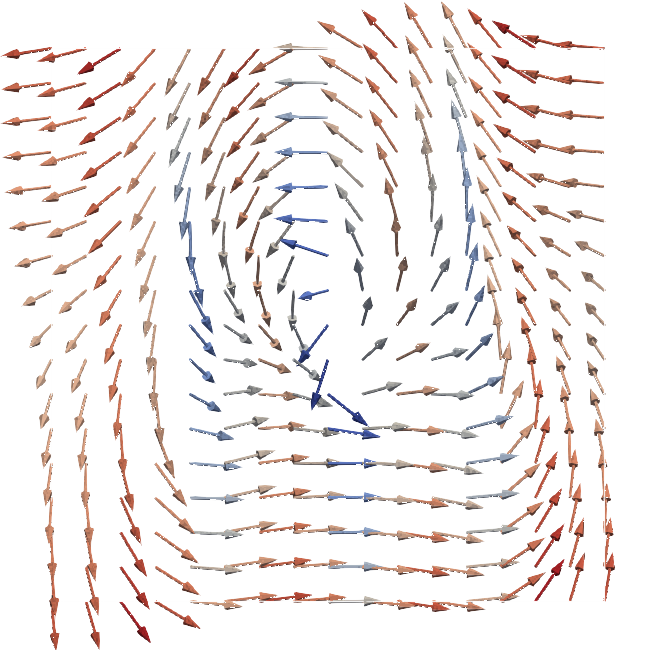}
		\caption{$t=2.5\times 10^{-7}$}
	\end{subfigure}
	\begin{subfigure}[b]{0.28\textwidth}
		\centering
		\includegraphics[width=\textwidth]{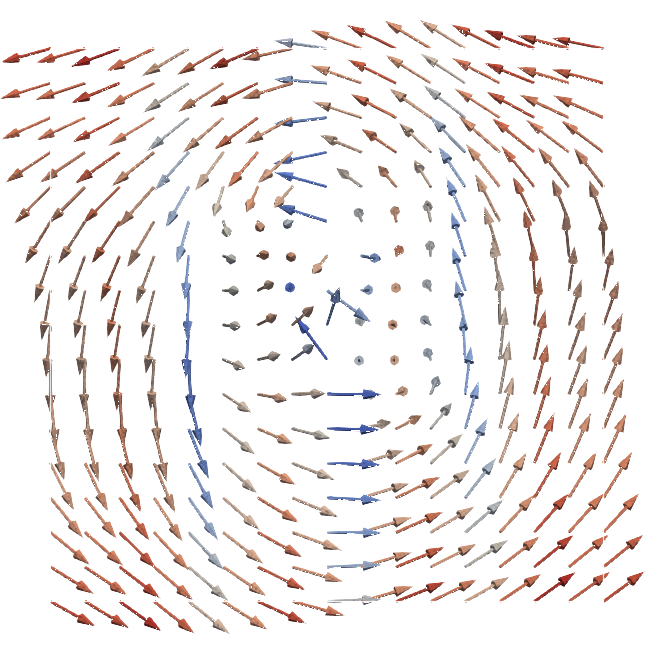}
		\caption{$t=1.0 \times 10^{-6}$}
	\end{subfigure}
	\begin{subfigure}[b]{0.1\textwidth}
		\centering
		\includegraphics[width=\textwidth]{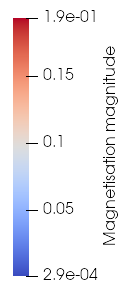}
	\end{subfigure}
	\begin{subfigure}[b]{0.28\textwidth}
		\centering
		\includegraphics[width=\textwidth]{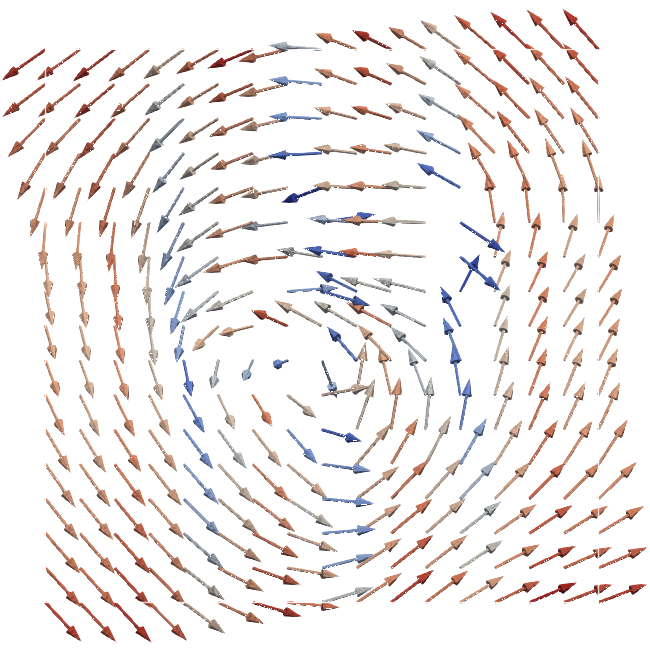}
		\caption{$t=2.0\times 10^{-6}$}
	\end{subfigure}
	\begin{subfigure}[b]{0.28\textwidth}
		\centering
		\includegraphics[width=\textwidth]{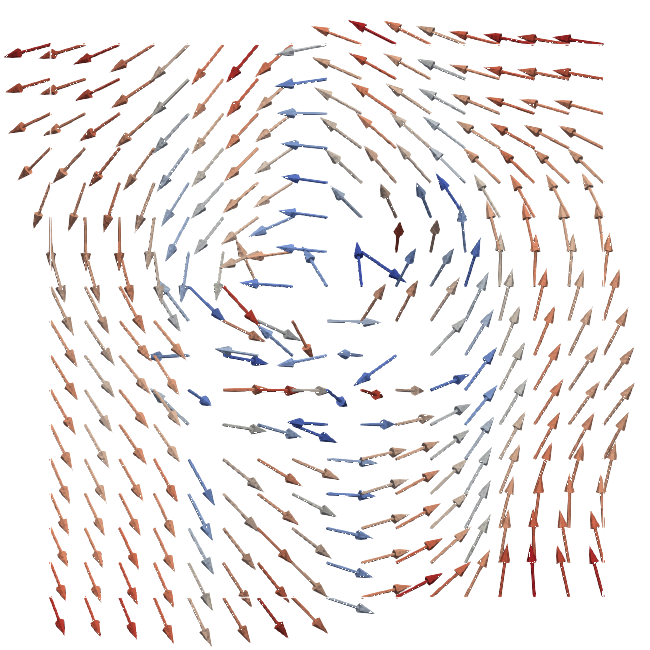}
		\caption{$t=4.0\times 10^{-6}$}
	\end{subfigure}
	\begin{subfigure}[b]{0.28\textwidth}
		\centering
		\includegraphics[width=\textwidth]{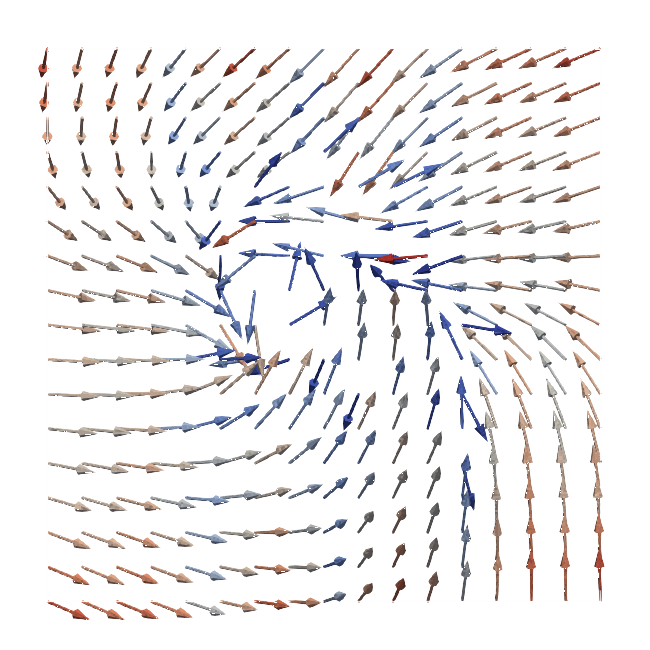}
		\caption{$t=5.0\times 10^{-6}$}
	\end{subfigure}
		\begin{subfigure}[b]{0.1\textwidth}
		\centering
		\includegraphics[width=\textwidth]{m1_legend.png}
	\end{subfigure}
	\caption{Snapshots of the magnetisation vector field $\bff{m}$ (projected onto $\bb{R}^2$) for simulation 2.}
	\label{fig:snapshots field 2d 2}
\end{figure}

\begin{figure}[!htb]
	\centering
	\begin{subfigure}[b]{0.28\textwidth}
		\centering
		\includegraphics[width=\textwidth]{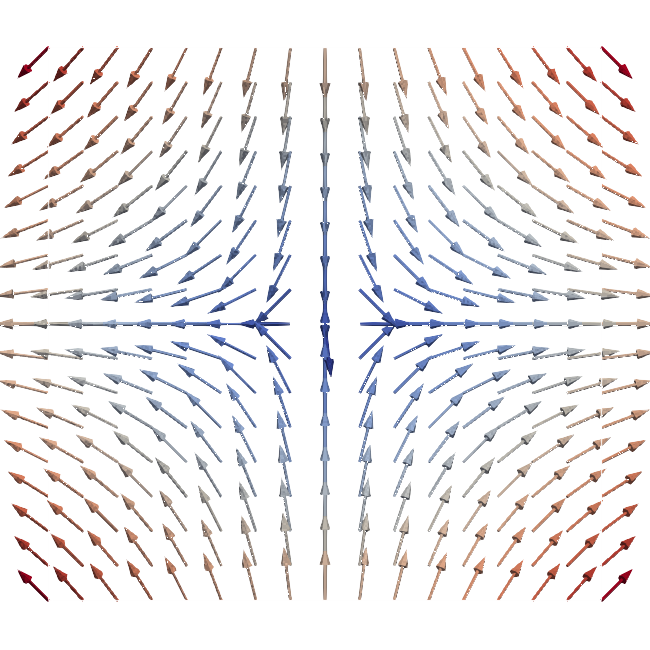}
		\caption{$t=0$}
	\end{subfigure}
	\begin{subfigure}[b]{0.28\textwidth}
		\centering
		\includegraphics[width=\textwidth]{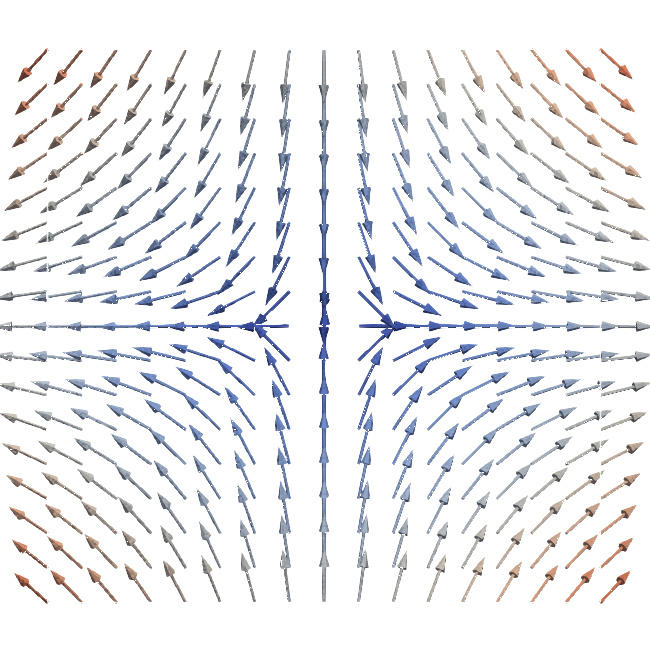}
		\caption{$t=2.5\times 10^{-6}$}
	\end{subfigure}
	\begin{subfigure}[b]{0.28\textwidth}
		\centering
		\includegraphics[width=\textwidth]{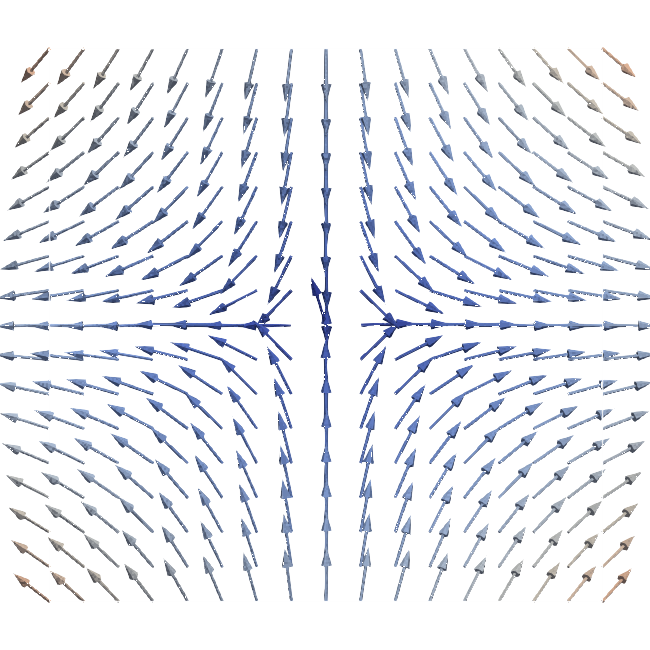}
		\caption{$t=5.0 \times 10^{-6}$}
	\end{subfigure}
	\begin{subfigure}[b]{0.1\textwidth}
		\centering
		\includegraphics[width=\textwidth]{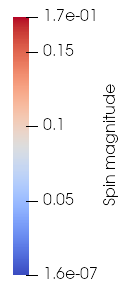}
	\end{subfigure}
	\caption{Snapshots of the spin accumulation vector field $\bff{s}$ (projected onto $\bb{R}^2$) for simulation 2.}
	\label{fig:snapshots spin 2d 2}
\end{figure}

\begin{figure}[!htb]
	\begin{subfigure}[b]{0.45\textwidth}
		\centering
		\begin{tikzpicture}
			\begin{axis}[
				title=Plot of $\bff{e}_h(\bff{m})$ against $1/h$,
				height=1.3\textwidth,
				width=1\textwidth,
				xlabel= $1/h$,
				ylabel= $\bff{e}_h$,
				xmode=log,
				ymode=log,
				legend pos=south west,
				legend cell align=left,
				]
				\addplot+[mark=*,red] coordinates {(8,0.38)(16,0.2)(32,0.11)(64,0.057)(128,0.03)};
				\addplot+[mark=*,blue] coordinates {(8,0.0254)(16,0.0066)(32,0.0017)(64,0.00042)(128,0.00012)};
				\addplot+[dashed,no marks,red,domain=35:135]{6.7/x};
				\addplot+[dashed,no marks,blue,domain=35:135]{3.8/x^2};
				\legend{\scriptsize{$\max_n \norm{\bff{e}_h(\bff{m})}{\bb{H}^1}$}, \scriptsize{$\max_n \norm{\bff{e}_h(\bff{m})}{\bb{L}^2}$}, \scriptsize{order 1 line}, \scriptsize{order 2 line}}
			\end{axis}
		\end{tikzpicture}
		\caption{Error order of $\bff{u}$ for simulation 2.}
		\label{fig:order u 2}
	\end{subfigure}
	\begin{subfigure}[b]{0.45\textwidth}
		\centering
		\begin{tikzpicture}
			\begin{axis}[
				title=Plot of $\bff{e}_h(\bff{s})$ against $1/h$,
				height=1.3\textwidth,
				width=1\textwidth,
				xlabel= $1/h$,
				ylabel= $\bff{e}_h$,
				xmode=log,
				ymode=log,
				legend pos=south west,
				legend cell align=left,
				]
				\addplot+[mark=*,red] coordinates {(8,0.029)(16,0.014)(32,0.0073)(64,0.0037)(128,0.0019)};
				\addplot+[mark=*,blue] coordinates {(8,0.00126)(16,0.000314)(32,0.0000786)(64,0.0000197)(128,0.000005)};
				\addplot+[dashed,no marks,red,domain=35:135]{0.4/x};
				\addplot+[dashed,no marks,blue,domain=35:135]{0.18/x^2};
				\legend{\scriptsize{$\max_n \norm{\bff{e}_h(\bff{s})}{\bb{H}^1}$}, \scriptsize{$\max_n \norm{\bff{e}_h(\bff{s})}{\bb{L}^2}$}, \scriptsize{order 1 line}, \scriptsize{order 2 line}}
			\end{axis}
		\end{tikzpicture}
		\caption{Error order of $\bff{s}$ for simulation 2.}
		\label{fig:order H 2}
	\end{subfigure}
	\caption{Order of convergence of $\bff{m}$ and $\bff{s}$ for simulation 2.}
\end{figure}
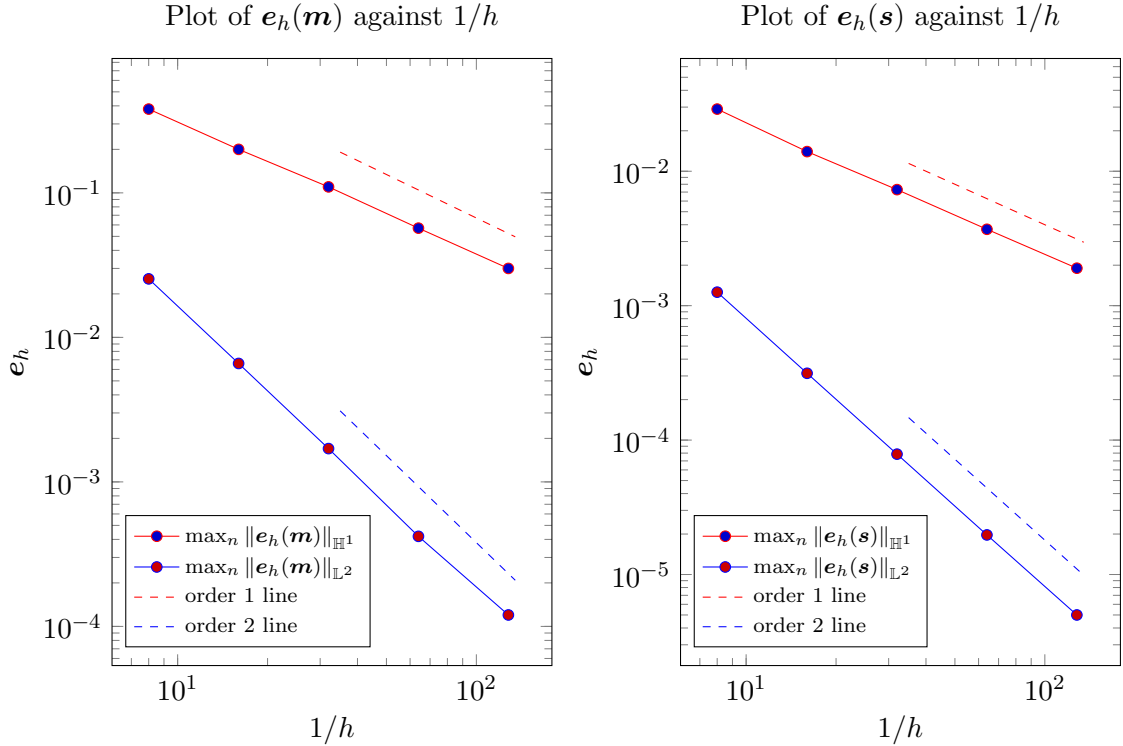

\begin{figure}[!htb]
	\centering
	\begin{subfigure}[b]{0.28\textwidth}
		\centering
		\includegraphics[width=\textwidth]{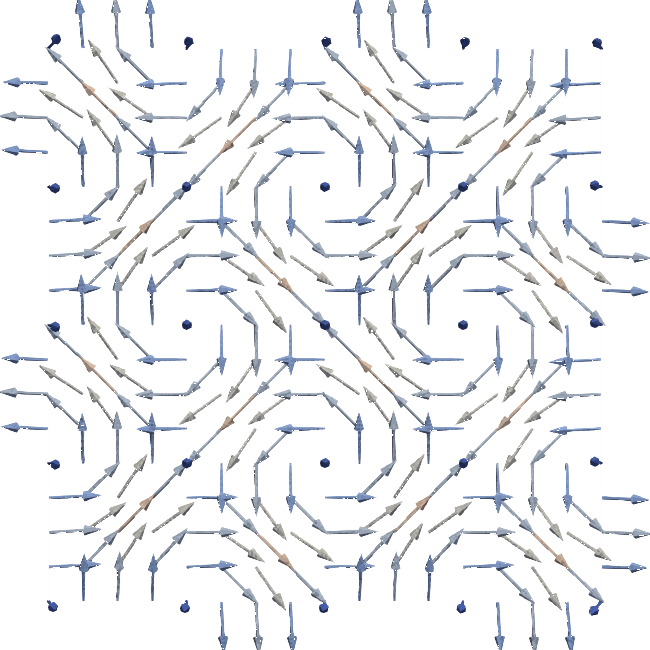}
		\caption{$t=0$}
	\end{subfigure}
	\begin{subfigure}[b]{0.28\textwidth}
		\centering
		\includegraphics[width=\textwidth]{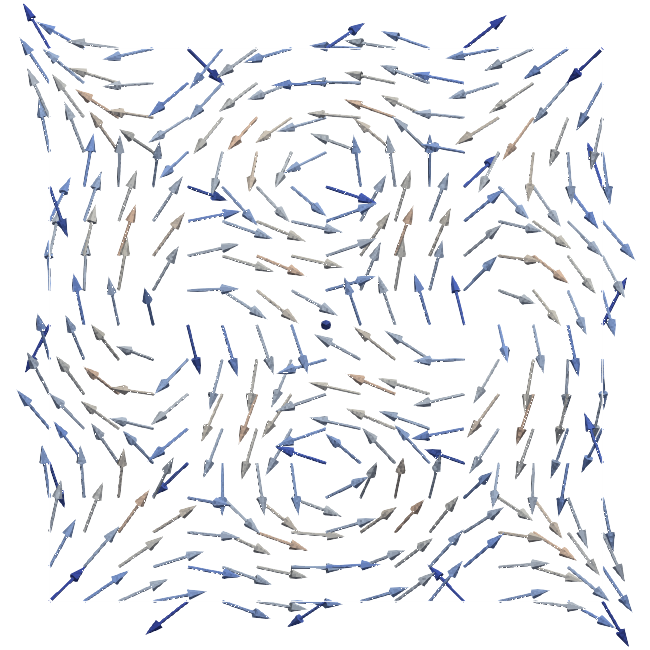}
		\caption{$t=1.0\times 10^{-7}$}
	\end{subfigure}
	\begin{subfigure}[b]{0.28\textwidth}
		\centering
		\includegraphics[width=\textwidth]{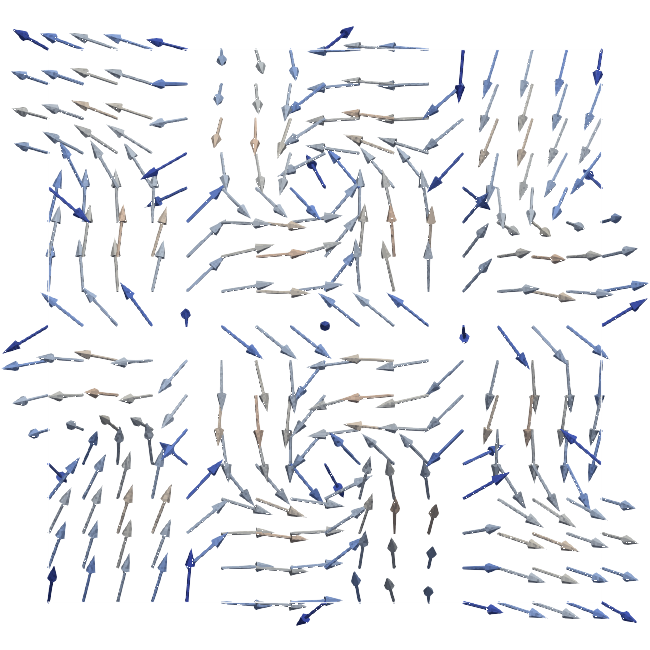}
		\caption{$t=2.5 \times 10^{-7}$}
	\end{subfigure}
	\begin{subfigure}[b]{0.1\textwidth}
		\centering
		\includegraphics[width=\textwidth]{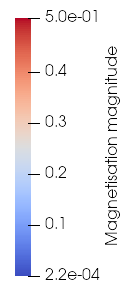}
	\end{subfigure}
	\begin{subfigure}[b]{0.28\textwidth}
		\centering
		\includegraphics[width=\textwidth]{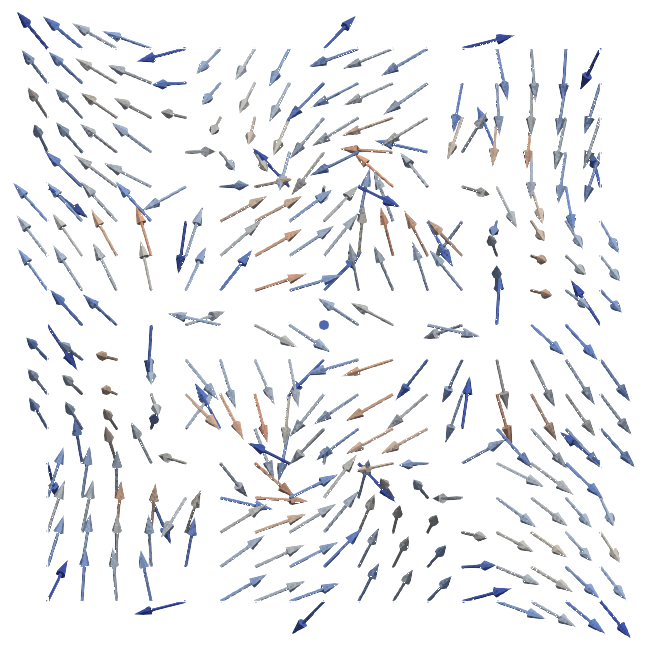}
		\caption{$t=5.0\times 10^{-7}$}
	\end{subfigure}
	\begin{subfigure}[b]{0.28\textwidth}
		\centering
		\includegraphics[width=\textwidth]{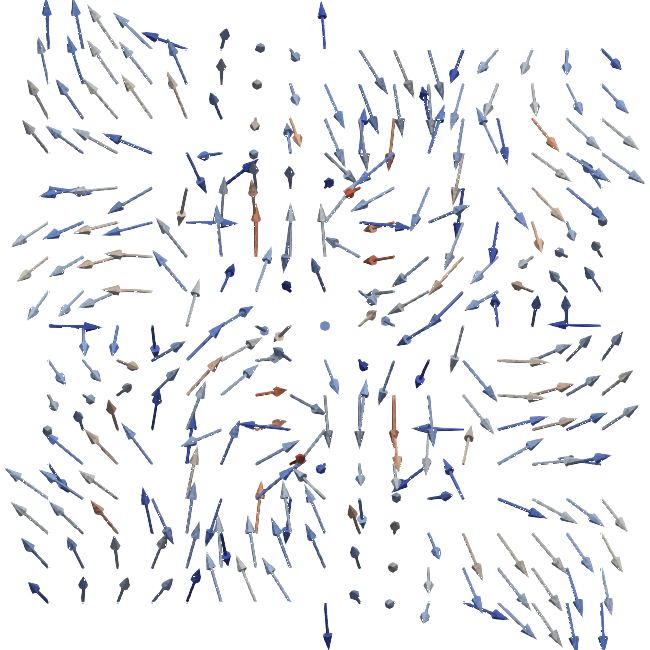}
		\caption{$t=1.0\times 10^{-6}$}
	\end{subfigure}
	\begin{subfigure}[b]{0.28\textwidth}
		\centering
		\includegraphics[width=\textwidth]{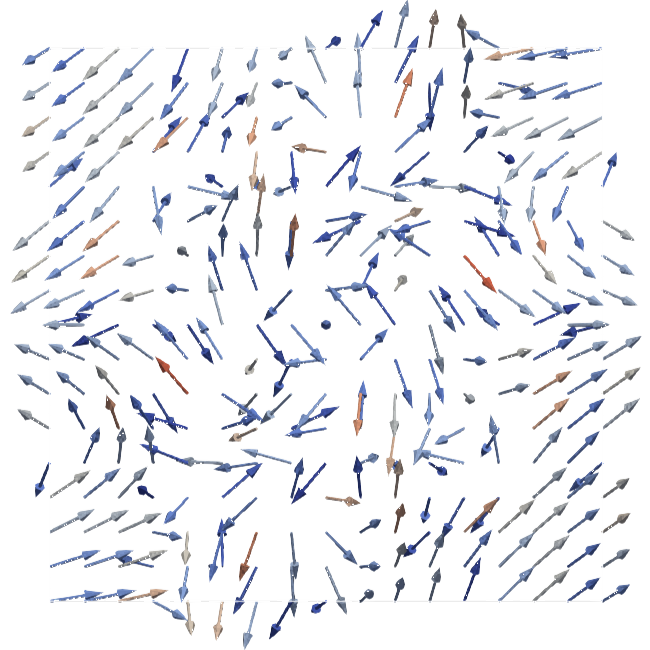}
		\caption{$t=2.5\times 10^{-6}$}
	\end{subfigure}
	\begin{subfigure}[b]{0.1\textwidth}
		\centering
		\includegraphics[width=\textwidth]{mb_legend.png}
	\end{subfigure}
	\caption{Snapshots of the magnetisation vector field $\bff{m}$ (projected onto $\bb{R}^2$) for simulation 3.}
	\label{fig:snapshots field 2d 3}
\end{figure}

\begin{figure}[!htb]
	\centering
	\begin{subfigure}[b]{0.28\textwidth}
		\centering
		\includegraphics[width=\textwidth]{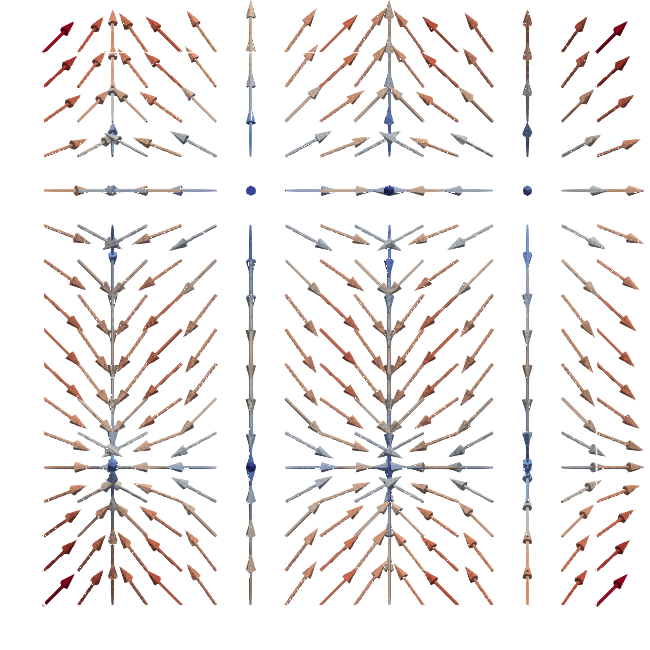}
		\caption{$t=0$}
	\end{subfigure}
	\begin{subfigure}[b]{0.28\textwidth}
		\centering
		\includegraphics[width=\textwidth]{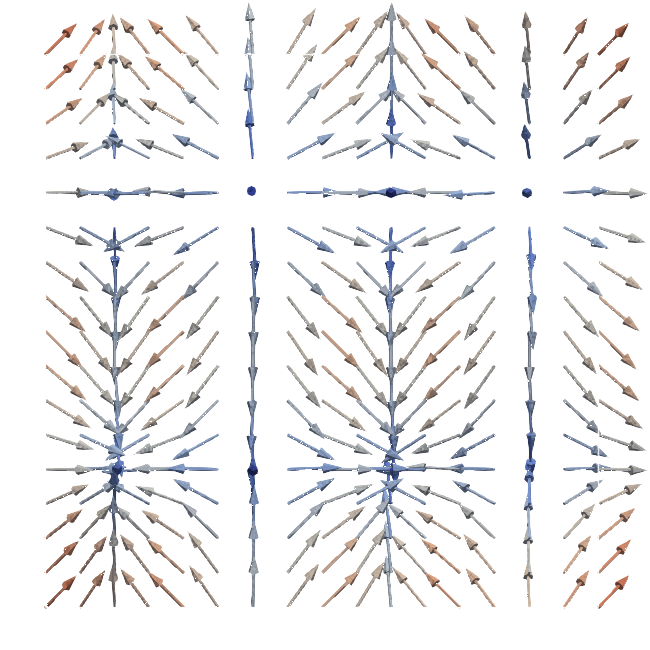}
		\caption{$t=2.5\times 10^{-6}$}
	\end{subfigure}
	\begin{subfigure}[b]{0.28\textwidth}
		\centering
		\includegraphics[width=\textwidth]{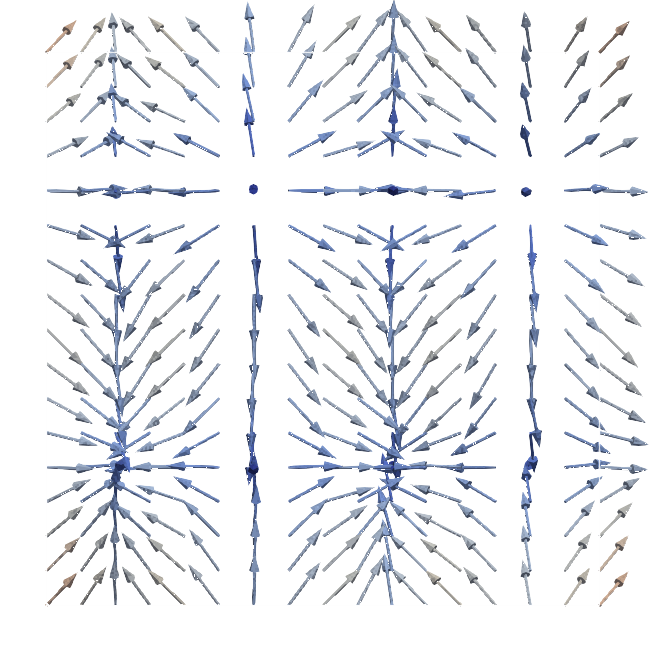}
		\caption{$t=5.0 \times 10^{-6}$}
	\end{subfigure}
	\begin{subfigure}[b]{0.1\textwidth}
		\centering
		\includegraphics[width=\textwidth]{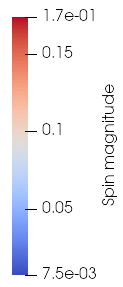}
	\end{subfigure}
	\caption{Snapshots of the spin accumulation vector field $\bff{s}$ (projected onto $\bb{R}^2$) for simulation 3.}
	\label{fig:snapshots spin 2d 3}
\end{figure}

\begin{figure}[!htb]
	\begin{subfigure}[b]{0.45\textwidth}
		\centering
		\begin{tikzpicture}
			\begin{axis}[
				title=Plot of $\bff{e}_h(\bff{m})$ against $1/h$,
				height=1.3\textwidth,
				width=1\textwidth,
				xlabel= $1/h$,
				ylabel= $\bff{e}_h$,
				xmode=log,
				ymode=log,
				legend pos=south west,
				legend cell align=left,
				]
				\addplot+[mark=*,red] coordinates {(8,1.1)(16,0.57)(32,0.29)(64,0.15)};
				\addplot+[mark=*,blue] coordinates {(8,0.072)(16,0.019)(32,0.0048)(64,0.0012)};
				\addplot+[dashed,no marks,red,domain=30:65]{13.5/x};
				\addplot+[dashed,no marks,blue,domain=30:65]{9/x^2};
				\legend{\scriptsize{$\max_n \norm{\bff{e}_h(\bff{m})}{\bb{H}^1}$}, \scriptsize{$\max_n \norm{\bff{e}_h(\bff{m})}{\bb{L}^2}$}, \scriptsize{order 1 line}, \scriptsize{order 2 line}}
			\end{axis}
		\end{tikzpicture}
		\caption{Error order of $\bff{m}$ for simulation 3.}
		\label{fig:order u 3}
	\end{subfigure}
	\begin{subfigure}[b]{0.45\textwidth}
		\centering
		\begin{tikzpicture}
			\begin{axis}[
				title=Plot of $\bff{e}_h(\bff{s})$ against $1/h$,
				height=1.3\textwidth,
				width=1\textwidth,
				xlabel= $1/h$,
				ylabel= $\bff{e}_h$,
				xmode=log,
				ymode=log,
				legend pos=south west,
				legend cell align=left,
				]
				\addplot+[mark=*,red] coordinates {(8,0.4)(16,0.21)(32,0.105)(64,0.053)};
				\addplot+[mark=*,blue] coordinates {(8,0.026)(16,0.0068)(32,0.00172)(64,0.00043)};
				\addplot+[dashed,no marks,red,domain=30:65]{5/x};
				\addplot+[dashed,no marks,blue,domain=30:65]{3/x^2};
				\legend{\scriptsize{$\max_n \norm{\bff{e}_h(\bff{s})}{\bb{H}^1}$}, \scriptsize{$\max_n \norm{\bff{e}_h(\bff{s})}{\bb{L}^2}$}, \scriptsize{order 1 line}, \scriptsize{order 2 line}}
			\end{axis}
		\end{tikzpicture}
		\caption{Error order of $\bff{s}$ for simulation 3.}
		\label{fig:order H 3}
	\end{subfigure}
	\caption{Order of convergence of $\bff{m}$ and $\bff{s}$ for simulation 3.}
\end{figure}
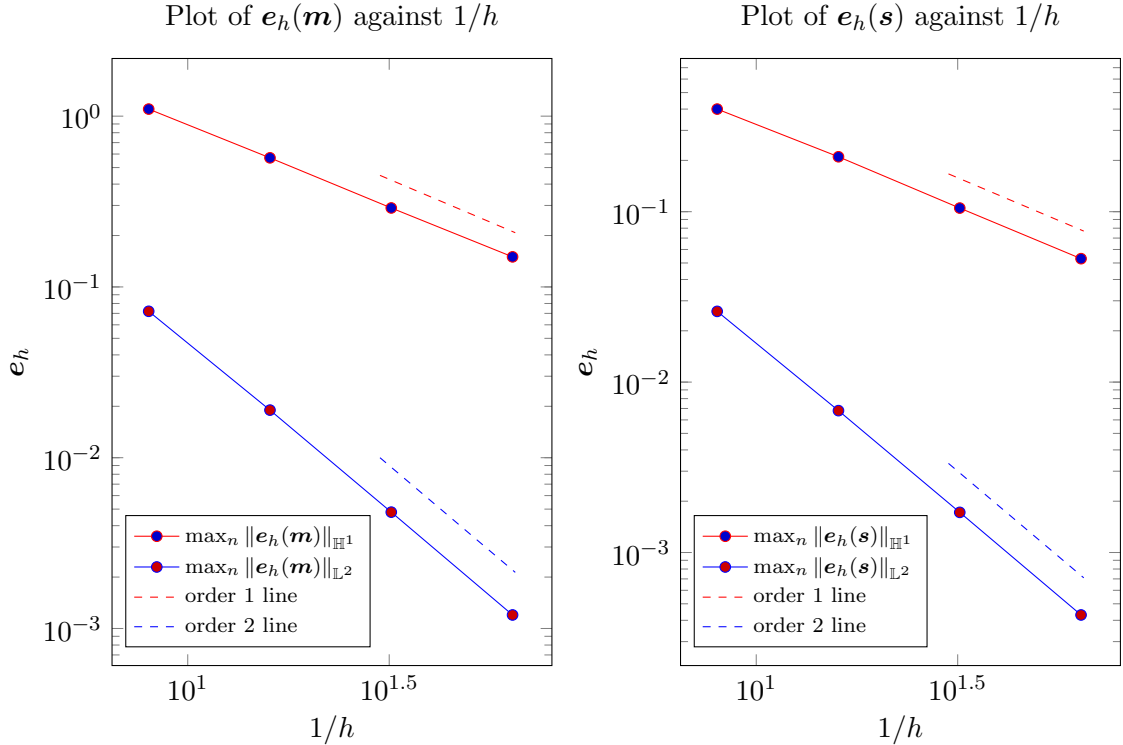


\newcommand{\noopsort}[1]{}\def\cprime{$'$}
\def\soft#1{\leavevmode\setbox0=\hbox{h}\dimen7=\ht0\advance \dimen7
	by-1ex\relax\if t#1\relax\rlap{\raise.6\dimen7
		\hbox{\kern.3ex\char'47}}#1\relax\else\if T#1\relax
	\rlap{\raise.5\dimen7\hbox{\kern1.3ex\char'47}}#1\relax \else\if
	d#1\relax\rlap{\raise.5\dimen7\hbox{\kern.9ex \char'47}}#1\relax\else\if
	D#1\relax\rlap{\raise.5\dimen7 \hbox{\kern1.4ex\char'47}}#1\relax\else\if
	l#1\relax \rlap{\raise.5\dimen7\hbox{\kern.4ex\char'47}}#1\relax \else\if
	L#1\relax\rlap{\raise.5\dimen7\hbox{\kern.7ex
			\char'47}}#1\relax\else\message{accent \string\soft \space #1 not
		defined!}#1\relax\fi\fi\fi\fi\fi\fi}

\end{document}